\documentclass[11pt,a4paper,reqno]{amsart}
\usepackage[english]{babel}
\pdfoutput=1
\usepackage[applemac]{inputenc}
\usepackage[T1]{fontenc}
\usepackage{palatino}
\usepackage{amsmath}
\usepackage{amssymb}
\usepackage{amsthm}
\usepackage{enumitem}
\usepackage{amsfonts}
\usepackage{esint}
\usepackage{graphicx}
\usepackage{xcolor}
\usepackage{mathtools}
\usepackage{comment}
\mathtoolsset{showonlyrefs}

\usepackage[colorlinks = true, citecolor = black]{hyperref}
\title[Geometric lemmas and uniform rectifiability: Part 2]{On various Carleson-type geometric lemmas\\ and uniform rectifiability in metric spaces: Part 2}

\author[K. F\"assler]{Katrin F\"assler}
\author[I. Y. Violo]{Ivan Yuri Violo}

\address{Department of Mathematics and Statistics\\ University of Jyv\"askyl\"a \\ P.O. Box 35 (MaD),
FI-40014 University of Jyv\"askyl\"a, Finland}

\email{katrin.s.fassler@jyu.fi}
 \email{ivanyuri.violo@dm.unipi.it {\it (current affiliation: Universit\`{a} di Pisa, Dipartimento di Matematica, Largo Bruno Pontecorvo 5, 56127 Pisa, Italy)}}

\thanks{The authors were supported by the Academy of Finland (now: Research Council of Finland)
grants no.\ 321696 (K.F.) and 328846, 321896 (I.Y.V.). The second author is supported by the
European Union (ERC, ConFine, 101078057). }
\date{\today}

\newcommand{\R}{\mathbb{R}}
\newcommand{\N}{\mathbb{N}}

\renewcommand{\H}{\mathbb{H}}
\newcommand{\V}{\mathbb{V}}
\newcommand{\W}{\mathbb{W}}

\newcommand{\calV}{\mathcal{V}}

\newcommand{\cH}{\mathcal{H}}

\newcommand{\diam}{\operatorname{diam}}
\newcommand{\card}{\operatorname{card}}

\newcommand{\ep}{\varepsilon}

\newcommand{\eps}{\varepsilon}

\renewcommand{\d}{{\rm d}}
\newcommand{\sfd}{{\sf d}}
\newcommand{\X}{{\rm X}}
\newcommand{\rr}{\mathbb R}

\def\Barint_#1{\mathchoice
          {\mathop{\vrule width 6pt height 3 pt depth -2.5pt
                  \kern -8pt \intop}\nolimits_{#1}}%
          {\mathop{\vrule width 5pt height 3 pt depth -2.6pt
                  \kern -6pt \intop}\nolimits_{#1}}%
          {\mathop{\vrule width 5pt height 3 pt depth -2.6pt
                  \kern -6pt \intop}\nolimits_{#1}}%
          {\mathop{\vrule width 5pt height 3 pt depth -2.6pt
                  \kern -6pt \intop}\nolimits_{#1}}}

\numberwithin{equation}{section}

\theoremstyle{plain}
\newtheorem{thm}[equation]{Theorem}
\newtheorem{thmdef}[equation]{Theorem and Definition}

\newtheorem{lemma}[equation]{Lemma}

\newtheorem{cor}[equation]{Corollary}

\newtheorem{proposition}[equation]{Proposition}

\theoremstyle{definition}

\newtheorem{definition}[equation]{Definition}

\theoremstyle{remark}
\newtheorem{remark}[equation]{Remark}

\newcommand{\reg}{\mathrm{Reg}}
\newcommand{\nb}{\iota}

\usepackage{hyperref}
\hypersetup{%
  colorlinks = true,
  linkcolor  = black
}

\allowdisplaybreaks

\begin{document}

\begin{abstract}{
We characterize  uniform $k$-rectifiability in Euclidean spaces in terms of 
a Carleson-type geometric lemma for
 a new notion of  flatness coefficients, which we call $\nb$-numbers. The characterization follows from an abstract statement about approximation by
generalized planes in metric spaces, which also applies to the
study of low-dimensional sets in Heisenberg groups.
A key aspect is that the $\nb$-coefficients are in general \emph{not} pointwise comparable to the usual squared $\beta$-numbers for dyadic cubes on $k$-regular sets in $\mathbb{R}^n$, however our result implies that they are still equivalent 
in terms of a Carleson-type geometric lemma.} 
\end{abstract}

\maketitle

\tableofcontents

\section{Introduction}\label{s:intro}  {This note is the second part of a series of two papers
concerned with  new quantitative coefficients, which we call
\emph{$\nb$-numbers}, and their relation with notions of uniform
rectifiability in Euclidean and abstract metric spaces. 
We refer
to the first part \cite{CarlesonPart1} for a detailed introduction
to the topic, and focus here on describing the concepts relevant
for the present paper. 

\textcolor{black}{A set $E$ in $\mathbb{R}^n$ is \emph{uniformly $k$-rectifiable}, for $k\in \mathbb{N}$, if it is $k$-regular and has \emph{big pieces of Lipschitz images from $\mathbb{R}^k$}. The study of such sets was pioneered by David and Semmes \cite{MR1113517,MR1251061}, who found many equivalent conditions. 
Using a suitable variant of $\nb$-numbers,
for $k\in \mathbb{N}$, we give a new characterization of
 uniform $k$-rectifiability }in
Euclidean spaces (Theorem \ref{t:Char_iota_Eucl}). The proof
passes through an abstract axiomatic result (Theorem \ref{thm:axiomaticIntro}), which we believe to be of
independent interest and which applies also to non-Euclidean
Heisenberg groups (Theorem \ref{t:from beta_to_alpha_in_heis}).}

\subsection{From $\beta$-numbers to $\nb$-numbers in Euclidean spaces}
Uniformly $k$-{rectifiable} sets in
$\mathbb{R}^n$ ($k,n\in \mathbb{N}$, $1\leq k<n$) can be
characterized {as $k$-regular sets that are well
approximated by $k$-dimensional planes as quantified by means of a
``geometric lemma'' for Jones
\emph{$\beta_{q,\mathcal{V}_k}$-numbers} for $1\leq
q<\frac{2k}{k-2}$ if $k\geq 2$ and $1\leq q\leq \infty$ if $k=1$,
recall \cite[I,1.4]{MR1251061}. By ``$k$-regular'' we mean sets
that satisfy the Ahlfors $s$-regularity condition
\eqref{eq:regular} for $s=k$.} For the purpose of this
introduction, we say that a $k$-regular set $E$ in Euclidean space
$\mathbb{R}^n$ satisfies the \emph{$2$-geometric lemma with
respect to $\beta_{q,\mathcal{V}_k}$}, denoted $E\in
\mathrm{GLem}(\beta_{q,\mathcal{V}_k},2)$, if there is a constant
$M\geq 0$ such that
\begin{equation}\label{eq:Beta_q}
\int_{B_R({x_0})\cap E} \int_0^{R}
\beta_{q,\mathcal{V}_k}(B_r(x)\cap
E)^2\,\frac{dr}{r}d\mathcal{H}^k(x)\leq M R^k\,\quad
{x_0}\in E,\,0<R\leq \mathrm{diam}E,\,R<\infty,
\end{equation}
where the coefficients
\begin{equation}\label{eq:beta_intro}
\beta_{q,\mathcal{V}_k}(B_r(x)\cap E):= \inf_{V\in \mathcal{V}_k}
\left( \Barint_{B_r(x)\cap
E}\left[\frac{\sfd(y,V)}{\mathrm{diam}(B_r(x)\cap
E)}\right]^q\,d\mathcal{H}^k(y)\right)^{1/q}, \quad q\in
(0,\infty),
\end{equation}
quantify in a scale-invariant and $L^q$-based way how well the set
$E$ is approximated by $k$-planes $V\in \mathcal{V}_k$ at $x\in E$
and scale $r>0$ in the Euclidean distance.

{In \cite{CarlesonPart1} and in this paper, we
consider another family of quantitative coefficients that we call
\emph{$\nb$-numbers}.}  Roughly speaking, $\nb$-numbers measure
``flatness'' of a set using mappings into model spaces, rather
than using the metric distance from approximating sets.
{We consider first a Euclidean variant of the
$\nb$-coefficients.} We denote by $\pi_V:\mathbb{R}^n \to V$ the
Euclidean orthogonal projection onto the affine $k$-plane $V$ in
$\mathbb{R}^n$, and we define for $q\in (0,\infty)$,
\begin{equation}\label{eq:nb_intro}
\nb_{q,\mathcal{V}_k}(B_r(x)\cap E):= \inf_{V\in \mathcal{V}_k}
\left( \Barint_{B_r(x)\cap E}\Barint_{B_r(x)\cap
E}\left[\frac{\left||y-z|-|\pi_V(y)-\pi_V(z)|\right|}{\mathrm{diam}(B_r(x)\cap
E)}\right]^q\,d\mathcal{H}^k(y)d\mathcal{H}^k(z)\right)^{1/q}.
\end{equation}
By the triangle inequality we always have
\begin{equation}\label{eq:easy iota beta}
    \nb_{q,\mathcal{V}_k}(B_r(x)\cap E)\le 2 \beta_{q,\mathcal{V}_k}(B_r(x)\cap E).
\end{equation}
{\color{black} For illustration, suppose that the double integral in \eqref{eq:nb_intro} vanishes for some $V\in \mathcal V_k.$ Then $\pi_V$ is an isometry in $B_r(x)\cap E$ up to an $\mathcal{H}^k$-null set, which means that $E\cap B_r(x)$ is contained in a $k$-plane $V'$ parallel to $V$, except for an $\mathcal{H}^k$-null set. In other words, the coefficients  $\nb_{q,\mathcal{V}_k}$ provide an alternative to the  coefficients $\beta_{q,\mathcal{V}_k}$ for measuring how well $E$ is approximated by $k$-planes. The main motivation is that these new numbers are better suited to study settings more general than the Euclidean one.
In fact, we will define in Section \ref{s:axiomatic}  a variant of the coefficients $\nb_{q,\mathcal{V}_k}$ in  metric spaces with projections onto more abstract, axiomatically defined, planes. 
}

 The 
 coefficients {\color{black}$\nb_{q,\mathcal{V}_k}$} can be used to formulate a geometric lemma
analogous to \eqref{eq:Beta_q}; see Definition \ref{d:GL} for a
very general definition of geometric lemmas
{stated} in terms of systems of Christ-David
dyadic cubes. Roughly speaking, the symbol $\mathrm{GLem}(h,p,M)$
denotes a Carleson measure condition in the spirit of
\eqref{eq:Beta_q} with $\beta$-numbers replaced by other
coefficients given by $h$, and the integrability exponent ``$2$''
replaced by ``$p$''. \textcolor{black}{We write $E\in \mathrm{GLem}(h,p)$ to express that $E$ satisfies the condition $\mathrm{GLem}(h,p,M)$ for some choice of $M<\infty$.}

{\color{black}
Despite \eqref{eq:easy iota beta}, it turns out that the $\iota$- and $\beta$-coefficients differ regarding the exponents for which a geometric lemma holds for a given set $E$.
The aim of this paper is precisely to provide rigorous versions of this fact. More precisely, 
 our goal  is twofold: first we show  that uniform $k$-rectifiability in $\rr^n$, which is known to be equivalent to $\mathrm{GLem}(\beta_{2,\mathcal{V}_k},2)$, can be characterized using the geometric lemma $\mathrm{GLem}(\nb_{1,\mathcal{V}_k},1)$. Second, we show that a geometric lemma of the form $\mathrm{GLem}(\beta,2q)$ implies $\mathrm{GLem}(\iota,q)$ in a rather abstract setting of metric spaces that admit a suitable system of approximating sets playing the role of ``$k$-planes".    
}

Combining Euclidean geometry and a special case of \textcolor{black}{the} more general
axiomatic statement that we derive in Theorem \ref{thm:axiomatic}
{(see Theorem \ref{thm:axiomaticIntro})}, we
obtain the following characterization:
\begin{thm}\label{t:Char_iota_Eucl}
A $k$-regular set $E\subset\mathbb{R}^n$ is uniformly
$k$-rectifiable if and only if $E\in
\mathrm{GLem}(\nb_{1,\mathcal{V}_k},1)$. 
\end{thm}
{The proof reveals that the constants involved in the two conditions can also be controlled quantitatively in terms of each other independently of $E$.}
To be more precise, we will prove directly that $E\in
\mathrm{GLem}(\nb_{1,\mathcal{V}_k},1)$ is equivalent to $E\in
\mathrm{GLem}(\beta_{2,\mathcal{V}_k},2)$ {in a quantitative way}. Note that for $\nb$ we consider the geometric lemma for $p=1$, while for $\beta$ the usual $p=2.$
This result is non-trivial because a `pointwise' version of this equivalence  cannot hold \textcolor{black}{in general}, i.e., it is \emph{not} generally true   that
\begin{equation}\label{eq:pointwise}
    \nb_{1,\mathcal{V}_k}(B_r(x)\cap E)\le C\beta_{2,\mathcal{V}_k}(B_r(x)\cap E)^2, \quad x\in E,\quad r\in (0,\diam E),
\end{equation}
with a constant $C$ independent of $x$ and $r$, {and $E$} \textcolor{black}{(see however Proposition \ref{prop:pointwise}).}
Nevertheless Theorem \ref{t:Char_iota_Eucl} still holds true. For
{a family of examples showing that}
 \eqref{eq:pointwise} fails {(for $k=1$, with a uniform constant)}, take any $\eps\ll 1$
and $r>0$ and consider $E\subset \rr^2$ to be the union of the horizontal axis $l_0$ and a parallel line $l$ at distance $\eps r$.  Then,
for any $q\in [1,\infty)$, it holds
$\beta_{{\color{black}2q},\mathcal{V}_1}(B_r(x)\cap E) \sim \eps$ and
$\nb_{q,\mathcal{V}_1}(B_r(x)\cap E) \gtrsim \log(\eps^{-1})^\frac1q \eps^2$,
for every $x\in E,$ {see Proposition
\ref{p:failure} for the details and a picture.}

 \begin{remark}\label{r:CharDiffIota}
{Theorem \ref{t:Char_iota_Eucl} continues to hold
if in formula \eqref{eq:nb_intro} for
$\nb_{q,\mathcal{V}_k}(B_r(x)\cap E)$ we replace the projections
$\pi_V:\mathbb{R}^n\to V$ onto $k$-planes by \emph{arbitrary}
(Borel) maps $f:B_r(x)\cap E \to \mathbb{R}^k$ (endowed with the Euclidean norm) and take the
infimum over all such maps.  The `only if' part of Theorem
\ref{t:Char_iota_Eucl} clearly still holds; for the `if' part see
Proposition \ref{prop:converse inequalities}. }
\end{remark}

{The definition of $\nb_{1,\mathcal{V}_k}$ does
not make sense in general metric spaces, as it refers to
orthogonal projections onto planes. Remark \ref{r:CharDiffIota}
motivated the definition of $\nb$-numbers for subsets of metric
spaces that we gave in \cite{CarlesonPart1}. Namely,} for $k\in
\mathbb{N}$ and a $k$-regular set $E$ in a metric space
$(\X,\sfd)$, and for $q\in (0,\infty)$, we defined
$\nb_{q,k}(B_r(x)\cap E)$ as the number
\begin{equation}\label{eq:nb_intro_metric}
\inf_{\|\cdot\|}\inf_{f:B_r(x)\cap E\to \mathbb{R}^k} \left(
\Barint_{B_r(x)\cap E}\Barint_{B_r(x)\cap
E}\left[\frac{|\sfd(y,z)-\|f(y)-f(z)\||}{\mathrm{diam}(B_r(x)\cap
E)}\right]^q\,d\mathcal{H}^k(y)d\mathcal{H}^k(z)\right)^{1/q}.
\end{equation}
Here the first infimum is taken over all norms on $\mathbb{R}^k$,
and the functions $f$ in the second infimum are assumed to be
\emph{Borel}. {We also defined  the number
$\nb_{q,k,{\sf Eucl}}(B_r(x)\cap E)$ by considering in the first
infimum in \eqref{eq:nb_intro_metric} only the \emph{Euclidean}
norm $\|\cdot\|_{\sf Eucl}$.}

{\color{black} Quantities similar to the integrand of \eqref{eq:nb_intro_metric} appeared previously in the context of metric differentiability of Lipschitz maps \cite{As14,Kir94}.}

{In the present paper, we generalize the
$\nb_{q,\mathcal{V}_k}$-numbers to metric spaces in a different
way. We replace the orthogonal projections in \eqref{eq:nb_intro}
by an abstract family of mappings from a metric space $X$ onto
subsets $V\in \mathcal{V}$ of $X$, where the family $\mathcal{V}$
satisfies axiomatic properties akin to the family of affine
$k$-dimensional subspaces of $\mathbb{R}^n$.}

\subsection{Geometric lemmas in an axiomatic setting: from $\nb$- to $\beta$-numbers}
{Finding  a (suitable version of) Theorem
\ref{t:Char_iota_Eucl} for $k$-regular sets in general metric
spaces remains an open problem for $k>1$. For $k=1$ \textcolor{black}{and the coefficients defined in \eqref{eq:nb_intro_metric}}, we obtained
such a characterization in \cite{CarlesonPart1}. {\color{black} We report here a simplified version of this statement. 
\begin{thm}
    Let $(\X,\sfd)$ be a complete, doubling and quasiconvex metric space and let $E\subset \X$ be a 1-regular set. Then $E$ is contained in a closed and connected 1-regular set  if and only if $E$ satisfies $\mathrm{GLem}(\iota_{1,1},1)$.
\end{thm}
Concerning results valid in higher dimensions,
}
Bate, Hyde, and
Schul \cite{2023arXiv230612933B} characterized, for all $k\in
\mathbb{N}$ and in arbitrary metric spaces, $k$-regular sets with
big pieces of Lipschitz images of $\mathbb{R}^k$ as those
$k$-regular sets that satisfy a Gromov-Hausdorff \emph{bilateral
weak geometric lemma}, or some other equivalent conditions
inspired by Euclidean quantitative rectifiability, but this
characterization does not include a condition in terms of a
(strong) geometric lemma.  We do not claim to obtain here new
characterizations of uniform rectifiability beyond the Euclidean
setting, but motivated by this quest, we prove Theorem
\ref{thm:axiomatic}, which allows to pass from a geometric lemma
for $\beta$-type numbers to a corresponding statement for
$\nb$-type numbers. In particular, this abstract theorem, which we
state here in shortened form, is an important ingredient in the
proof of Theorem \ref{t:Char_iota_Eucl}:}

\begin{thm}[GLem for $\beta$-numbers implies GLem for $\nb$-numbers]\label{thm:axiomaticIntro}
Fix {$p\in [1,\infty)$}.
     Let $(\X,\sfd)$ be a metric space, $E\subset \X $ be $s$-Ahlfors regular and let $\Delta$
     be  a system of dyadic cubes for $E$ (see Definition \ref{dl:dyadic}).
      Let also  $(\mathcal V,\mathcal P,\angle)$ be a system of planes-projections-angle in the sense of Definition \ref{def:planes system}
      and assume that it satisfies the tilting estimate for $E$ stated
      in Theorem \ref{thm:axiomatic}.
     Then for all {$q\ge p$} {\color{black}(and for all $q\in [1,\infty)$ in the case $p<s$)} it holds:
     \[
     E \in {\rm GLem}(\beta_{2p,\mathcal V},2q) \implies  E \in {\rm GLem}(\nb_{p,\mathcal V},q),
     \]
    {where the constant in ${\rm GLem}(\nb_{p,\mathcal V},q)$ can be controlled in a quantitative way independent of $E$.}
\end{thm}

Formally speaking, a system of \textit{planes-projections-angle}  $(\mathcal V,\mathcal P,\angle)$ is composed by a family $\mathcal V$ of subsets of $\X$ (`planes') together with a collection $\mathcal P$  of 1-Lipschitz `projections' from $\X$ to elements in $\mathcal V$ and an angle function  $\angle(\cdot,\cdot)$ which allows to measure the `distance' between elements in $\mathcal V.$ Additionally a Pythagorean-type inequality which relates points with their projections onto planes is assumed.  The \textit{tilting estimate} instead, roughly speaking, asks that whenever the set $E$ is well approximated 
in two nearby balls $B_1,B_2$ respectively by two planes $V_1,V_2\in \mathcal V$,  then $\angle(V_1,V_2)$ is small in a quantitative sense.

The second author proved in \cite{MR4489627} a result similar in spirit to Theorem \ref{thm:axiomaticIntro}    in the Euclidean setting, but for Jones $\beta_{\infty}$-numbers and coefficients akin to
$\nb_{\infty}$-numbers, and for a summability condition linked to parametrization and rectifiability results \textcolor{black}{\cite{MR1731465,MR2907827}}.

It is not difficult to see that the  assumptions in
Theorem \ref{thm:axiomaticIntro} are satisfied in $\rr^n$ for   $\mathcal V$ the
$k$-dimensional affine planes, $\mathcal P$ the orthogonal projections onto elements of $\mathcal V $ and $\angle$ the usual angle between planes. As a consequence,
the conclusion of Theorem \ref{thm:axiomaticIntro} holds true for
$k$-regular sets in Euclidean space $\mathbb{R}^n$. This is
Theorem \ref{thm:beta implies alpha euclidean}. In Euclidean
spaces a converse implication is also true, as follows from
Proposition \ref{prop:converse inequalities}. Together with the
known characterization of uniform $k$-rectifiability via
$\beta$-numbers \cite{MR1251061}, this yields Theorem
\ref{t:Char_iota_Eucl}.

The primary new application of Theorem \ref{thm:axiomaticIntro}
in this note is for $k$-regular sets in Heisenberg
groups $\mathbb{H}^n$ with the Kor\'{a}nyi distance, for $n\geq k$ (Theorem \ref{t:from
beta_to_alpha_in_heis}). The dimension range is
crucial here. For instance, $(\mathbb{H}^1,d_{\mathbb{H}^1})$ is purely $k$-unrectifiable for $k\in
\{2,3,4\}$, recall \cite{MR1800768}, and thus (bi)-Lipschitz
images of subsets in $\mathbb{R}^k$ for $k\geq 2$ cannot be used
as building blocks for an interesting theory of quantitative
rectifiability in $\mathbb{H}^1$. On the other hand, for
\emph{low-dimensional} sets in Heisenberg groups $\mathbb{H}^n$, a
definition of (quantitative) rectifiability based on Lipschitz
images from $\mathbb{R}^k$ for $k\in \{1,\ldots,n\}$ is natural, see for instance \cite{MR4277829}.

In this setting, condition ${\rm GLem}(\beta_{1,\mathcal V^k},2)$
(for suitable \emph{horizontal} subspaces $\mathcal{V}$) has been
studied earlier by Hahlomaa \cite{hahheisenberg}, who proved that
it implies for $k$-regular sets in $\mathbb{H}^n$, $k\leq n$, the
existence of big pieces of bi-Lipschitz images of subsets of
$\mathbb{R}^k$. We believe that the $\iota$-numbers could be better suited to
\emph{characterize} uniform $k$-rectifiability in $\mathbb{H}^n$
for $k\leq n$ than the horizontal $\beta$-numbers. It is easy to
see that  ${\rm GLem}(\beta_{\infty,\mathcal V^1},p)$
\emph{cannot} be used to characterize $1$-uniform rectifiability
in $\mathbb{H}^n$, $n>1$, see Proposition \ref{p:NonChar}. This observation is based on a construction in
$\mathbb{H}^1$ due to N.\ Juillet \cite{MR2789375} and it is in
stark contrast with the situation in Euclidean spaces. A
similar phenomenon has been observed earlier by Li
\cite{https://doi.org/10.1112/jlms.12582} in the Carnot group $\mathbb{R}^2\times \mathbb{H}^1$ in connection with
the traveling salesman theorem.

\smallskip

\textcolor{black}{\textbf{Acknowledgments.}  We wish to thank the referees for their feedback and valuable suggestions that helped to improve the presentation of the paper.}

\smallskip

\textbf{Structure of the paper.} {Section
\ref{s:prelim} contains preliminaries. In Section
\ref{s:axiomatic}, we prove the axiomatic result, Theorem
\ref{thm:axiomaticIntro}, and deduce the Euclidean result, Theorem
\ref{t:Char_iota_Eucl} . In the second part of the paper, we apply
the abstract results from Section \ref{s:axiomatic} to $k$-regular
sets in Heisenberg groups $\mathbb{H}^n$ for $n\geq k$ (Theorem
\ref{t:from beta_to_alpha_in_heis}), and we make some related
observations. In Appendix \ref{s:AppendixB} we show a technical result about planes in the Euclidean space, which is used in the proof of Theorem
\ref{t:from beta_to_alpha_in_heis}.}

\smallskip

\section{Preliminaries}\label{s:prelim}
\textbf{Notation.} We write $A \lesssim B$ to denote the existence
of an absolute constant $C \geq 1$ such that $A \leq CB$. The
 inequality $A \lesssim B \lesssim A$ is abbreviated to
$A \sim B$. If the constant $C$ is allowed to depend on a
parameter "$p$", we indicate this by writing $A \lesssim_{p} B$.
We denote the diameter of a set $E$ in a metric space by
$\mathrm{diam}(E)$ and use the convention that
$\mathrm{diam}(E)=+\infty$ if $E$ is unbounded.

\subsection{Standard quantitative notions}

{Throughout the paper we employ various
quantitative notions related to uniform rectifiability. The
terminology used in Sections \ref{ss:Ahlfors}-\ref{sec:glem}
closely follows the presentation in \cite{MR4485846} in the case
of Hausdorff measures $\mu=\mathcal{H}^s|_E$. The same notions
were also used in \cite{CarlesonPart1}, where we proved relevant
properties and stated additional examples.}

We denote by $B_r(x)=\{y\in \X: \sfd(x,y)< r\}$ the open ball with
center $x$ and radius $r$ in a given metric space $(\X,\sfd)$.

\subsubsection{Ahlfors regular sets and dyadic
systems}\label{ss:Ahlfors}

\begin{definition}[$s$-regular sets]
A set $E\subset (\X,\sfd)$ with $\mathrm{diam}(E)>0$ is said to be
\emph{$s$-regular}, $s>0$, if it is closed
 and
there exists $C\ge 1$, called \emph{regularity constant}, such
that
\begin{equation}\label{eq:regular}
    C^{-1}r^s\le \mathcal{H}^s(B_r(x)\cap E)\le C r^s, \quad  x \in E,\,\,
    r\in(0,2\diam(E)),
\end{equation}
in which case we write $E\in\mathrm{Reg}_s(C)$.
 Furthermore if only the first (resp.\ the second) inequality in
     \eqref{eq:regular} is satisfied and  $E$ is not necessarily closed we say that $E$ is \emph{lower} (resp.\ \emph{upper})
      $s$-regular.
      Finally we say that the metric space $(\X,\sfd)$ is $s$-regular if the whole set $\X$ is an $s$-regular
      set with respect to $\sfd$. We also use the term \emph{Ahlfors regular} to denote the class of sets that are $s$-regular for some exponent $s$.
\end{definition}

Regular sets in
metric spaces admit systems of generalized dyadic cubes. For
$k$-regular sets in $\mathbb{R}^n$, the existence of such systems
was proven by David in \cite[B.3]{MR1009120}, \cite{MR1123480}.
More generally, Christ constructed dyadic cube systems for spaces
of homogeneous type in \cite[Theorem 11]{MR1096400}.
We use the version for Ahlfors regular sets in metric spaces as
stated in \cite[Lemma 2.5]{MR4485846}, see also \cite[Sect.
5.5]{MR1616732}{, with a separate notational
convention for bounded sets.}
 If the
regular set $E$ is bounded, we define $\mathbb{J}:= \{j\in
\mathbb{Z}\colon j\geq n\}$ where $n\in \mathbb{Z}$ is such that
$2^{-n} \leq \mathrm{diam}(E) < 2^{-n+1}$, otherwise we denote
$\mathbb{J}:=\mathbb{Z}$.

\begin{thmdef}[Dyadic systems \cite{MR1009120,MR1096400}]\label{dl:dyadic} For any $s>0$ and $C\geq 1$, there exists
a constant $c_0\in (0,1)$  such that in an
arbitrary
metric space, every set $E\in\mathrm{Reg}_s(C)$
 admits a \emph{system of dyadic cubes}
$\Delta=\bigcup_{j\in\mathbb{J}} \Delta_j$, where $\Delta_j$ is a
family of pairwise disjoint Borel sets $Q\subset E$ (\emph{cubes})
satisfying
\begin{enumerate}
\item\label{dyadic1} $E=\bigcup_{Q\in \Delta_j}Q$ for each $j\in
\mathbb{J}$, \item\label{dyadic2} for $i,j\in \mathbb{J}$ with
$i\leq j$, if $Q\in \Delta_i$ and $Q'\in \Delta_j$, then either
$Q'\subset Q$ or $Q\cap Q'=\emptyset$, \item\label{dyadic3} for
$j\in \mathbb{J}$, $Q'\in \Delta_j$ and $i<j$ with $i\in
\mathbb{J}$, there is a unique $Q\in \Delta_i$ (\emph{ancestor})
such that $Q'\subset Q$, \item\label{dyadic4} for $j\in
\mathbb{J}$ and $Q\in \Delta_j$, it holds $\mathrm{diam}(Q)\leq
c_0^{-1} 2^{-j}$, \item\label{dyadic5} for $j\in \mathbb{J}$ and
$Q\in \Delta_j$, there is a point $x_Q \in E$ (\emph{center}) such
that $B_{c_0 2^{-j}}(x_Q)\cap E \subset Q$.
\end{enumerate}
For $j\in \mathbb{J}$ and $Q\in \Delta_j$, we denote $\ell(Q):=
2^{-j}$ and refer to this as the \emph{side length} of the cube.
We also define
\begin{displaymath}
\Delta_{Q_0}:=\{Q\in \Delta\colon Q\subset Q_0\},\quad Q_0 \in
\Delta,
\end{displaymath}
and  for a given constant $K>1$, we set
\begin{equation}\label{eq:KQ}
KQ:= \{x\in E\colon \mathrm{dist}(x,Q) \leq
(K-1)\,\mathrm{diam}(Q)\}.
\end{equation}
\end{thmdef}

{It follows from the definition that}
\begin{equation}\label{eq:MeasCube}
C^{-1} (c_0 \ell(Q))^s \leq \mathcal{H}^s(Q) \leq C (c_0^{-1}
\ell(Q))^s \quad \text{and}\quad C^{-2/s} c_0 \ell(Q) \leq
\mathrm{diam}(Q) \leq c_0^{-1} \ell(Q).
\end{equation}
Combining the second {estimate} in
\eqref{eq:MeasCube} and condition \eqref{dyadic1} we can infer the
existence of a constant $K=K(s,C)>1$ such that the following holds
for all $z\in E$ and $0<R< \mathrm{diam}(E)$. If $j\in \mathbb{J}$
is such that $2^{-j}\leq R <2^{-j+1}$, then there exists $Q\in
\Delta_j$ such that
\begin{equation*}
E\cap B(z,R)\subset K Q.
\end{equation*}
For every $Q\in \Delta_{j_0}$ and $j \in \mathbb N\cup\{0\}$ we
define the $j$-th descendants of $Q$ by
 \begin{equation}\label{def:descendants}
    F_j(Q)\coloneqq \{Q'\in  \Delta_{j+j_0} \ : \ Q'\subset Q\}.
 \end{equation}
It is easy to deduce from the first part of \eqref{eq:MeasCube},
and observing  that the cubes in $F_i(Q)$ are pairwise
disjoint, that
\begin{equation}\label{eq:number of descendants}
    \card(F_j(Q))\le c 2^{s\cdot j},
\end{equation}
for some constant $c$ depending only on $s$ and $C$.  Similarly, using again \eqref{eq:MeasCube}, for all $K\ge 1$ and all $Q\in \Delta_j$, $j\in \mathbb J,$ we deduce that there exist cubes $Q_1,\dots,Q_m \in \Delta_j$, not necessarily distinct, such that
\begin{equation}\label{eq:cube patch}
    KQ\subset \cup_{i=1}^m Q_i\subset K_0Q
\end{equation} where $m\in \mathbb N$ and $K_0>1$ are constants depending only on $s,$  $C$ and $K$.
We observe also the following elementary fact
\begin{equation}\label{eq:granchildren}
    \bigcup_{Q'\in F_j(Q)} F_i(Q')=F_{i+j}(Q), \quad Q\in \Delta, \,\,  i,j \in \mathbb N\cup \{0\}.
\end{equation}
Finally we  note that combining \eqref{dyadic1} and \eqref{dyadic2} in Definition \ref{dl:dyadic} it follows that
\begin{equation}\label{eq:sum of parts}
    \sum_{Q'\in F_j(Q)}\mathcal{H}^s(Q')=\mathcal{H}^s(Q), \quad Q \in \Delta,\,  j \in \mathbb N \cup\{0\}.
\end{equation}

\subsubsection{Geometric lemmas for various coefficient functions}\label{sec:glem}

{The main notion studied in this paper is a
Carleson-type summability condition in the spirit of a
\emph{geometric lemma} for a given set of coefficients. These
coefficients measure how well an $s$-regular set $E$ satisfies a
certain property at the scale and location of a given dyadic cube
$Q$. We use the same terminology as in \cite{CarlesonPart1}, and
refer to the latter paper for more details and examples.}

We let $\mathcal{B}(\X)$ be the Borel $\sigma$-algebra of a metric
space $(\X,\sfd)$. For a closed set $E\subset \X$, the family
$\{B\cap E\colon B\in \mathcal{B}(\X)\}$ coincides with the Borel
$\sigma$-algebra on $E$ with respect to the topology induced by
the metric $\sfd|_E$. We denote by $\mathcal{D}_s(E)$ the family
of bounded Borel sets in $E$ that have positive $\mathcal{H}^s$
measure. In particular, if $E$ is $s$-regular and $\Delta$ a
dyadic system on $E$, then $\Delta \subset \mathcal{D}_s(E)$ and
also $KQ\in \mathcal{D}_s(E)$ for every $Q\in \Delta$ and $K>1$.

\begin{definition}[Geometric lemma]\label{d:GL}
Given $p\in (0,\infty)$, $s>0$, an $s$-regular set $E$ in a metric
space, $\mu\coloneqq\mathcal{H}^s\lfloor_E$
 and a
 function
$h:\mathcal{D}_s(E) \to [0,1]$, we say that $E$ satisfies the
\emph{$p$-geometric lemma with respect to $h$}, and  write
$E\in\mathrm{GLem}(h,p)$, if there exists a constant $M$ such that
for every dyadic system $\Delta$ on $E$, we have
\begin{equation}\label{eq:SGL}
\sum_{Q\in \Delta_{Q_0}} h(2Q)^p\, \mu(Q) \leq M \mu(Q_0), \quad
Q_0\in \Delta.
\end{equation}
In this case, we also write  $E\in\mathrm{GLem}(h,p,M)$.
\end{definition}

{An important instance of a  geometric lemma
concerns the coefficient function $h$ that yields the classical
$\beta$-numbers from Jones' traveling salesman theorem, or the
variants used by David and Semmes in the uniform rectifiability
theory, recall \eqref{eq:beta_intro}. In the following we will
focus on functions $h$ that yield a generalization of
$\beta$-numbers or $\nb$-numbers, see \eqref{eq:beta_q_V_planes}
and \eqref{d:newbetas proj}.} Under mild regularity conditions on the function $h$, it is equivalent to ask \eqref{eq:SGL} for a single dyadic system $\Delta$ (see \cite[Remark 2.16]{CarlesonPart1}).

\section{Relations between geometric lemmas for \texorpdfstring{$\beta$}{beta}- and  \texorpdfstring{$\nb$}{iota}-numbers}
\label{s:axiomatic}

The goal of this section is to compare two ways of measuring
``flatness'' for subsets of a metric space $\X$, where
``flatness'' is understood in a broad sense as approximation by
elements from a family $\mathcal{V}$ of subsets of $\X$. The result will be stated in the form of Theorem \ref{thm:axiomaticIntro} from the introduction, {see Theorem \ref{thm:axiomatic} for a detailed version.}
This is
inspired by \cite{MR4489627}, and specifically by \cite[Theorem
B]{MR4489627} and \cite[Proposition 4.2]{MR4489627}, where the
second author proved related results for {summability conditions} involving Jones'
$\beta_{\infty}$-numbers and coefficients similar to the
$\nb_{\infty}$-numbers in Euclidean spaces.

\begin{definition}[System of planes-projections-angle]\label{def:planes system}
    Let $(\X,\sfd)$ be a metric space. A \emph{system of planes-projections-angle}
    is a triple $(\mathcal V,\mathcal P,\angle)$, where $\mathcal V$ is a non-empty family of non-empty subsets of $\X$,
     called  \emph{planes},   $\mathcal P=\{\pi_V\}_{V\in \mathcal V}$ is a family of 1-Lipschitz maps
      $\pi_V:\X\to V$, called  \emph{projections} and $\angle$ is a function $\angle: \mathcal V\times \mathcal V\to [0,1]$,
      called \emph{angle function},  {\color{black}vanishing on the diagonal} and such that the following conditions hold:
    \begin{enumerate}[label=\roman*)]
        \item\label{it:triangle angle} $\angle(V_1,V_3)\le \angle(V_1,V_2)+\angle(V_2,V_3)$,
        for all $V_1,V_2,V_3\in \mathcal  V$,
        \item\label{it:pit} for some constant $C_P\ge 1$ (``Pythagorean constant'') and  for every $x,y\in \X$, all $V
        \in \mathcal V$ satisfying
        \[
        C_P\max(\sfd(x,V),\sfd(y,V))\le \sfd(x,y),
        \]
         and all $ W\in \mathcal{V}$ it holds that
 \begin{equation}\label{eq:pit}
 \begin{split}
      \sfd(x,y)^2\le
      \sfd(\pi_W(x),\pi_W(y))^2+C_P^2(\angle(V,W)\sfd(x,y)+\sfd(x,V)+\sfd(y,V))^2.
 \end{split}
        \end{equation}
    \end{enumerate}
\end{definition}
A concrete and model example of a system of planes-projections-angle is  the family $\mathcal V_k$ of $k$-dimensional affine planes  in $\rr^n$ endowed with the orthogonal projections (see Section \ref{s:euclidean planes} for the details). We will  show in Section \ref{s:beta_alpha_Heis} that the Heisenberg groups also admit such structures.

{Given a metric space $(\X,\sfd)$, assume that  $\mathcal{V}$ is a
family of subsets of $\X$ such that every point in  $\X$ is
contained in at least one element of $\mathcal{V}$.} Let $E\subset
\X$ be an $s$-regular set, {and $\mu\coloneqq
\mathcal H^s|_E$}.  {We will use the following
coefficients:}

\begin{definition}[$\beta$-numbers]\label{d:AbstractBeta}
 For every  ${p}\in[1,\infty)$ and every $S\in D_s(E)$ ,
{we define the coefficient} $\beta_{p,\mathcal
V}(S)$
{as follows:}
\begin{equation}\label{eq:beta_q_V_planes}
\beta_{p,\mathcal V}(S)=\inf_{V\in \mathcal V}\beta_{p,V}(S)\coloneqq \inf_{V\in \mathcal V} \left(\frac{1}{\mu(S)}\int_S \left[\frac{\sfd(x,V)}{\diam(S)}\right]^p d\mu(x)  \right)^\frac1p.
\end{equation}
\end{definition}

\begin{definition}[$\nb$-numbers]\label{d:nb_V}
For every $V\in \mathcal V$, $p\in[1,\infty)$ and every $S\in
\mathcal{D}_s(E)$,  we  define
\[
\nb_{p,V}(S)\coloneqq \left(\frac{1}{\mu(S)^2}\int_S\int_S
\left[\frac{|\sfd(x,y)-\sfd(\pi_V(x),\pi_V(y))|}{\diam(S)}\right]^p
d\mu(x) d\mu(y) \right)^\frac1p
\]
and
\begin{equation}\label{d:newbetas proj}
    \nb_{p,\mathcal V}(S)=\inf_{V\in \mathcal V}\nb_{p,V}(S).
\end{equation}
\end{definition}
Taking $\mathcal V=\mathcal V_k$ the class of $k$-dimensional affine planes in $\rr^d$ and $\pi_V$ the orthogonal projection onto $V$, the coefficients defined above for $S=E\cap B_r(x)$ coincide  with the numbers $\beta_{p,\mathcal V_k}$ and $\iota_{p,\mathcal V_k}$ defined in the introduction in \eqref{eq:beta_intro}  and \eqref{eq:nb_intro}.

\begin{remark}
{Definition \ref{d:nb_V} reminds of the
$\nb_{p,k}$- and $\nb_{q,k,{\sf Eucl}}$-numbers
which were studied in \cite{CarlesonPart1}. In and below \eqref{eq:nb_intro_metric}, we recalled the form of these coefficients for $S=B_r(x)\cap E$ for a $k$-regular set $E$, but the definition can be stated for any $S\in \mathcal{D}_k(E)$ as in \cite[Definition 2.31]{CarlesonPart1}}. If every $V\in \mathcal{V}$ is isometric to
$\mathbb{R}^k$ with the Euclidean distance, $s=k$, and $E$ is a
$k$-regular subset of $(\X,\sfd)$, then, for $p\in [1,\infty)$, we
have
\begin{displaymath}
\nb_{p,k}(S)\leq\nb_{p,k,{\sf Eucl}}(S)\leq\nb_{p,\mathcal{V}}(S),\quad S\in
\mathcal{D}_k(E).
\end{displaymath}
\end{remark}

{Our main result, Theorem \ref{thm:axiomatic},
relates $\beta_{q,\mathcal{V}}$- and $\nb_{q,\mathcal{V}}$-numbers
in an axiomatic setting, by providing conditions under which the
validity of $\mathrm{GLem}(\beta_{2p,\mathcal{V}},2q)$ for a set
$E$ implies $\mathrm{GLem}(\nb_{p,\mathcal{V}},q)$. \textcolor{black}{For a certain range of parameters, this implication follows from a "pointwise" inequality, valid for each individual cube.
For this argument, condition i) in Definition \ref{def:planes system} is not needed, and condition ii) is only required to hold for $V=W$, in which case it can be stated without reference to an angle function. The pointwise estimate holds thanks to the integrability of a certain factor which is guaranteed in this specific range of parameters, see \eqref{eq:Integrability}.}

{\color{black}

\begin{proposition}[Pointwise estimate for  $p<s$]\label{prop:pointwise}
     Let $(\X,\sfd)$ be a metric space, let $E\subset \X $ be  $s$-regular and $\Delta$ be a system of dyadic cubes for $E$. Let $V\subset \X$ be a set and $\pi_V:\X\to V$ be an $L$-Lipschitz map  with the following property: There is a constant  $C_V\ge 1$ such that,
     for every $x,y\in \X$ satisfying
        \[
        C_V\max(\sfd(x,V),\sfd(y,V))\le \sfd(x,y),
        \]
          it holds that
 \begin{equation}\label{eq:SimplifiedPyth}
 \begin{split}
      \sfd(x,y)^2\le
      \sfd(\pi_V(x),\pi_V(y))^2+C_V^2(\sfd(x,V)+\sfd(y,V))^2.
 \end{split}
        \end{equation}
    Then for all $p\in [1,s)$ and all $Q\in \Delta$ it holds
     \begin{equation}\label{eq:pointwise iotabeta}
         \nb_{p,V}(2Q)\le C_{p,s} \beta_{2p,V}(2Q)^2,
     \end{equation}
     where $C_{p,s}$ is a positive constant depending only on $p,s,C_V,L$ and the regularity constant of $E.$
\end{proposition}

\begin{proof}
    Fix any $Q\in \Delta$. Let $x,y\in2 Q$ be such that $\sfd(x,y)\ge C_V\max(\sfd(x,V),\sfd(y,V))$. Then using \eqref{eq:SimplifiedPyth}: 
    \begin{align*}
       & |\sfd(x,y)-\sfd(\pi_V(x),\pi_V(y))|=\frac{ |\sfd(x,y)^2-\sfd(\pi_V(x),\pi_V(y))^2|}{ \sfd(x,y)+\sfd(\pi_V(x),\pi_V(y))}\\
        &\le C_V^2\frac{(\sfd(x,V)+\sfd(y,V))^2}{\sfd(x,y)+\sfd(\pi_V(x),\pi_V(y))}\le C_V^2\frac{(\sfd(x,V)+\sfd(y,V))^2}{\sfd(x,y)}.
    \end{align*}
    If instead $x,y\in 2Q$ are such that $\sfd(x,y)\le C_V\max(\sfd(x,V),\sfd(y,V))$ we have
    \begin{align*}
         |\sfd(x,y)-\sfd(\pi_V(x),\pi_V(y))|&\le (1+L)\sfd(x,y)\le (1+L)C_V^2\frac{\max(\sfd(x,V),\sfd(y,V))^2}{\sfd(x,y)}\\
         &\le  (1+L)C_V^2\frac{\sfd(x,V)^2+\sfd(y,V)^2}{\sfd(x,y)}.
    \end{align*}
    Using the above estimates and integrating we have
    \begin{align*}
       \nb_{p,V}(2Q)^p&=\frac{1}{\mu(2Q)^2}\int_{2Q}\int_{2Q}
\left[\frac{|\sfd(x,y)-\sfd(\pi_V(x),\pi_V(y))|}{\diam(2Q)}\right]^p
d\mu(x) d\mu(y)\\
&\le \frac{2^{p-1}(1+L)^pC_V^{2p}}{\mu(2Q)^2}\int_{2Q}\int_{2Q}
\frac{\sfd(x,V)^{2p}+\sfd(y,V)^{2p}}{\sfd(x,y)^p\diam(2Q)^p}d\mu(x) d\mu(y)\\
&\leq 2^{p}(1+L)^pC_V^{2p}\fint_{2Q} \frac{\sfd(z,V)^{2p}}{\diam(2Q)^p}d\mu(z)\cdot  \sup_{w\in 2Q} \fint_{2Q} \frac{1}{\sfd(z,w)^p}\d \mu(z)\\
&=2^{p}(1+L)^pC_V^{2p} \beta_{2p,V}(2Q)^{2p}\cdot  \sup_{w\in 2Q} \fint_{2Q} \frac{\diam(2Q)^p}{\sfd(z,w)^p}\d \mu(z)
    \end{align*}
    Using now the $s$-regularity of $E$ we can estimate the last integral as follows
    \begin{align}\label{eq:Integrability}
        \fint_{2Q} \frac{\diam(2Q)^p}{\sfd(z,w)^p}\d \mu(z)&=\frac{\diam(2Q)^p}{\mu(2Q)}\sum_{k=0}^\infty \int_{\{\diam(2Q)2^{-k-1}\le \sfd(z,w)\le \diam(2Q)2^{-k}\}}  \sfd(z,w)^{-p}\d \mu(z)\notag\\
        &\le \frac{1}{\mu(2Q)}\sum_{k=0}^\infty 2^{p(k+1)} \mu(B_{\diam(2Q)2^{-k+1}}(w))\notag\\
        &\le c\frac{\diam(2Q)^{s}}{\mu(2Q)}\sum_{k=0}^\infty 2^{p(k+1)+s(1-k)} 
        = c\frac{\diam(2Q)^{s}}{\mu(2Q)} c_{p,s}\\
        &\le \tilde c\cdot  c_{p,s}
    \end{align}
    where $c>0$, $\tilde c>0$ are constants depending only on $s$ and $C_E$ (the regularity constant of $E$) and where $c_{p,s}>0$ is finite number depending only on $p$ and $s,$ provided $p<s.$ In this last step, we used the definition \eqref{eq:KQ} of $2Q$ as well as \eqref{eq:MeasCube}.
    This shows \eqref{eq:pointwise iotabeta}.
\end{proof}

}

\subsection{Beyond pointwise estimates}
\textcolor{black}{In this section, we will complete the proof of our main result, Theorem \ref{thm:axiomatic}.}
As  alluded to
in the introduction, even in the Euclidean plane, where this
implication holds, the \emph{pointwise} inequality
    $\nb_{p,\mathcal{V}_k}(B_r(x)\cap E)\lesssim\beta_{2p,\mathcal{V}_k}(B_r(x)\cap E)^2$ \textcolor{black}{from Proposition \ref{prop:pointwise}} does not hold in general. We now give the  details  of a construction showing this fact.}

    {\begin{proposition}\label{p:failure}
    Given any $\eps\ll 1$
and $r>0$, let $E\subset \rr^2$ be the union of the horizontal axis $l_0$ and a parallel line $l$ at distance $\eps r$.  Then,
for any $q\in [1,\infty)$, it holds
$\beta_{{\color{black}2q},\mathcal{V}_1}(B_r(x)\cap E) \sim \eps$ and
$\nb_{q,\mathcal{V}_1}(B_r(x)\cap E) \gtrsim \log(\eps^{-1})^\frac1q \eps^2$,
for every $x\in E$.
\end{proposition}\begin{proof} The fact that $\beta_{{\color{black}2q},\mathcal{V}_1}(B_r(x)\cap E) \sim \eps$ is easily checked, so we focus on showing that $\nb_{q,\mathcal{V}_1}(B_r(x)\cap E) \gtrsim \log(\eps^{-1}) \eps^2$. We need to consider the infimum
among all projections onto lines $V$. We assume for now that
$V=l_0$. For every couple of points $x\in l_0$ and  $y\in l$
such that $ \overline{xy}\coloneqq |x-y|\ge \eps r$ we have
\begin{equation}\label{eq:pythagora trick}
    |\pi_V(x)-\pi_V(y)| =  \overline{xy} \sqrt{1-\eps^2 r^2/ \overline{xy}^2}\le  \overline{xy}-\frac{\eps^2 r^2}{2 \overline{xy}},
\end{equation}
using that $\sqrt{1-t^2}\le 1-t^2/2$ for all $t\in [0,1]$ (see Figure \ref{fig:example}). It is easy to check that for any $x\in l_0$ \textcolor{black}{and} any number $d \in[2\ep,1/2]$, the points $y\in l$ such that  $d r\le  \overline{xy}\le 2d r$ form a set of $\mathcal H^1$-measure comparable to $d r$ and thus the couples $(x,y)\in (B_r(0)\cap E)^2$ of this type form a set of $\mathcal H^1\otimes \mathcal H^1$-measure comparable to $d r^2$. Hence, thanks to \eqref{eq:pythagora trick}, their contribution to the integral inside \eqref{eq:nb_intro} is $\gtrsim\eps^{2q}.$ Summing over all $d=2^{-k}\in [2\eps,1/2]$ we deduce that
\[
\nb_{q,\mathcal{V}_1}(B_r(x)\cap E) \gtrsim \log(\eps^{-1})^\frac1q \eps^2
\]
For a general line $V$, the argument is the same noting that, for any $x$ in $l_0$, it holds $|\pi_{V}(x)-\pi_{V}(y)|\le |\pi_{l_0}(x)-\pi_{l_0}(y)| $  for half of the points $y\in l$ such that $ \overline{xy}\ge 2\eps r.$ Indeed we can assume that $V$ forms an angle $\theta\ge \pi/2$ and we can take the points $y$ such that the segment $xy$ forms an angle $\alpha \le \pi/4$ with $l_0$, so that $|\alpha|\le |\alpha-\theta|$ (see Figure \ref{fig:example}).
\begin{figure}[!htb]    \label{fig:line}
    \centering
    \includegraphics[scale=0.6]{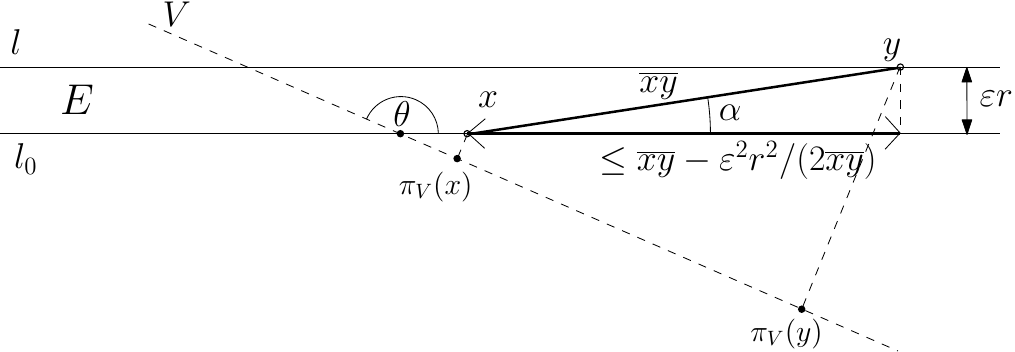}
    \caption{Example of set $E$ where $ \nb_{1,\mathcal{V}_k}\lesssim(\beta_{2,\mathcal{V}_k})^2$ fails at scale $r$.}\label{fig:example}
\end{figure}
\end{proof}

{\color{black}
\begin{remark}
    The above counterexample for the pointwise estimate easily extends to higher dimensions. More precisely given $\eps\ll 1$
and $r>0$, take  $E$ to be the union of two parallel $k$-dimensional planes $V_1,V_0$ at distance $\eps r$.  Then,
for any $q\in [k,\infty)$, it holds
$\beta_{2q,\mathcal{V}_k}(B_r(x)\cap E) \sim \eps$, while
$\nb_{q,\mathcal{V}_k}(B_r(x)\cap E) \gtrsim \log(\eps^{-1})^\frac1q \eps^2$,
for every $x\in E$. Indeed estimate \eqref{eq:pythagora trick} still holds with $V=V_0$ and  the tuples $(x,y)\in (V_0\times V_1)$  such that  $d r\le  \overline{xy}\le 2d r$ have $\mathcal H^k\otimes \mathcal H^k$-measure comparable to $r^k(dr)^k.$ Hence their contribution to the integral inside \eqref{eq:nb_intro} is $\gtrsim d^{k-q}\eps^{2q}\gtrsim \eps^{2q},$ since $d\le 1/2$ and $q\ge k.$ The extension to a general $k$-plane $V$ is also similar.
\end{remark}
}

  {Despite the examples in Proposition \ref{p:failure}, the following implication holds true:}

\begin{thm}\label{thm:axiomatic}
Fix $p\in [1,\infty).$
     Let $(\X,\sfd)$ be a metric space, and let $(\mathcal V,\mathcal P,\angle)$ be a system of planes-projections-angle {such that every point in $\X$ is contained in at least one element of $\mathcal{V}$. Let $E\subset \X $ be  $s$-Ahlfors regular and suppose that  for all $\bar \lambda$, there exists a constant $C_T(\bar \lambda)>0$ (``tilting constant'') such that for every system
    $\Delta$ of dyadic cubes for $E$ (see Definition \ref{dl:dyadic}) the following \emph{tilting estimate} holds.} For every $Q_1 \in \Delta_j, Q_0\in \Delta_{j-1}\cup \Delta_j$, for some $j\in \mathbb J,$ and all constants $\lambda_0,\lambda_1 \in [1,\bar \lambda]$ satisfying $\lambda_1Q_1\subset \lambda_0Q_0$, it holds
 \begin{equation}\label{eq:tilting assumption}
             \angle(V_1,V_0)\le C_T(\bar \lambda)( \beta_{p,V_0}(\lambda_1Q_1)+\beta_{p,V_1}(\lambda_0Q_0)), \quad \text{ for all $V_0,V_1\in \mathcal V$.}
         \end{equation}

     Then for all $q\ge p$ {\color{black} (and all $q\in[1,\infty)$ in the case $p<s$)} it holds:
     \begin{equation}\label{eq:axiomatic}
          E \in {\rm GLem}(\beta_{2p,\mathcal V},2q,\textcolor{black}{M}) \implies  E \in {\rm GLem}(\nb_{p,\mathcal V},q,{\hat C M}),
     \end{equation}
     {where $\hat C$ depends only on  the Ahlfors regularity constant and exponent of $E$, $p,q$, the constant $C_P$ in Definition \ref{def:planes system}, and on the function $C_T(\cdot)$.}
\end{thm}

 The non-trivial part of Theorem \ref{thm:axiomatic} is the presence of  $q$ instead of $2q$ in the right-hand-side of \eqref{eq:axiomatic}. Indeed the implication
\[
 E \in {\rm GLem}(\beta_{2p,\mathcal V},2q) \implies  E \in {\rm GLem}(\nb_{p,\mathcal V},2q),
\]
always trivially holds by the fact that
\begin{equation}\label{eq:easy beta iota inequality}
     \nb_{p,\mathcal V}(S)\le 2\beta_{p,\mathcal V}(S), \quad S \in \mathcal D_s(E), \quad  p \in[1,\infty),
\end{equation}
which follows immediately by the triangle inequality. The gain of a factor 2 in the exponent in \eqref{eq:axiomatic} comes, roughly speaking,   from the assumption of  a Pythagorean-type inequality in \ref{it:pit} in the Definition \ref{def:planes system}.

The proof of the above theorem rests on the following key technical  result (recall \eqref{def:descendants} for the definition of the $j$-descendants $F_j(Q)$). Roughly speaking it says that we can estimate $\nb_{p,\mathcal V}(2 Q_0)$ with a sum of the square of the coefficients $\beta_{2p,\mathcal V}$ on all smaller scales and locations near $2Q_0$. The main point is the presence of the weight $2^{-sj}$, which implies that smaller scales become exponentially less relevant. 
\begin{lemma}\label{lem:weighted packing}
Let $(\X,\sfd)$, $(\mathcal V,\mathcal P,\angle)$, $E\in\reg_s(C)$ {for some $s,C>0$}, $\Delta$ dyadic system and $p\in [1,\infty)$ be as in Theorem \ref{thm:axiomatic}. Denote $\mu\coloneqq \mathcal{H}^s|_E$.
    Fix $j_0\in \mathbb{J}$ and $Q_0 \in \Delta_{j_0}$. Then there exist cubes $\{Q_0^i\}_{i=1}^m\subset \Delta_{j_0}$ such that
    $$2Q_0\subset \cup_{i=1}^m Q_0^i\subset  K_0 Q_0$$
     it holds
    \begin{equation}\label{eq:weighted packing}
        \mu(Q_0)\nb_{p,\mathcal V}(2 Q_0)^p\le \bar C \sum_{i=1}^m\sum_{j\ge 0} 2^{-sj}\sum_{Q \in F_j(Q_0^i)} \mu(Q) \beta_{2p,\mathcal V}(K_0 Q)^{2p},
    \end{equation}
    where $m\in \mathbb N$,  and $K_0\ge 1$ are constants depending only on $s$ and $C$, while $\bar C>0$   is a constant depending only on $p,s$, $C$, $C_P$ and {$C_T(\cdot)$ (where the last two are, respectively,  the constant in Definition \ref{def:planes system} and the function in \eqref{eq:tilting assumption}).}
\end{lemma}

We first show that this lemma is enough to conclude Theorem \ref{thm:axiomatic}.
\begin{proof}[Proof of Theorem \ref{thm:axiomatic}]
Let $(\X,\sfd)$, $E\in\reg_s(C)$, $\Delta$ dyadic system and $p\in [1,\infty)$ be as in the statement. {\color{black} If $p<s$ the statement follows immediately from Proposition \ref{prop:pointwise}, hence we need to consider only the case $p\ge s.$}
    Fix {$j_0\in \mathbb{J}$ and}  $Q_0\in \Delta_{j_0}$. For every $j\ge 0$ and $Q \in F_j(Q_0)$ let $\{Q^i(Q)\}_{i=1}^m\subset \Delta_{j_0+j}$ be the cubes given by Lemma \ref{lem:weighted packing} applied to $Q$. In particular $Q^i(Q)\subset K_0 Q\subset K_0Q_0$. Here $m$ and $K_0$ are constants depending only on $s$ and $C$. Moreover, by \eqref{eq:cube patch}, there exist cubes $Q_0^h\in\Delta_{j_0}$, $h=1,\dots,\tilde m$ such that
    $ K_0 Q_0\subset \cup_{h=1}^{\tilde m} Q_0^h $, where $\tilde m$ depends only on $s$ and $C.$  In particular
    \begin{equation}\label{eq:finally all inside}
         \bigcup_{Q\in F_j(Q_0)} \bigcup_{i=1}^m  Q^i(Q)\subset \bigcup_{h=1}^{\tilde m} F_{j}(Q_0^h), \quad  j \in \mathbb N\cup \{0\}.
    \end{equation}
We also observe that there is not too much overlap in the above inclusion, in the sense that for every $h=1,\dots,\tilde m$, $j \in \mathbb N\cup\{0\}$, and every $Q'\in  F_{j}(Q_0^h)$, it holds
\begin{equation}\label{eq:bounded overlapping}
    \#\mathcal S_{Q'}\coloneqq \#\{Q \ : \ Q^i(Q)=Q' \text{ for some $i=1,\dots,m$} \}\le c,
\end{equation}
where $c\ge 1$ is a constant depending only on $s$ and $C$.  Indeed for every $Q\in \mathcal S_{Q'}$ it holds that $Q\subset K_0 Q'.$ Moreover, $Q,Q'\in \Delta_{j+j_0}$ for every $Q\in \mathcal S_{Q'}$, hence the cubes in $\mathcal S_{Q'}$ are pairwise disjoint and $\mu(Q)\ge \tilde c \mu(K_0Q')$ for every $Q\in \mathcal S_{Q'}$, for some constant $\tilde c>0$ depending only on $s$ and $C$ (recall \eqref{eq:MeasCube}). This proves \eqref{eq:bounded overlapping}.

In what follows, we let $q\geq p$, and we write
$\nb_p(\cdot)$ and $\beta_{2p}(\cdot)$ in place of
$\nb_{p,\mathcal V}(\cdot)$ and $\beta_{2p,\mathcal V}(\cdot)$.
Moreover with $C_1>0$  we will denote a constant
whose value might change from line to line but
which is allowed to depend only $p,s,q$, $C$, $C_P$ and
{$C_T(\cdot)$ (where the last two are, respectively, the constant in
Definition \ref{def:planes system} and the function in \eqref{eq:tilting
assumption}).}

{We now derive a bound for the expression
appearing in the statement of $\mathrm{GLem}(\nb_p,q)$. First,
expressing the family of children of the fixed cube $Q_0$ in terms
of $j$-descendents, we obtain
 \begin{displaymath}
 \left(\sum_{Q\subset Q_0,\, Q\in \Delta} \mu(Q)\nb_p(2Q)^q\right)^\frac{p}{q}=\left(\sum_{j\ge 0}\sum_{Q \in F_j(Q_0)} \mu(Q)\nb_p(2 Q)^q\right)^\frac{p}{q}.
 \end{displaymath}
}
Using
\begin{equation}\label{eq:rewrite pq}
    \mu(Q)\nb_p(2Q)^q=[\mu(Q)\nb_p(2Q)^p \mu(Q)^{\frac pq-1}]^{\frac{q}{p}},
\end{equation}
{and Lemma \ref{lem:weighted packing} we can now
write}
\begin{align*}
   &\left(\sum_{j\ge 0}\sum_{Q \in F_j(Q_0)} \mu(Q)\nb_p(2 Q)^q\right)^\frac{p}{q}\\
   &{\overset{\eqref{eq:rewrite pq}}{=}\left(\sum_{j\geq 0}\sum_{Q\in F_j(Q_0)}\left[\mu(Q)\nb_p(2Q)^p\mu(Q)^{\frac{p}{q}-1}\right]^{\frac{q}{p}}\right)^{\frac{p}{q}}}\\
    &\overset{\eqref{eq:weighted packing}}{\le } {C_1}\left( \sum_{j\ge 0}\sum_{Q \in F_j(Q_0)}\left(\sum_{i=1}^m \sum_{l\ge 0} 2^{-sl}\sum_{Q'\in F_{l}(Q^i(Q))} \mu(Q)^{\frac pq -1}\mu(Q')\beta_{2p}(K_0 Q')^{2p}\right)^{\frac qp}\right)^\frac{p}{q}.
\end{align*}
{Applying the Minkowski inequality for sums (with
exponent $\alpha=q/p$) to the last expression, we conclude from
the above that}
\begin{align*}
    &\left(\sum_{Q\subset Q_0,\, Q\in \Delta} \mu(Q)\nb_p(2Q)^q\right)^\frac{p}{q}\\
    &\leq  {C_1}\sum_{l\geq 0}2^{-sl}\left( \sum_{j\ge 0}\sum_{Q \in F_j(Q_0)}\left(\sum_{i=1}^m \sum_{Q'\in F_{l}(Q^i(Q))} \mu(Q)^{\frac pq -1}\mu(Q')\beta_{2p}(K_0 Q')^{2p}\right)^{\frac qp}\right)^\frac{p}{q}.
    \end{align*}
{Since ${C_1}$ is allowed to depend on $s$ and $C$,
and since $m$ and depends only on these two parameters, up to
enlarging ${C_1}$ we can write}
\begin{align}\label{eq:iSum}
    &\left(\sum_{Q\subset Q_0,\, Q\in \Delta} \mu(Q)\nb_p(2Q)^q\right)^\frac{p}{q}\\
    &\le {C_1}  \sum_{l\ge 0} 2^{-sl} \left(\sum_{j\ge 0}\sum_{Q \in F_j(Q_0)} \sum_{i=1}^m \left( \sum_{Q'\in F_{l}(Q^i(Q))}\mu(Q)^{\frac pq -1}\mu(Q')\beta_{2p}(K_0 Q')^{2p}\right)^{\frac qp}\right)^\frac{p}{q}.\notag
\end{align}
{To bound the inner most sum, we use the inequality
$(a_1+\dots + a_n)^a\le n^{a-1} (a_1^a+\dots + a_n^a)$, which is
valid for all $a_i\ge 0$ and $a\geq 1$ as a consequence of
H\"older's inequality. Here we apply the inequality with $a=q/p$
and $n=\# F_{l}(Q^i(Q))$. Since $\# F_{l}(Q^i(Q))\lesssim_{s,C}
2^{sl}$ by \eqref{eq:number of descendants}, we thus have}
\begin{equation*}
\begin{split}
    & \left( \sum_{Q'\in F_{l}(Q^i(Q))}\mu(Q)^{\frac pq -1}\mu(Q')\beta_{2p}(K_0 Q')^{2p}\right)^{\frac qp}
    \\
    &\quad \quad\quad\quad\quad\quad\quad\quad\quad\quad\lesssim_{s,C,p,q} \sum_{Q'\in F_{l}(Q^i(Q))}\mu(Q)^{1-\frac q p} \mu(Q')^{\frac qp}2^{sl(\frac qp-1)}\beta_{2p}(K_0 Q')^{2q},
\end{split}
\end{equation*}
{which plugged into \eqref{eq:iSum} gives}
\begin{align*}
 &\left(\sum_{Q\subset Q_0,\, Q\in \Delta} \mu(Q)\nb_p(2Q)^q\right)^\frac{p}{q}\\
    &\le {C_1}  \sum_{l\ge 0} 2^{-sl} \left(\sum_{j\ge 0}\sum_{Q \in F_j(Q_0)} \sum_{i=1}^m  \sum_{Q'\in F_{l}(Q^i(Q))}\mu(Q)^{1-\frac q p} \mu(Q')^{\frac qp}2^{sl(\frac qp-1)}\beta_{2p}(K_0 Q')^{2q}\right)^\frac{p}{q}\\
     &\le  {C_1}  \sum_{l\ge 0} 2^{-sl} \left(\sum_{j\ge 0}\sum_{Q \in F_j(Q_0)} \sum_{i=1}^m  \sum_{Q'\in F_{l}(Q^i(Q))}\mu(Q')\beta_{2p}(K_0 Q')^{2q}\right)^\frac{p}{q}.
\end{align*}
{In the last inequality we used that $\mu(Q)2^{-sl}\sim_{s,C} \mu(Q')$
since the cubes $Q^i(Q)$ are of the same generation as $Q$, and
since $Q'\in F_l(Q^i(Q))$.}

    {We now continue the chain of inequalities applying  inclusion \eqref{eq:finally all inside} and inequality \eqref{eq:bounded overlapping}}:
\begin{align*}
 &\left(\sum_{Q\subset Q_0,\, Q\in \Delta} \mu(Q)\nb_p(2Q)^q\right)^\frac{p}{q}\\
     &\overset{\eqref{eq:finally all inside}, \eqref{eq:bounded overlapping}}{\le } c^{p/q}\cdot {C_1}  \sum_{l\ge 0} 2^{-sl} \left(\sum_{j\ge 0} \sum_{h=1}^{\tilde m} \sum_{\overline Q\in F_j(Q_0^h)} \sum_{Q'\in F_{l}(\overline Q)}\mu(Q')\beta_{2p}(K_0 Q')^{2q}\right)^\frac{p}{q}\\
    &\overset{\eqref{eq:granchildren}}{= } c^{p/q}\cdot {C_1}\sum_{l\ge 0} 2^{-sl} \left( \sum_{j\ge 0}\sum_{h=1}^{\tilde m} \sum_{\overline Q \in F_l(Q_0^h)}\sum_{Q'\in F_{j}(\overline Q)}\mu(Q')\beta_{2p}(K_0 Q')^{2q}\right)^\frac{p}{q}\\
    & \le {C_1}\sum_{l\ge 0} 2^{-sl} \left( \sum_{h=1}^{\tilde m} \sum_{\overline Q \in F_l(Q_0^h)}\sum_{j\ge 0}\sum_{Q'\in F_{j}(\overline Q)}\mu(Q')\beta_{2p}(K_0 Q')^{2q}\right)^\frac{p}{q}\\
    &\overset{E \in {\rm GLem}(\beta_{2p,\mathcal V},2q)}{\le } {C_1}\sum_{l\ge 0} 2^{-sl} \left(\sum_{h=1}^{\tilde m}  \sum_{\overline Q \in F_l(Q_0^h)} M\mu(\overline Q)\right)^\frac{p}{q}
    \overset{\eqref{eq:sum of parts}}{=} {C_1}\sum_{l\ge 0} 2^{-sl} \left( \sum_{h=1}^{\tilde m} M\mu(Q_0^h)\right)^\frac{p}{q}\\
    &\le {C_1}M^{p/q} \sum_{l\ge 0} 2^{-sl} \mu(Q_0)^\frac{p}{q}\le  {C_1}M^{p/q} \mu(Q_0)^\frac pq,
\end{align*}
where $M>0$ is the constant in the definition of ${\rm
GLem}(\beta_{2p,\mathcal V},2q)$ for $E$ (see also
{\cite[Lemma 2.23]{CarlesonPart1} and \cite[Remark
2.30]{CarlesonPart1}).}
This concludes the proof.
\end{proof}

It remains to prove Lemma \ref{lem:weighted packing}.

\begin{proof}[Proof of Lemma \ref{lem:weighted packing}]
    Fix $j_0\in \mathbb{J}$ and $Q_0\in \Delta_{j_0}$. {The proof is divided in five steps:}

    \medskip

    {\emph{Step 1: finding the cubes $\{Q_0^i\}_{i=1}^m$.}}

\noindent By  \eqref{eq:cube patch} there exist  cubes $\{Q_0^i\}_{i=1}^m\subset \Delta_{j_0}$ with $m\in \N$, possibly not distinct, such that
    \begin{equation}\label{eq:good patching}
        2Q_0\subset \cup_{i=1}^m Q_0^i\subset K_0Q_0
    \end{equation}
     where $m\in \mathbb N$ and $K_0>1$ are  constants depending only on the regularity constant of $E.$ Up to increasing $m$ by one and renumbering, we can  also assume  that $Q_0^1=Q_0$.

     We aim to prove \eqref{eq:weighted packing} for  $Q_0$ and this family $\{Q_0^i\}_{i=1}^m$.

     \bigskip

    { \emph{Step 2: partitioning the domain $2Q_0\times 2Q_0$.}}

\noindent {The expression
$\nb_{p,\mathcal{V}}(2Q_0)$, which we aim to control, involves a
double integral over $2Q_0$. Therefore we will partition
$2Q_0\times 2Q_0$ in a suitable way in terms of the distance
between points in $2Q_0$. For $j\in \mathbb{N}\cup\{0\}$, we}
denote
\begin{equation}\label{eq:A(j0,j)}
A(j_0,j):=\{(x,y)\in E{\times
E}\colon\frac{c_0}3\sfd (x,y) \in (2^{-j_0-j-1},2^{-j_0-j}) \}.
\end{equation}
{     We  claim that:
     \begin{enumerate}
     \item   \begin{equation}\label{eq:scales of distances}
        2Q_0\times 2Q_0 \subset \bigcup_{j\in \mathbb{N}\cup \{0\}}A(j_0,j),
    \end{equation}
     \item there exists a constant $K>1$ depending only on the regularity constant $C$ of $E$ such that
     \begin{equation}\label{eq:good overlapping}
 A(j_0,j)\cap  (2Q_0\times 2Q_0)\subset \bigcup_{i\in \{1,\ldots,m\}}\bigcup_{Q\in F_j(Q_0^i)}(KQ \times KQ),\quad j\in \mathbb{N}\cup\{0\}.
\end{equation}
     \end{enumerate}}
Indeed,   for all $x,y \in 2Q_0$, we have $\sfd (x,y)\le 3\diam
(Q_0)\le 3c_0^{-1} 2^{-j_0} $ and {hence
\eqref{eq:scales of distances} holds.}

{To see why \eqref{eq:good overlapping} holds,
fix $x,y\in 2Q_0$ such that $(x,y)\in A(j_0,j)$. Then} by
\eqref{dyadic2} in Definition \ref{dl:dyadic} we have $x \in Q$
for some $Q\in \Delta_{j+j_0}$ and, as observed above, $x\in
Q_0^i$ for some $i$. Hence by \eqref{dyadic2} in Definition
\ref{dl:dyadic} we must have $Q\in F_j(Q_0^i)$. Moreover
$$\sfd(y,Q)\le 3c_0^{-1}2^{-j_0-j}\overset{\eqref{eq:MeasCube}}{\le} C^{2/s} 3c_0^{-2} \diam(Q),$$
which shows \eqref{eq:good overlapping}.

{For later use, we also observe that up} to
enlarging the constant $K_0$ given above we can assume that $K_0\ge K$ and
\begin{equation}\label{eq:K_K0_incl}
\cup_{i=1}^m KQ_0^i\subset K_0Q_0.
\end{equation}

\bigskip

{ \emph{Step 3: decomposing the double integral
in $\nb_{p,\mathcal{V}}(2Q_0)$ using the partition from Step 2.}}

\noindent  If $\beta_{2p,\mathcal V}(K_0Q_0)=0$  then $\nb_{p,\mathcal V}(2Q_0)=0$ (recall \eqref{eq:easy beta iota inequality}) and there is nothing to prove. Hence we can assume $\beta_{2p,\mathcal V}(K_0Q_0)>0$ and choose $V\in \mathcal V$  such that  
\begin{equation}\label{eq:realizing plane V}
      \beta_{2p,V}(K_0Q_0)\le 2 \beta_{2p,\mathcal V}(K_0Q_0),
\end{equation}
which exists by the definition in \eqref{eq:beta_q_V_planes}.
From now on we will drop for convenience the subscript $\mathcal V$ and simply write $\beta_{2p}(\cdot),\iota_p(\cdot)$ instead of $\beta_{2p,\mathcal V}(\cdot ),\iota_{p,\mathcal V}(\cdot ).$

We aim to use $V$ to bound the quantity $\nb_{p}(2Q_0)$ and combine \eqref{eq:scales of distances}-\eqref{eq:good overlapping} to decompose the double integral  as follows:
    \begin{align}\label{eq:I_Q_term}
        &\mu(2Q_0)^2\nb_{p}^p(2Q_0)\le  \diam(2Q_0)^{-p} \int_{2Q_0}\int_{2Q_0}|\sfd (x,y)-\sfd(\pi_V(x),\pi_V(y))|^pd\mu(x) d \mu(y)\notag\\
        &\le   \diam(2Q_0)^{-p}   \sum_{j\ge 0}\int_{A(j_0,j)\cap (2Q_0\times 2 Q_0)}|\sfd (x,y)-\sfd(\pi_V(x),\pi_V(y))|^pd\mu(x) d \mu(y)\notag\\
        &\le   \diam(2Q_0)^{-p}     \sum_{i=1}^m\sum_{j\ge 0} \sum_{Q \in F_j(Q_0^i)}\int_{A(j_0,j)\cap (KQ\times KQ)} |\sfd(x,y)-\sfd(\pi_V(x),\pi_V(y))|^pd\mu(x) d \mu(y)\notag\\
        &= : \diam(2Q_0)^{-p}    \sum_{i=1}^m\sum_{j\ge 0} \sum_{Q \in F_j(Q_0^i)} \mathcal I_{Q}.
    \end{align}

    \bigskip

    { \emph{Step 4: estimating the summands $\mathcal{I}_Q$ from Step 3}}

\noindent The goal is now to estimate each $\mathcal I_Q$ separately.  From now on ${C_0}>0$ will denote a constant, the  value of which may  change from line to line, depending only on $p,C,s,C_P,C_T(K_0)$  (where {$C$ is the Ahlfors regularity of $E$,} $C_P$ is the constants appearing in Definition \ref{def:planes system} and $C_T(K_0)$ is the constant appearing in assumption \eqref{eq:tilting assumption} {for $\bar \lambda = K_0$}). We also fix a large constant $M>0$ to be determined later and
depending on the same parameters $p,C,s,C_P,C_T(K_0)$. We
stress that $M$ will be chosen \emph{depending on
{(the final choice of)} ${C_0}$}, hence we will
not be allowed in what follows to modify ${C_0}$ in terms of $M$.

{We will show that for all $i=1,\dots,m$, all
$j \in \mathbb N\cup \{0\}$, and all $Q\in F_j(Q_0^i)$, the
existence of a chain of cubes
     $Q=Q_j\subset Q_{j-1}\subset ...\subset Q_0^i $ with $Q_{h}\in F_h(Q_0^i)$ such that
     \begin{equation}\label{eq:I_Q estimate}
         \mathcal{I}_Q\le C_12^{-(j+j_0)(p+2s)} \left(\beta_{2p}(K_0 Q_j)+\beta_{2p}(K_0 Q_{j-1})+\beta_{2p}(K_0 Q_{j-2})+...+\beta_{2p}(K_0 Q_0)\right)^{2p},
     \end{equation}
     where $C_1$ is a constant depending only on  $p,C,s,C_P,C_T$.}


{Observe first that if $Q\in F_j(Q_0^i)$, then by
definition there exists at least one chain of cubes
     $Q=Q_j\subset Q_{j-1}\subset ...\subset Q_0^i $ with $Q_{h}\in F_h(Q_0^i)$.      Hence for each summand $\mathcal I_Q$ in \eqref{eq:I_Q_term}, that is, for each $Q \in F_j(Q_0^i)$,  we can distinguish two cases:}

     \medskip

    \textbf{Case 1:} \emph{For every chain of cubes $Q=Q_j\subset Q_{j-1}\subset ...\subset Q_0^i $, with $Q_{h}\in F_h(Q_0^i)$ it holds that
    \[
    M\big [\beta_{2p}( KQ)+\beta_{2p}(K Q_{j-1})+\beta_{2p}(K Q_{j-2})+...+\beta_{2p}(K Q_0^i)+\beta_{2p}( K_0 Q_0)\big]>\tfrac{1}{2}.
    \]}
    The presence of $\beta_{2p}( K_0 Q_0)$ might seem odd, but  it will be useful later on; {see \textbf{Case 2.b} below.}
In this case we have, since $\pi_V$ is 1-Lipschitz,
    \begin{align*}
        \mathcal{I}_Q\le &\int_{A(j_0,j)\cap (KQ\times KQ)} |\sfd(x,y)-\sfd(\pi_V(x),\pi_V(y))|^pd\mu(x) d \mu(y)\\
        &\le (3c_0^{-1})^p (2M)^{2p} 2^{-p(j+j_0)}\int_{K Q\times K Q} (\beta_{2p}(K Q)+...+\beta_{2p}(K Q_0^i)+\beta_{2p}( K_0 Q_0))^{2p}d\mu d\mu\\
        &\le \bar{C} (3c_0^{-1})^p (2M)^{2p} 2^{-(j+j_0)(p+2s)}\, (\beta_{2p}(K Q)+...+\beta_{2p}(K Q_0^i)+\beta_{2p}( K_0 Q_0))^{2p},
    \end{align*}
    for any chain of cubes  $Q=Q_j\subset Q_{j-1}\subset \ldots\subset Q_0^i $, with $Q_{h}\in F_h(Q_0^i)$, where we have used that by \eqref{eq:MeasCube} it holds $\mu(K Q)\lesssim_{s,C}  2^{-s(j+j_0)}$.

    \medskip

    \emph{\textbf{Case 2:} There exists a chain of cubes   $Q=Q_j\subset Q_{j-1}\subset \ldots \subset Q_0^i $, with $Q_{h}\in F_h(Q_0^i)$ satisfying
    \begin{equation}\label{eq:case2}
        M(\beta_{2p}(K Q)+\beta_{2p}(K Q_{j-1})+\beta_{2p}(K Q_{j-2})+\ldots+\beta_{2p}(K Q_0^i)+\beta_{2p}(
        K_0 Q_0))\le \tfrac{1}{2}.
    \end{equation}}
    In this case we  further consider pointwise each couple $(x,y)\in A(j_0,j)\cap (KQ\times KQ)$ (which is the domain of the integral $\mathcal{I}_Q$), where $A(j_0,j)$ was defined in \eqref{eq:A(j0,j)}.
      In particular, $\tfrac{c_0}{3}\sfd(x,y)\in (2^{-j_0-j-1},2^{-j_0-j})$. Based on $(x,y)$, we
    distinguish two subcases, in each of which we will obtain good control over the expression $|  \sfd(\pi_V(x),\pi_V(y))- \sfd(x,y)|^p$.
To do so, we fix some $V_Q\in \mathcal V$  such that 
\begin{equation}\label{eq:choice of VQ}
    \beta_{2p,V_Q}(K Q)\le \beta_{2p}(K Q)+\beta_{2p}(K_0Q_0).
\end{equation}
 The following subcases can arise:

    \textbf{Case 2.a:} \emph{ $\sfd(x,V_Q)+\sfd(y,V_Q)\ge \frac1{2C_P} \sfd(x,y),$ where $C_P$ is the constant in \eqref{eq:pit}. }

 \noindent Since $\pi_V$ is 1-Lipschitz,

    \begin{align*}|\sfd(x,y)&-\sfd(\pi_V(x),\pi_V(y))|^p\le \sfd(x,y)^p\le \sfd(x,y)^{p} (2C_P)^{2p}  \left(\frac{\sfd(x,V_Q)+\sfd(y,V_Q)}{\sfd(x,y)}\right)^{2p}\\
        &\le {C_0}\frac{\sfd(x,V_Q)^{2p}+\sfd(y,V_Q)^{2p}}{\sfd(x,y)^p} \le  {C_0}\frac{\sfd(x,V_Q)^{2p}+\sfd(y,V_Q)^{2p}}{2^{-p(j+j_0)}}.
    \end{align*}

    \textbf{Case 2.b: } \emph{$\sfd(x,V_Q)+\sfd(y,V_Q)< \frac1{2C_P} \sfd(x,y),$ where $C_P$ is the constant in \eqref{eq:pit}. }

\noindent  
For all $k=0,\dots,j-1$ we choose a plane $V_k\in\mathcal V$ such that 
$$\beta_{2p,V_k}(KQ_k)\le \beta_{2p}(KQ_k)+j^{-1}\beta_{2p}(K_0Q_0)$$
(recall that we are assuming $\beta_{2p}(K_0Q_0)>0$).
Iterating the tilting assumption \eqref{eq:tilting
assumption} first on all the chain $Q\subset Q_{j-1}\subset \dots
\subset Q_0^i$, choosing at each step the planes $V_k,V_{k-1}$, and finally on the inclusion $K Q_0^i\subset  K_0
Q_0$ {stated in \eqref{eq:K_K0_incl}} (recalling
that $\angle(.,.)$ satisfies \ref{it:triangle angle} in Definition
\ref{def:planes system})  we find 
    \[
    \angle (V_{Q},V)\le {C_0} (\beta_{2p}(K Q)+\beta_{2p}(K Q_{j-1})+...+\beta_{2p}(K Q_0^i)+\beta_{2p}( K_0 Q_0)).
    \]
    The above inequality is the reason we added $\beta_{2p}(K_0 Q_0)$ in all the above cases, since this allows us to compare $V_Q$ with a single plane $V$ independent of the cube $Q_0^i$ containing $Q.$ 

    Applying condition \eqref{eq:pit} from the definition of system of planes-projections-angle,
    \begin{align*}
        &\sfd(x,y)^2\le \sfd(\pi_V(x),\pi_V(y))^2+\sfd(x,y)^2\left (C_P\angle (V_{Q},V)+ C_P\frac{\sfd(x,V_Q)+\sfd(y,V_Q)}{\sfd(x,y)}\right )^2\\
        &\le \sfd(\pi_V(x),\pi_V(y))^2+\sfd(x,y)^2\left ({C_0}(\beta_{2p}(K Q)+...+\beta_{2p}( K_0  Q_0))+C_P \frac{\sfd(x,V_Q)+\sfd(y,V_Q)}{\sfd(x,y)}\right )^2,
    \end{align*}
    since ${C_0}$ is allowed to depend on $C_P.$
We would like to move the rightmost term to the left hand-side and take the square root on both sides, however we need to check non-negativity of the terms.
    This is easily verified since by \eqref{eq:case2}, which we are currently assuming,
    $${C_0}(\beta_{2p}(K Q)+...+\beta_{2p}(K_0Q_0))<1/2,$$
    provided $M$ is chosen so that $M\geq {C_0}$ and moreover by the assumption in \textbf{Case 2.b} it holds $C_P\frac{\sfd(x,V_Q)+\sfd(y,V_Q)}{\sfd(x,y)}<1/2$. Hence we can write
    \[
    \sfd(\pi_V(x),\pi_V(y))\ge \sfd(x,y)\sqrt{1-\left (C_0(\beta_{2p}(K Q)+...+\beta_{2p}( K_0 Q_0)))+C_P \frac{\sfd(x,V_Q)+\sfd(y,V_Q)}{\sfd(x,y)}\right )^2}.
    \]
    From this, using the inequality $\sqrt{1-t}\ge 1-t,$ valid for all $t\in [0,1]$, and using that $|   \sfd(\pi_V(x),\pi_V(y))- \sfd(x,y)|= \sfd(x,y)-\sfd(\pi_V(x),\pi_V(y))$ since $\pi_V$ is 1-Lipschitz, and raising to the $p$-th power,  we obtain
    \begin{align*}
        &|  \sfd(\pi_V(x),\pi_V(y))- \sfd(x,y)|^p\\
        &\le \sfd(x,y)^p \left({C_0} \big(\beta_{2p}(K Q)+...+\beta_{2p}( K_0 Q_0)\big)+C_P \frac{\sfd(x,V_Q)+\sfd(y,V_Q)}{\sfd(x,y)}\right )^{2p}\\
        &\le {C_0} 2^{-p(j+j_0)} \left(\beta_{2p}(K Q)+...+\beta_{2p}( K_0 Q_0)\right)^{2p}+ {C_0}\frac{\sfd(x,V_Q)^{2p}+\sfd(y,V_Q)^{2p}}{2^{-p(j+j_0)}}.
    \end{align*}
    Recall that we are assuming that  $\frac{c_0}3\sfd (x,y) \in (2^{-j_0-j-1},2^{-j_0-j})$ .

    Combining {\bf Case 2.a} and \textbf{Case 2.b} we obtain that for every  $x,y \in K  Q$ with $\frac{c_0}3\sfd (x,y) \in (2^{-j_0+j+1},2^{-j_0+j})$, with $Q$ as in \textbf{Case 2}, it holds
    \[
    \begin{split}
         &|  \sfd(\pi_V(x),\pi_V(y))- \sfd(x,y)|^p\\
         &\le {C_0} 2^{-p(j+j_0)} \left(\beta_{2p}(K Q)+...+\beta_{2p}( K_0 Q_0)\right)^{2p}+ {C_0}\frac{\sfd(x,V_Q)^{2p}+\sfd(y,V_Q)^{2p}}{2^{-p(j+j_0)}}.
    \end{split}
    \]
We can now use this estimate to bound $\mathcal I_Q$:
    \begin{align*}
        \mathcal{I}_Q\le &\int_{A(j_0,j)\cap (K Q\times K Q)} |\sfd(x,y)-\sfd(\pi_V(x),\pi_V(y)|^pd\mu(x) d \mu(y)\\
        &\le \mu(KQ)^2{C_0} 2^{-p(j+j_0)} \left(\beta_{2p}(K Q)+...+\beta_{2p}( K_0 Q_0)\right)^{2p}+\\
        &+{C_0} 2^{-p(j+j_0)} \int_{K Q\times K Q} \frac{\sfd(x,V_Q)^{2p}+\sfd(y,V_Q)^{2p}}{2^{-2p(j+j_0)}}d\mu(x) d \mu(y)\\
        &\le {C_0} 2^{-(j+j_0)(p+2s)}(\beta_{2p}(K Q)+...+\beta_{2p}( K_0 Q_0))^{2p}\\
        &+{C_0} 2^{-(j+j_0)(p+2s)} \frac{1}{\mu(Q)}\int_{K Q} \frac{\sfd(x,V_Q)^{2p}}{\diam(KQ)^{2p} }d\mu(x)\\
        &\le {C_0} 2^{-(j+j_0)(p+2s)} (\beta_{2p}(K Q)+\beta_{2p}(K Q_{j-1})+\beta_{2p}(K Q_{j-2})+...+\beta_{2p}(K_0 Q_0))^{2p},
    \end{align*}
where in the last step we used \eqref{eq:choice of VQ}.
We are now ready to put everything together.
{Recall that one between  \textbf{Case 1}  or
\textbf{Case 2} must be verified. Hence combining the 
estimates for $\mathcal I_Q$ in these two cases,}  and since
$\beta_{2p}(K Q)\le {C_0}\beta_{2p}( K_0 Q)$ for all $Q\in
\Delta$, we obtain the claimed inequality \eqref{eq:I_Q estimate}.

     {Note that we cannot put ${C_0}$ in \eqref{eq:I_Q estimate}} in place of $C_1$ since in \textbf{Case 1} the estimate depends on $M$, which is chosen after ${C_0}$; {recall \textbf{Case 2.b}}).

     From now on we also allow $C_1$ to vary from line to line, but depending on the same parameters.

     \bigskip

  { \emph{Step 5:  concluding the estimate with the bounds for $\mathcal{I}_Q$ from Step 4.}}

\noindent     Plugging \eqref{eq:I_Q estimate} in the initial sum \eqref{eq:I_Q_term} we can now write, recalling also  $\diam(2Q_0)\ge  C_1^{-1} 2^{-j_0}$,
    \begin{align*}
        & \mu(2Q_0)^2\nb_{p}^p(2Q_0)\le{C_1} \diam (2Q_0)^{-p}   \sum_{i=1}^m\sum_{j\ge 0} \sum_{Q \in F_j(Q_0^i)} \mathcal I_{Q}\\
        &\le {C_1}2^{-2sj_0} \sum_{i=1}^m\sum_{j\ge 0} 2^{-j(p+2s)} \sum_{Q \in F_j(Q_0^i)} (\beta_{2p}( K_0 Q_j)+\beta_{2p}(K_0 Q_{j-1})+...+\beta_{2p}( K_0 Q_0))^{2p}\\
        &\le  {C_1}2^{-2sj_0}  \sum_{i=1}^m\sum_{j\ge 0} j^{2p}  2^{-j(p+2s)} \sum_{Q \in F_j(Q_0^i)}   \beta_{2p}( K_0 Q_j)^{2p}+\beta_{2p}(K_0 Q_{j-1})^{2p}+...+\beta_{2p}( K_0 Q_0)^{2p}.
\end{align*}

\parbox{12cm}{{{\color{black}  \underline{Claim}}:}  for every
$i=1,\dots,m$, every $j \in \mathbb N\cup\{0\}$,  and all $l\in
\mathbb N\cup\{0\}$ with $l\le j$, each cube $\bar Q \in
F_l(Q_0^i)$  belongs to at most ${C_1}2^{s(j-l)}$ chains starting
from some $Q \in F_j(Q_0^i).$}

\noindent {\color{black}The claim is true} because from
\eqref{eq:number of descendants} there are at most $
{C_1}2^{s(j-l)}$ cubes $Q \in F_j(Q_0^i)$ so that $Q\subset \bar
Q$ (indeed in this case $Q\in F_{j-l}(\bar Q)$). 

This allows to
write the following estimate
\begin{equation}\label{eq:bohboh}
\begin{split}
      \sum_{Q \in F_j(Q_0^i)}   \beta_{2p}( K_0 Q_j)^{2p}&+\beta_{2p}(K_0 Q_{j-1})^{2p}+...+\beta_{2p}( K_0 Q_0)^{2p}\\
      &\le {C_1}\bigg (\sum_{0\le l\le j} 2^{s(j-l)} \sum_{Q \in F_l(Q_0^i)}  \beta_{2p}( K_0 Q)^{2p}\bigg )+2^{s\cdot j}\beta_{2p}( K_0 Q_0)^{2p}.
\end{split}
\end{equation}
{Plugging \eqref{eq:bohboh} in the previous
inequality and manipulating gives}
\begin{align*}
           &\mu(2Q_0)^2\nb_{p}^p(2Q_0)\\
           &\le  {C_1} 2^{-2sj_0}  \sum_{i=1}^m\sum_{j\ge 0} j^{2p}  2^{-j(p+2s)}  \left(\bigg (\sum_{0\le l\le j} 2^{s(j-l)} \sum_{Q \in F_l(Q_0^i)}  \beta_{2p}( K_0 Q)^{2p}\bigg )+2^{s\cdot j}\beta_{2p}( K_0 Q_0)^{2p}\right)\\
        &\le {C_1}  2^{-sj_0}\sum_{i=1}^m\sum_{j\ge 0}j^{2p} 2^{-j(p+s)}\left( \bigg(  \sum_{0\le l\le j}  2^{-s(l+j_0)} \sum_{Q \in F_l(Q_0^i)}  \beta_{2p}( K_0 Q)^{2p}\bigg)+ 2^{-sj_0}\beta_{2p}( K_0 Q_0)^{2p}\right) \\
        &\le  {C_1}2^{-sj_0} \sum_{i=1}^m  \sum_{j\ge 0}j^{2p}   2^{-j(p+s)}  \left( \bigg(\sum_{0\le l\le j} \sum_{Q \in F_l(Q_0^i)}  \mu(Q)\beta_{2p}( K_0 Q)^{2p}\bigg)+\mu(Q_0)\beta_{2p}( K_0 Q_0)^{2p} \right),
        \end{align*}
        having used that $\mu(Q)\ge C_1^{-1} 2^{-s(l+j_0)}$ for all $Q\in F_l(Q_0).$ Next we invert the summing order on the first term as follows:
\begin{equation}
   \sum_{j\ge 0}j^{2p}   2^{-j(p+s)}  \sum_{0\le l\le j}[\dots]_{l}=\sum_{0\le l}[\dots]_{l}  \sum_{j\ge l} j^{2p}2^{-j(p+s)},
\end{equation}
where $[\dots]_l\coloneqq \sum_{Q \in F_l(Q_0^i)}  \mu(Q)\beta_{2p}( K_0 Q)^{2p}.$ We also observe that for all $l\ge 0$ it holds $\sum_{j\ge l} j^{2p}2^{-j(p+s)}\le c_p 2^{-sl}$ for some constant $c_p>0$ depending only on $p.$ Therefore we obtain
        \begin{align*}
      &\mu(2Q_0)^2\nb_{p}^p(2Q_0)\\  
        &\le  {C_1}2^{-sj_0}\sum_{i=1}^m  \left(\big(\sum_{0\le l}2^{-sl}\sum_{Q \in F_l(Q_0^i)}   \mu(Q)\beta_{2p}( K_0 Q)^{2p}\big) + \mu(Q_0)\beta_{2p}( K_0 Q_0)^{2p}\right) ,\\
        &\le   {C_1}2^{-sj_0}\sum_{i=1}^m  \sum_{0\le l}2^{-sl}\sum_{Q \in F_l(Q_0^i)}   \mu(Q)\beta_{2p}( K_0 Q)^{2p}.
    \end{align*}
   In the last inequality we used that $Q_0=Q_0^1.$
Recalling that $2^{-sj_0}\le {C_1} \mu(Q_0)$ concludes the proof.
\end{proof}

\subsection{The specific case of the Euclidean space}\label{s:euclidean planes}
We check that the abstract results of the previous section are applicable in Euclidean spaces by considering the usual $d$-dimensional planes. Already in this setting this will lead to a non-trivial result (Corollary \ref{cor:char nb euclidean}), which will provide a characterization of uniform rectifiability via $\nb$-coefficients as stated in Theorem
\ref{t:Char_iota_Eucl} in the introduction.

For every $n,d \in \mathbb N$ with $d< n$ we set:
\[
\mathcal V_d(\rr^n)\coloneqq \{\text{$d$-dimensional affine planes in $\rr^n$}\}.
\]
We will mainly write only $\mathcal V_d$ when no confusion can occur.
For every $V \in \mathcal V_d$ we also denote by $\pi_V: \rr^d\to V$ the orthogonal projection onto $V.$
\begin{definition}[Angles between Euclidean
planes]\label{d:AngleEucl}
    We define  $\angle_e : \mathcal V_d\times \mathcal V_d \to [0,1]$ by
    \[
    \angle_e(V_1,V_2)\coloneqq d_H(\tilde V_1\cap B_1^{\R^{n}}(0),\tilde V_2\cap B_1^{\R^n}(0)),
    \]
 where $\tilde V_i$ is the $d$-dimensional plane parallel to $V_i$ and containing the origin, and $d_H$ denotes the Hausdorff distance.
\end{definition}

\begin{proposition}\label{prop:euclidean planes}
      Fix $n,d \in \mathbb N$ with $d< n$.
      Then the triple $(\mathcal V_d(\rr^n),\mathcal P,\angle_e)$,
      where $\mathcal P\coloneqq \{\pi_V\}_{V \in \mathcal V_d(\rr^n)}$,
      is a system of planes-projections-angle for $\rr^n$ endowed with the Euclidean distance.
\end{proposition}
\begin{proof}
    The function $\angle_e$ clearly satisfies item \ref{it:triangle angle} in Definition \ref{def:planes system}. Item \ref{it:pit} instead is proved in  Lemma \ref{prop:two planes pitagora} below.
\end{proof}

The following elementary  lemma in Euclidean geometry was needed in the proof of Proposition \ref{prop:euclidean planes}, but it will be also used in the Heisenberg setting in the next section.
  \begin{lemma}[Euclidean two-planes Pythagorean theorem]\label{prop:two planes pitagora}
      Let $n,d \in \mathbb N$ with $d< n$ and fix $V_1,V_2\in \mathcal V_d$. Then for any  $x,y \in \R^d$ it holds
    \begin{equation}\label{eq:two planes pythagora}
        |x-y|^2\le |\pi_{ V_2}(x)-\pi_{V_2}(y)|^2+\left (|x-y|\angle_e( V_1,  V_2)+d(y,V_1)+d(x,V_1)\right )^2,
    \end{equation}
\end{lemma}
\begin{proof}
    Set $\Pi\coloneqq \pi_{V_2}$ and  $\pi'\coloneqq \pi_{V_1}$.
    We can assume that $x \in V_1.$ Indeed suppose that we have proven this case. Then for arbitrary $x,y$ consider the points $\tilde x,\tilde y$ given by $\tilde x\coloneqq x+(\pi'(x)-x)\in V_1$ and $\tilde y\coloneqq y+(\pi'(x)-x)$. Then, since $|\tilde x-\tilde y|=|x-y|$ and $|\Pi(\tilde x)-\Pi(\tilde y)|=|\Pi(x)-\Pi(y)|$ we have
    \begin{equation*}
        |x-y|^2\le |\Pi(x)-\Pi(y)|^2+|x-y|^2\left (\angle_e(V_1,V_2)+\frac{d(\tilde y ,V_1)}{|x-y|}\right )^2.
    \end{equation*}
    However it is clear that $d(\tilde y ,V_1)\le d(y,V_1)+|\pi'(x)-x|=d(y,V_1)+d(x,V_1)$, which gives the statement in the general case.

    Hence suppose from now on that $x \in V_1.$
    Let $\alpha\coloneqq \angle_e(V_1,V_2)$. Up to  translating both the plane $V_1$ and the points $x,y$ by the vector $\Pi(x)-x$,
    we can suppose $x \in V_1\cap V_2$. {Finally, up to further translating $V_1,V_2,x,$ and $y$ by the vector $-x$, we can assume that $x=0$}. Let now $p$ be the orthogonal projection of $y$ onto $V_1$. Since both $V_1$ and $V_2$  contain the origin, we have that
    \[d(p,V_2)\le d_H(V_2\cap B_{|p|}(0),V_1\cap B_{|p|}(0))\le |p|\alpha\le |y|\alpha.\]
    Therefore $d(y,V_2)\le d(y,p)+d(p,V_2)\le d(y,V_1)+|y|\alpha.$
    Then by Pythagoras' theorem
    \[ |y|^2=|\Pi(y)-y|^2+|\Pi(y)|^2=d(y,V_2)^2+|\Pi(y)|^2\le(d(y,V_1)+|y|\alpha)^2+|\Pi(y)-\Pi(x)|^2, \]
    since $\Pi(x)=0$. As $x=0$ and $\sfd(x,V_1)=0,$ this is exactly \eqref{eq:two planes pythagora} and  the proof is concluded.
\end{proof}

Before turning our attention to the Euclidean tilting estimate, we state another auxiliary lemma.

\begin{lemma}[{Existence of independent points, \cite[Lemma 5.8]{MR1113517}}]\label{lem:indep points}
    Let $E\in \reg_d(C)$ be a $d$-regular subset of $\mathbb R^n$, where $d\in \mathbb N$ and $d< n$ and let $\Delta$ be a system of dyadic cubes for $E.$ Then for every $Q\in \Delta$ there exist points $x_0,\dots,x_d\in Q$ such that  $\sfd(x_i,P_{i-1})\ge A^{-1}\diam(Q)$ for all $i=1,\dots,d$, where $P_j$ is the $j$-dimensional plane spanned by the points $x_0,\dots,x_j$ and where $A>0$ is a constant depending only on $C$ and $d.$
\end{lemma}

\begin{proposition}[Euclidean tilting estimate]\label{p:EuclTilt}
  Let $E\in \reg_d(C)$ be a $d$-regular subset of $\mathbb R^n$, where $d\in \mathbb N$ and $d<n$ and let $\Delta$ be a system of dyadic cubes for $E.$ Then for every $Q_1 \in \Delta_j, Q_0\in \Delta_{j-1}\cup \Delta_j$, for some $j\in \mathbb J,$ and all constants $\lambda_0,\lambda_1\ge 1$ satisfying $\lambda_1Q_1\subset \lambda_0Q_0$, it holds
         \begin{equation}\label{eq:tilting assumptionEucl}
             \angle_e(V_1,V_0)\le D \lambda_0^{{d+1}}( \beta_{p,V_1}(\lambda_1Q_1)+\beta_{p,{\color{black}V_0}}(\lambda_0Q_0)),\quad  p\in[1,\infty),   
         \end{equation}
         for any choice of $V_i\in \mathcal V_d$, $i=0,1,$ and  where $D$ is a constant depending only on $C$ and $d$.
\end{proposition}
\begin{proof}
It enough to show the case $p=1$, as the case $p\neq 1$ then follows from the H\"older inequality.
    By Lemma \ref{lem:indep points} and by $d$-regularity we can find points $x_0,\dots,x_d\in Q_1$ as in Lemma \ref{lem:indep points}  and also such that
    \[
    \sfd(x_i,V_0)+\sfd(x_i,V_1)\le D \left(\beta_{1,{\color{black}V_1}}(\lambda_1Q_1)+\beta_{1,{\color{black}V_0}}(\lambda_0Q_0)\right),\quad i=0,\ldots,d,
    \]
    where  $D$ is a constant depending only on $C$ and $d.$ In other words the planes $V_1$ and $V_0$ are both quantitatively close to the same set of independent points. From this \eqref{eq:tilting assumptionEucl} easily follows  (see e.g.\ \cite[Lemma 5.13]{MR1113517} {or the argument in the proof of Proposition \ref{p:HeisTilt}}).
\end{proof}

The above results show that the system of planes-projections-angle in the Euclidean space satisfies the hypotheses of the abstract  Theorem \ref{thm:axiomatic}  for all $p \in[1,\infty).$
 Therefore specializing its statement to the Euclidean setting we obtain the following.
\begin{thm}\label{thm:beta implies alpha euclidean}
     Let $E\in \reg_d(C)$ be a $d$-regular subset of $\mathbb R^n$, where $d\in \mathbb N$ and $d<n$.
     Then {\color{black}for all $p\in[1,\infty)$ and all $q\in [p,\infty)$ (and also all $q\in [1,\infty)$ if $p<d$)}  it holds:
     \[
     E \in {\rm GLem}(\beta_{2p,\mathcal V_d},2q,M) \implies  E \in {\rm GLem}(\nb_{p,\mathcal V_d},q,\hat C M),
     \]
     {where $\hat C$ can be chosen depending only on $d$, $C$, $p$, $q$.}
\end{thm}

\subsubsection{Converse inequalities in the Euclidean space}

In the special case of the Euclidean space we can also get a
converse of Theorem \ref{thm:beta implies alpha euclidean},
{yielding Corollary \ref{cor:char nb euclidean}.
This is thanks to an upper bound for squared
$\beta_{q,\mathcal{V}_d}$-numbers in terms of $\nb_{q,d,{\sf
Eucl}}$- and  $\nb_{q,\mathcal V_d}$-numbers. Recall that the
latter are given as in Definition \ref{d:nb_V} applied to
$\mathcal{V}=\mathcal{V}_d$. The $\nb_{q,d,{\sf Eucl}}$-numbers,
on the other hand, were studied in \cite{CarlesonPart1} and are
defined for  $k\in \mathbb{N}$, a closed set $E \subset
\mathbb{R}^n$ of locally finite $\mathcal{H}^k$-measure, and
$\mu:=\mathcal{H}^k\lfloor_E$, as
\begin{equation}\label{eq:nbEucl}
\nb_{q,k,{\sf Eucl}}(S):=\inf_{f:S\to \mathbb{R}^k}\left(
\frac{1}{\mu(S)^{2}}\int_{S}\int_S\left[\frac{\left||x-y|-|f(x)-f(y)|\right|}{\mathrm{diam}(S)}\right]^q\,d\mu(x)\,d\mu(y)\right)^{1/q},
\end{equation}
for $S\in \mathcal{D}_s(E)$, where the functions $f$ are assumed to be Borel.
For later use, we also recall in the same setting the definition
\begin{equation}\label{eq:nb}
\nb_{q,k}(S):=\inf_{\|\cdot\|\text{ norm on
}\mathbb{R}^k}\inf_{f:S\to \mathbb{R}^k}\left(
\frac{1}{\mu(S)^{2}}\int_{S}\int_S\left[\frac{\left||x-y|-\|f(x)-f(y)\|\right|}{\mathrm{diam}(S)}\right]^q\,d\mu(x)\,d\mu(y)\right)^{1/q},
\end{equation}
 where the functions $f$ in the second infimum are again assumed to be Borel;
recall the formula below \eqref{eq:nb_intro_metric} or \cite[Definition 2.31]{CarlesonPart1}. Thus, the difference between the definitions in \eqref{eq:nbEucl} and \eqref{eq:nb} is whether the distance of $f(x)$ and $f(y)$ is measured in the usual Euclidean norm, or whether all possible norms in $\mathbb{R}^k$ are considered. We obtain the following bound in terms of the coefficients from \eqref{eq:nbEucl}:
}
\begin{proposition}\label{prop:converse inequalities}
     Let $E\in \reg_d(C)$ be a $d$-regular subset of $\mathbb R^n$, where $d\in \mathbb N$ and $d<n$ and let $\Delta$ be a system of dyadic cubes for $E.$ Then for all $q\in [1,\infty)$ and all $Q\in \Delta$ it holds
     \begin{equation}\label{eq:converse inequality}
         \beta_{q,\mathcal V_d}^2(2Q)\le \beta_{2q,\mathcal V_d}^2(2Q)\le  \tilde C {\nb_{q,d,{\sf Eucl}}(2Q)}\le \tilde C \nb_{q,\mathcal V_d}(2Q),
     \end{equation}
     where $\tilde C$ is a constant depending only on $d,q$ and $C$.
\end{proposition}
{ We begin with a lemma that will be used in the
proof of Proposition \ref{prop:converse inequalities}. First we
fix some notations.}  Set $\mu\coloneqq \mathcal H^d|_E.$  
 Let $\eps>0$ be arbitrary and to be fixed until the very end of the proof. Set  $\nb_{\eps,q}(2Q)\coloneqq \nb_{q,d,{\sf Eucl}}(2Q)+2\eps$.
    Fix a {Borel}  map $f:2Q\to \mathbb R^d$ such that
    \begin{equation}\label{eq:realizing f}
        \frac{1}{\mu(2Q)^{2}}\int_{2Q}\int_{2Q}\left[\frac{\left||x-y|-|f(x)-f(y)|\right|}{\mathrm{diam}(2Q)}\right]^q\,d\mu(x)\,d\mu(y)\le 2\nb_{\eps,q}(2Q)^q,
    \end{equation}
    which exists by definition. 
    We fix also a large constant $D\ge2$ to be chosen later and depending only on $C$ and $d$.

\begin{lemma}\label{lem:z0zdpoints}
    There exist points $z_0,\dots, z_d\in 2Q$ satisfying:
    \begin{enumerate}[label=$\roman*)$]
        \item $\int_{2Q}\left[\frac{\left||z_i-y|-|f(z_i)-f(y)|\right|}{\mathrm{diam}(2Q)}\right]^q\,d\mu(y)\le D\nb_{\eps,q}(2Q)^q\mu(2Q),$ for all $i=0,\dots,d,$
        \item $\left[\frac{\left||z_i-z_j|-|f(z_i)-f(z_j)|\right|}{\mathrm{diam}(2Q)}\right]^q\le D\nb_{\eps,q}(2Q)^q$, for all $i,j \in \{0,\dots,d\},$
        \item $\mathrm{Vol}_d(\{z_0,\dots,z_d\})\ge D^{-1} \diam(2Q)^d,$ where $\mathrm{Vol}_d(\{z_0,\dots,z_d\})$ denotes the $\mathcal H^d$-measure of the $d$-dimensional simplex with vertices $z_0,\dots,z_d.$
    \end{enumerate}
\end{lemma}
\begin{proof}
    Define the  sets
    \[
    \begin{split}
            &A\coloneqq \left\{z\in 2Q : \ \fint_{2Q}\left[\frac{\left||z-y|-|f(z)-f(y)|\right|}{\mathrm{diam}(2Q)}\right]^q\,d\mu(y)\le D\nb_{\eps,q}(2Q)^q \right\}\subset 2Q,\\
             &B\coloneqq \left \{(x,y) \in 2Q\times 2Q \ : \ \left[\frac{\left||x-y|-|f(x)-f(y)|\right|}{\mathrm{diam}(2Q)}\right]^q\leq D \nb_{\eps,q}(2Q)^q\right \}\subset 2Q\times 2Q.
    \end{split}
    \]
    By \eqref{eq:realizing f}, and applying the Markov inequality {to the corresponding (Borel) functions},
    \begin{equation}
         \mu(2Q\setminus A)\le \frac{2\mu(2Q)}{D}, \quad  \mu\otimes \mu((2Q\times 2Q)\setminus B)\le \frac{2\mu(2Q)^2}{D}.
    \end{equation}
    Combining these we get
    \begin{equation}\label{eq:markov}
        \begin{split}
         &{\mu\otimes \cdots\otimes \mu}(\{z_0,\dots,z_d \in (2Q)^{d+1} \ : \ \text{$i)$ does not hold}  \})\le (d+1) \frac{2\mu(2Q)^{d+1}}{D},\\
         &{\mu\otimes \cdots\otimes \mu}(\{z_0,\dots,z_d \in (2Q)^{d+1} \ : \ \text{$ii)$ does not hold}  \})\le \frac12(d+1)d \frac{2\mu(2Q)^{d+1}}{D}
    \end{split}
    \end{equation}
    Next, by Lemma \ref{lem:indep points}, there exist \emph{independent} points $x_0,\dots,x_d\in Q$  such that  $\sfd(x_i,P_{i-1})\ge A_0^{-1}\diam(Q)$ for all $i=1,\dots,d$, where $P_j$ is the $j$-dimensional plane spanned by the points $x_0,\dots,x_j$ and where $A_0>0$ is a constant depending only on $C$ and $d$. Up to increasing the constant $A_0$, with the same dependency, every choice of points $z_i \in B_{c\diam(Q)}(x_i)$ satisfies the same property provided $c\in(0,1)$ is a small enough constant depending again only on $d$ and $C.$ In particular, choosing $D>0$ large enough, we have that every choice of points $z_i \in B_{c\diam(Q)}(x_i)\cap 2Q$ satisfies $iii)$ above. By the $d$-regularity of $E$ we have that $$\mathcal H^d(B_{c\diam(Q)}(x_i)\cap E)\ge C^{-1} (c\diam(Q))^d \overset{\eqref{eq:MeasCube}}{\ge } \tilde c\mu(2Q), \quad i=0,\dots,d,$$
    where $\tilde c$ is a constant depending only on $C$ and $d$.
    Moreover $B_{c\diam(Q)}(x_i)\cap E\subset 2Q.$ This shows that
    \[
 {\mu\otimes \cdots\otimes \mu}(\{z_0,\dots,z_d \in (2Q)^{d+1} \ : \ \text{$iii)$  holds}  \})\ge  (\tilde c\mu(2Q))^{d+1}.
    \]
 This combined with \eqref{eq:markov}, if $D$ is large enough, proves the existence of $z_0,\dots,z_d\in 2Q$ satisfying all three of $i)$, $ii)$ and $iii)$ at the same time.
\end{proof}

\begin{proof}[Proof of Proposition \ref{prop:converse inequalities}]
{\color{black} The proof is based on estimates involving volumes of simplexes. Similar ideas to estimate $L^2$-approximation by planes appeared previously in \cite{LW09,LW11}.}
    The first and third inequality \eqref{eq:converse inequality} are obvious from the definitions, hence we only need to prove the second inequality in \eqref{eq:converse inequality}. We can also assume that 
    { $ \nb_{\eps,q}(2Q)\le 1$,}
    otherwise there is nothing to prove. Fix $z_0,\dots,z_d\in 2Q$ as given in Lemma \ref{lem:z0zdpoints}. Note that $ii)$ implies that
 \begin{equation}\label{eq:f(zi)-f(zj)}
 |f(z_i)-f(z_j)|\le (D+1)\diam(2Q).
 \end{equation}
Similarly $i)$ implies that
\begin{equation}\label{eq:F set bound}
    \mu(F)\le (d+1)\nb_{\eps,q}(2Q)^q\mu(2Q).
\end{equation}
where $F\coloneqq \{z\in 2Q \ : \  |f(z)-f(z_i)|\ge (D+1) \diam(2Q) \text{ for some $i=0,\dots,d$}\}$. Denoting by $V$ the $d$-dimensional plane spanned by $z_0,\dots,z_d$, we have that
    \begin{equation}\label{eq:bound volume distance}
          \sfd(z,V)= \frac{(d+1)\mathrm{Vol}_{d+1}(\{z_0,\dots,z_d,z\})}{\mathrm{Vol}_d(\{z_0,\dots,z_d\})}\overset{iii)}{\le }\frac{D(d+1)\mathrm{Vol}_{d+1}(\{z_0,\dots,z_d,z\})}{  \diam(2Q)^{d}}  , \quad  z \in 2Q.
    \end{equation}
    Moreover it is well known that  for all $n,d \in \N$ and all $y_0,\dots,y_{d+1}\in \rr^n$  it holds \sloppy $\mathrm{Vol}_{d+1}(y_0,\dots,y_{d+1})^2=\mathcal F_d(\{|y_i-y_j|^2\}_{0\le i<j\le d+1})$ for some locally Lipschitz function $\mathcal F_d: \rr^{\frac{(d+2)(d+1)}{2}}\to \rr$  independent of $n$ (see e.g.\ \cite[$\S$ 40]{MR0268781} or \cite[Section 4]{MR4489627}). From the scaling property of the volume it clearly holds
    \begin{equation}\label{eq:scaling volume}
         (t^{d+1}\mathrm{Vol}_{d+1}(y_0,\dots,y_{d+1}))^2=\mathcal F_d(\{t^2|y_i-y_j|^2\}_{0\le i<j\le d+1}).
    \end{equation}
    { Moreover by \eqref{eq:f(zi)-f(zj)}  and by definition of $F$ we have that
    \[
    t|z_i-z_j|,\, t|z_i-z|,\, t|f(z_i)-f(z_j)|,\,t|f(z_i)-f(z)|\le D+1, \quad \forall i,j=0\dots,d,\,\, z \in 2Q\setminus F,
    \]
    where $t\coloneqq \diam(2Q)^{-1}.$ This  combined with \eqref{eq:scaling volume}  gives
    \begin{equation}\label{eq:lip volume}
    \begin{split}
           &\left(\frac{\mathrm{Vol}_{d+1}(\{z_0,\dots,z_d,z\}}{\diam(2Q)^{d+1}}\right )^2-\left( \frac{\mathrm{Vol}_{d+1}(\{f(z_0),\dots,f(z_d),f(z)\}}{\diam(2Q)^{d+1}}\right )^2\\
           &\quad \le \frac{L(d,D)}{\diam(2Q)^2}\left(\sup_{0\le i < j\le d} ||z_i-z_j|^2-|f(z_i)-f(z_j)|^2|
    +\sup_{0\le i \le d} ||z_i-z|^2-|f(z_i)-f(z)|^2|\right),
    \end{split}
    \end{equation}
    where $L(d,D)$ is the Lipschitz constant of $\mathcal F_d$ restricted to ball of radius $(D+1)^2$ centered at the origin in $\rr^{\frac{(d+2)(d+1)}{2}}$ (with respect to the $\sup$-norm).
    }
   On the other hand $\mathrm{Vol}_{d+1}(\{f(z_0),\dots,f(z_d),f(z)\})=0$, because as $f$ maps into $\mathbb{R}^d$. Hence  for all $z\in 2Q\setminus F$  we have that
    \begin{align*}
    &\frac{\sfd(z,V)^2}{\diam(2Q)^{2}}\overset{\eqref{eq:bound volume distance}}{\le}  \left(D \frac{(d+1)\mathrm{Vol}_{d+1}(\{z_0,\dots,z_d,z\}}{\diam(2Q)^{d+1}}\right )^2&\\
    &\overset{\eqref{eq:lip volume}}{\le }\frac{c_{D,d}}{\diam(2Q)^{2}}\left(\sup_{0\le i < j\le d} ||z_i-z_j|^2-|f(z_i)-f(z_j)|^2|
    +\sup_{0\le i \le d} ||z_i-z|^2-|f(z_i)-f(z)|^2|\right)\\
    &\overset{\eqref{eq:f(zi)-f(zj)}}{\le }\frac{2c_{D,d}(D+2)}{\diam(2Q)}\left(\sup_{0\le i < j\le d} ||z_i-z_j|-|f(z_i)-f(z_j)||
    +\sup_{0\le i \le d} ||z_i-z|-|f(z_i)-f(z)||\right),
    \end{align*}
  where $c_{D,d}$ is a constant depending only on $D$ and $d.$
 {  Since the points $z_0,\dots,z_d$ satisfy $ii)$ in Lemma \ref{lem:z0zdpoints} we obtain
    \begin{align*}
     \frac{\sfd(z,V)^2}{\diam(2Q)^{2}}\le \frac{2c_{D,d}(D+2)}{\diam(2Q)}\left(D^\frac1q \nb_{\eps,q}(2Q)\diam(2Q)+\sum_{i=0}^d ||z_i-z|-|f(z_i)-f(z)||\right).
    \end{align*}}
    Raising to the $q$-th power both sides and integrating we obtain
    \begin{align*}
        &\frac{1}{\mu(2Q)}\int_{2Q\setminus F} \left[\frac{\sfd(x,V)}{\diam(2Q)}\right]^{2q} d\mu(x)\\
        &\le
        C(q,D,d)\left(\nb_{\eps,q}(2Q)^q+\sum_{i=0}^d  \fint_{2Q}\left[\frac{\left||z_i-y|-|f(z_i)-f(y)|\right|}{\mathrm{diam}(2Q)}\right]^q\,d\mu(y) \right)
    \end{align*}
    where $C(q,D,d)$ is a constant depending only on $q,d$ and $D$ and so ultimately only on $C,d,q.$
   { Finally, since the points $z_0,\dots,z_d$ satisfy also $i)$ in Lemma \ref{lem:z0zdpoints}, we obtain}
    \[
    \frac{1}{\mu(2Q)}\int_{2Q\setminus F} \left[\frac{\sfd(x,V)}{\diam(2Q)}\right]^{2q} d\mu(x) \le   C(q,D,d)  \nb_{\eps,q}(2Q)^q,
    \]
    up to increasing the constant $C(q,D,d).$
    This combined with \eqref{eq:F set bound} and sending $\eps\to 0^+$ completes the proof of the second of \eqref{eq:converse inequality} and of the proposition.
\end{proof}

Combining Proposition \ref{prop:converse inequalities} with Theorem \ref{thm:beta implies alpha euclidean} and the classical characterization of uniform $d$-rectifiability in $\mathbb{R}^n$ via $\beta_{p,\mathcal V_d}$-numbers \cite{MR1113517},
we  obtain the following characterization of uniform rectifiability in the Euclidean setting using the abstract $\nb$-numbers.
\begin{cor}\label{cor:char nb euclidean}
    Let $E\in \reg_d(C)$ be a $d$-regular subset of $\mathbb R^n$, where $d\in \mathbb N$ and $d<n$. Then
     \[
     E \in {\rm GLem}(\beta_{2,\mathcal V_d},2) \iff  E \in {\rm GLem}(\nb_{1,\mathcal V_d},1),
     \]
     and if any of the two holds then $E$ is uniformly rectifiable. {Moreover, these equivalences are quantitative: the constants involved in the definition of the geometric lemmas and the uniform rectifiability conditions can be chosen depending only on each other and on $d$ and $C$.}
\end{cor}

{We recall that  the conditions ${\rm GLem}(\beta_{q,\mathcal V_d},2)$ for $q< \frac{2d}{d-2}$ are known to be all equivalent to each other. The exponent $q=2$ is the largest one that falls in this range for any choice of $d\in \mathbb{N}$. }

\subsubsection{Vanishing $\nb$-numbers}\label{sss:VanishingNB}

We do not know if it possible to replace in Proposition \ref{prop:converse inequalities} the number $ {\nb_{q,d,{\sf Eucl}}(2Q)}$ with (the smaller) $ {\nb_{q,d}(2Q)}$, where $\nb_{q,k}(\cdot)$ is
defined as in \eqref{eq:nb}. 
However we  are able to show the weaker implication:
$$ {\nb_{q,d}(2Q)}=0\implies \beta_{q,\mathcal{V}_k}(Q)=0.$$
This is the content of the next proposition.


\begin{proposition}\label{prop:zero banach}
    Let $E\in \reg_k(C)$ be a $k$-regular subset of $\mathbb R^n$, where $k\in \mathbb N$ and $k\le n$ and let $\Delta$ be a system of dyadic cubes for $E.$ Suppose that for some $Q\in \Delta$ and $q\in [1,\infty)$  it holds
\[
\nb_{q,k}(Q)=0.
\]
Then $\nb_{q,k,{\sf Eucl}}(Q)=\beta_{q,\mathcal V_k}(Q)=0$ (where $\mathcal V_k$ is the family of $k$-dimensional affine planes in $\rr^n$). In particular up to a $\mathcal{H}^k$-zero measure set, $Q$ is contained in a $k$-dimensional plane.
\end{proposition}

 For the proof {Proposition \ref{prop:zero banach}} we will need the following technical result (proved below).

\begin{lemma}\label{lem:isometry positive measure}
    Let $(\X,\|\cdot\|_{\X})$ be an $N$-dimensional Banach space, $N\in \mathbb N$,    and $(Y,\|\cdot\|_{Y})$ be a strictly convex Banach space. Let also $A\subset \X$ be such that $\mathcal H^N(A)>0$ ($\mathcal H^N$ being the $N$-dimensional Hausdorff measure on $X$ with respect to $\sfd(x,y)\coloneqq \|x-y\|_{\X}$) and
    $f: A\subset \X \to Y$ be a map satisfying
    \begin{equation}\label{eq:isometry banach}
        \|f(x)-f(y)\|_{Y}= \|x-y\|_{\X}, \quad  x,y\in X.
    \end{equation}
    Then there exists a linear isometry $F:(\X,\|\cdot\|_{\X})\to (Y,\|\cdot\|_Y)$.
\end{lemma}

{This can be seen as a measure-theoretic substitute for the well-known fact that an isometric embedding of a real normed vector space into another one is necessarily affine if the target space is assumed to be strictly convex \cite{MR287283}.}

\begin{proof}[Proof of Proposition \ref{prop:zero banach} using Lemma \ref{lem:isometry positive measure}]
It suffices to show that $\nb_{q,k,{\sf Eucl}}(Q)=0$, then by Proposition \ref{prop:converse inequalities} (the statement is for $2Q$, but essentially the same proof works also for $Q$)  this will imply that $\beta_{q,\mathcal V_k}(Q)=0$, where $\mathcal V_k$ is the family of $k$-dimensional affine planes in $\rr^d.$

    By assumption there exists a sequence of norms $\|\cdot \|_i$, $i \in \mathbb N$, in $\rr^k$ and maps $f_i: Q \to \rr^n$ such that
    \begin{equation}\label{eq:l1 to 0}
          \int_{Q}\int_Q\left||x-y|-\|f_i(x)-f_i(y)\|_i\right|^q\,d\mathcal{H}^k(x)\,d\mathcal{H}^k(y) \to 0,
    \end{equation}
    where $|\cdot |$ denotes the Euclidean norm.  By \eqref{eq:l1 to 0} and since by definition $Q$ is bounded and $\mathcal H^k(Q)>0$, up to passing to a subsequence, there exists a set $C\subset Q\times Q$ independent of $i$, with $\mathcal{H}^k\otimes \mathcal{H}^k(C)>0$ and such that $\|f_i(x)-f_i(y)\|_i\le 2\diam(Q)$ for all $i$ and for all $(x,y) \in  C.$
    Up to passing to a subsequence, we can also assume that the functions $|x-y|-\|f_i(x)-f_i(y)\|_i$ convergence pointwise $\mathcal H^k\otimes \mathcal H^k$-a.e.\ to $0$ in $Q\times Q.$ Then by Egorov's theorem we can find a compact set $K\subset C$ of positive $\mathcal H^k\otimes \mathcal H^k$-measure where the convergence is uniform. Moreover  by  compactness (see e.g.\ \cite[pag.\ 278]{MR993774}), again up to a subsequence, the norms $\|\cdot \|_i$ converge to a limit norm $\|\cdot \|$ in the Banach-Mazur distance. In particular there exists a sequence of linear maps $T_i : (\mathbb R^k, \|\cdot \|_i) \to (\mathbb R^k,\|\cdot \|)$ that are $(1+\eps_i$)-biLipschitz for some $\eps_i\to 0.$
    Define then the maps $F_i: K \to \rr^k\times \rr^k$ by $F_i(x,y)\coloneqq (T_i\circ f_i(x),T_i\circ f_i(y)).$ Since $K\subset C$ and by how we chose $C$ it holds
    $$|\|T_i(f_i(x))-T_i(f_i(y))\|-\|f_i(x)-f_i(y)\|_i|\le 2\eps_i\diam(Q), \quad  x,y \in K.$$
    From this and the uniform convergence we have
    \[
     |\|F_i(x_1,y_1)-F_i(x_2,y_2)\|_{\mathrm{prod}}- |(x_1,y_1)-(x_2,y_2)|| \le \delta_i, \quad x_i,y_i \in K, \, i=1,2.
    \]
   for some $\delta_i \to 0,$ where we define the product norm by $\|(\cdot,\cdot) \|_{\mathrm{prod}}\coloneqq \sqrt{\|\cdot\|^2+\|\cdot\|^2}.$
    Then by the generalized Ascoli-Arzel\'a convergence theorem (see e.g.\ \cite[Prop.\ 3.3.1]{MR2401600}), up to a further subsequence, there exists $ F: (K,|\cdot|)\to (\rr^k\times \rr^k,\|\cdot\|_{\mathrm{prod}})$ such that $F_i\to F$ uniformly in $K$. Thus we must have
    \[
    \|F(x_1,y_1)-F(x_2,y_2)\|_{\mathrm{prod}}=|(x_1,y_1)-(x_2,y_2)|, \quad  x_i,y_i \in K, \, i=1,2.
    \]
    In particular $\mathcal H^{2k}(F(K))>0$ (where $\mathcal H^{2k}$ denotes the $2k$-dimensional Hausdorff measure in $(\rr^k\times \rr^k,\|\cdot\|_{\mathrm{prod}})$ and applying  Lemma \ref{lem:isometry positive measure} to $F^{-1}$ and with $A=F(K)$, we deduce that there exists a linear isometry $T:(\mathbb R^k\times \rr^k,\|\cdot \|_{\mathrm{prod}})\to (\rr^{2n},|\cdot |).$ Therefore (restricting $T$ to $\rr^k$) we conclude that there exists a linear isometry $\tilde T: (\rr^k,\|\cdot \|) \to (\rr^k,|\cdot |).$ Next we define the maps $\tilde f_i: (Q,|\cdot|) \to (\rr^k,|\cdot|)$ by $\tilde f_i\coloneqq \tilde T\circ T_i \circ f_i$. Set now $$B_i\coloneqq \{(x,y)\in Q\times Q \ : \ \|f_i(x)-f_i(y)\|_i\le 4\diam(Q)  \}$$
    and note that by \eqref{eq:l1 to 0} it holds $\mathcal H^k\otimes\mathcal H^k((Q\times Q)\setminus B_i)\to 0$ as $i\to \infty$. Moreover for all $(x,y) \in B_i$ it holds
    \begin{align*}
        &||\tilde f_i(x)-\tilde f_i(y)|-\|f_i(x)-f_i(y)\|_i |= |\|T_i \circ f_i(x)-T_i \circ f_i(y)\|-\|f_i(x)-f_i(y)\|_i |\\
        &\le \|f_i(x)-f_i(y)\|_i \left|\frac{\|T_i \circ f_i(x)-T_i \circ f_i(y)\|}{\|f_i(x)-f_i(y)\|_i}-1 \right|\le \eps_i \|f_i(x)-f_i(y)\|_i\le 4\diam(Q)\eps_i.
    \end{align*}
    On the other hand, for $(x,y)\in (Q\times Q)\setminus B_i$, we have $|\tilde f_i(x)-\tilde f_i(y)|\ge 2\diam(Q)$ as long as $1+\eps_i<2,$ hence
    \begin{align*}
        &||x-y|-|\tilde f_i(x)-\tilde f_i(y)||= |\tilde f_i(x)-\tilde f_i(y)|-|x-y|\le|\tilde f_i(x)-\tilde f_i(y)|\\
        &\le  (1+\eps_i) \|f_i(x)-  f_i(y)\|_i \le  2(1+\eps_i)| \|f_i(x)-  f_i(y)\|_i-|x-y||,
    \end{align*}
    where we used in the last step that $\|f_i(x)-f_i(y)\|_i\ge 4\diam(Q)$. Combining these two estimates with \eqref{eq:l1 to 0} we obtain
    \begin{align*}
          \int_{Q}\int_Q ||x-y|-|\tilde f_i(x)-\tilde f_i(y)||^q\,d\mathcal{H}^k(x)\,d\mathcal{H}^k(y)\to 0\quad \text{as }i\to \infty.
    \end{align*}
   This shows that $\nb_{q,k,{\sf Eucl}}(Q)=0$.
\end{proof}

We now prove the technical lemma used above.
\begin{proof}[Proof of Lemma \ref{lem:isometry positive measure}]
{In the case $A=\bar B_1(0)\subset \X$, the statement is known. Indeed, up to redefining $f$ as $f(\cdot)-f(0)$ we can assume that $f(0)=0$, and then the conclusion follows from \cite[Theorem 3.4]{MR3491489}, see also \cite[Remark 2.11]{MR3491489}.}

We pass now to the general case of $A\subset X$ with $\mathcal H^N(A)>0. $ Let $x\in \X$ be a one-density point for $A$ with respect to $\mathcal H^N$. Up to a translation we can assume that $x=0.$ Fix a sequence $r_n\to 0$. Setting $A_n\coloneqq  r_n^{-1}(A\cap B_{r_n}(0))\subset B_1(0)$ it holds
\begin{equation}\label{eq:zero density An}
    \frac{\mathcal H^N(B_{1}(0)\setminus A_n)}{\mathcal H^N(B_{1}(0))}=\frac{\mathcal H^N(B_{r_n}(0)\setminus A)}{r_n^N\mathcal H^N(B_{1}(0))}
=\frac{\mathcal H^N(B_{r_n}(0)\setminus A)}{\mathcal H^N(B_{r_n}(0))}\to 0, \quad \text{as $n\to +\infty.$}
\end{equation}
Next we define maps $f_n: A_n\subset B_1(0)\to Y$ by $f_n(x)=r_n^{-1}f(r_nx)$, which by \eqref{eq:isometry banach} satisfy $\|f_n(x)-f_n(y)\|_Y=\|x-y\|_{\X}$ for all $x,y \in A_n.$ In particular each $f_n$ is 1-Lipschitz and can be extended to a 1-Lipschitz map to the whole $\bar B_1(0)$, still denoted by $f_n$. Then by the Arzel\`a-Ascoli theorem and up to passing to a subsequence, the functions $f_n$ converge uniformly in $\bar B_1(0)$ to a limit function $\bar f.$
Thanks to \eqref{eq:zero density An} we have that $\bar B_1(0)\subset (A_n)^{\eps_n}$ for some $\eps_n\to 0$ (where $(A)^\eps$ denotes the $\eps$-tubular neighbourhood of a set $A$). In particular by \eqref{eq:isometry banach}, the triangle inequality and the 1-Lipschitzianity of $f_n$ it holds
\begin{equation}\label{eq:fn almost isom}
    | \|f_n(x)-f_n(y)\|_Y-\|x-y\|_{\X}|\le 4\eps_n, \quad \forall x,y \in \bar B_1(0).
\end{equation}
Therefore passing to the $\limsup_n$ on both sides we obtain that $ \|\bar f(x)-\bar f(y)\|_Y=\|x-y\|_{\X}$ for every $x,y \in \bar B_1(0)$. From this the conclusion follows from the first part of the proof.
\end{proof}


\section{Low-dimensional uniformly rectifiable subsets of Heisenberg groups}\label{sec:heis}

The purpose of this section is to discuss the
{$\nb_{p,\mathcal{V}}$-coefficients} in specific
metric spaces, the \emph{Heisenberg groups} with \emph{Kor\'{a}nyi
distances}. (Quantitative) rectifiability has been studied
extensively in this setting using various analogs of Jones'
$\beta$-numbers. As we will show in Section \ref{s:Juillet},
horizontal $\beta$-numbers
$\beta_{\infty,\mathcal{V}^1(\mathbb{H}^n)}$ defined with respect
to the Kor\'anyi distance \emph{cannot} be used to characterize
uniform $1$-rectifiability in $\mathbb{H}^n$, $n>1$, by means of
$\mathrm{GLem}(\beta_{\infty,\mathcal{V}^1(\mathbb{H}^n)},p)$ for
any fixed exponent $p\geq 1$.  S.\ Li observed a similar
phenomenon regarding rectifiability of curves in the Carnot group
$\mathbb{R}^2\times \mathbb{H}^1$ in \cite[Proposition
1.4]{https://doi.org/10.1112/jlms.12582}, motivating the use of
\emph{stratified $\beta$-numbers} in Carnot groups.

In Section \ref{s:beta_alpha_Heis} we show for $k$-regular sets in
$\mathbb{H}^n$, $1\leq k\leq n$, that
$\mathrm{GLem}(\beta_{2p,\mathcal{V}^k(\mathbb{H}^n)},2p)$ implies
$\mathrm{GLem}(\nb_{p,\mathcal{V}^k(\mathbb{H}^n)},p)$ for any $p\geq 1$. Here $\mathcal{V}^k(\mathbb{H}^n)$ denotes the family of affine horizontal $k$-planes (see section below for details).  This is an
application of our axiomatic statement in Theorem \ref{thm:axiomatic}.

\subsection{The Heisenberg group}\label{ss:Heis}
We consider the \emph{$n$-th Heisenberg group}
$\mathbb{H}^n=(\mathbb{R}^{2n+1},\cdot)$, given by the group
product
\begin{displaymath}
(x,t) \cdot (x',t')= \Big(x_1+x_1',\ldots,x_{2n}+x_{2n'},t+t'+
\omega(x,x')\Big),\quad (x,t),(x',t')\in \mathbb{R}^{2n}\times
\mathbb{R},
\end{displaymath}
where
\begin{displaymath}
\omega(x,x'):= \tfrac{1}{2} \sum_{i=1}^n x_i
x_{n+i}'-x_{n+i}x_i',\quad x,x'\in \mathbb{R}^{2n}.
\end{displaymath}
For every point $p \in \H^n$, we denote by $[p]\in \R^{2n}$ its
first $2n$ coordinates. The Euclidean norm on $\mathbb{R}^m$ is
denoted by $|\cdot|_{\mathbb{R}^m}$, or simply by $|\cdot|$. We
equip $\mathbb{H}^n$ with the left-invariant \emph{Kor\'{a}nyi
metric}
\begin{displaymath}
d_{\mathbb{H}^n}(p,p'):=\|p^{-1}\cdot
p'\|_{\mathbb{H}^n},\quad\text{where}\quad
\|(x,t)\|_{\mathbb{H}^n}:= \sqrt[4]{|x|_{\mathbb{R}^{2n}}^4 + 16
t^2}.
\end{displaymath}
In particular it holds
\begin{equation}\label{eq:koranyi inequaly}
    d_{\mathbb{H}^n}(p,p')\ge |[p']-[p]|_{\R^{2n}}.
\end{equation}

\subsubsection{Isotropic subspaces}\label{ss:isotropic}
We focus our attention on $k$-regular sets in
$(\mathbb{H}^n,d_{\mathbb{H}^n})$ for $k\leq n$. The threshold
$k=n$ is related to the dimension of isotropic subspaces in
$\mathbb{R}^{2n}$. A subspace $V\subset \R^{2n}$ is called
\emph{isotropic} if $\omega(x,y)=0$ for every $x,y \in V.$ If $V$
is isotropic then ${\rm dim}(V)\le n$. The subspace property and the vanishing of the form $\omega$ on $V$ ensure that $V\times \{0\}$ is a subgroup of $\mathbb{H}^n=(\mathbb{R}^{2n+1},\cdot)$ if $V$ is isotropic.

For all $k \in \N$, $1\le
k\le n$, we define the \emph{horizontal subgroups}
\[
\calV^k_0:=\calV^k_0(\mathbb{H}^n)\coloneqq\{V\times \{0\}, \ : \
V\subset \R^{2n} \text{ isotropic of dimension } k\}
\]
and the \emph{affine horizontal $k$-planes}
\[
\calV^k:=\calV^k(\mathbb{H}^n)\coloneqq\{p\cdot \mathbb{V}, \ : \
\mathbb{V}\in \calV^k_0, \, p \in \H^n\}.
\]
The elements of $\calV^1$ are also called \emph{horizontal lines}.
With our choice of coordinates for $\H^n$ we have
\begin{equation}\label{eq:model isotropic}
    \{(v,0) \ : \ v =(v_1,\ldots,v_k,0,\ldots,0)\in \R^{2n}\}\in \calV^k_0, \quad k=1,\ldots,n.
\end{equation}
For every $\V \in \calV^k$, we define  $V'\subset \R^{2n}$ as the
unique $k$-dimensional subspace such that $\V=p\cdot (V'\times
\{0\}),$ for some $p \in \H^n.$

A key property of the spaces $\V \in \calV^k$ is that
\begin{equation}\label{eq:euclidean metric}
    d_{\H^n}(v_1,v_2)=|v_1'-v_2'|_{\R^{2n}}=|[v_1]-[v_2]|_{\R^{2n}}, \quad v_1,v_2 \in \V,
\end{equation}
where $\V=p\cdot (V'\times \{0\})$ and $v_i'\in V'$ is such that $v_i=p\cdot (v',0)$ (recall that $[v_i]$ denote the
first $2n$-entries of $v_i$). In particular $(\V,d_{\H^n})$ is
isometric to $(\R^k,|.|)$

For more detailed preliminaries on isotropic subspaces in the
context of the Heisenberg group, see for instance
\cite{MR2955184}.

\subsubsection{Horizontal projections onto affine horizontal planes}\label{ss:HorizProj}

The \emph{horizontal projection} $P_{\mathbb{V}'}$ from $\mathbb{H}^n$ onto a horizontal subgroup $\mathbb{V}'=V'\times \{0\}\in \mathcal{V}^k_0$ is simply defined as
\begin{equation}\label{eq:HorizProjSubgroup}
P_{\mathbb{V}'}:\mathbb{H}^n \to \mathbb{V}',\quad P_{\mathbb{V}'}(x,t)=(\pi_{V'}(x),0),
\end{equation}
where $\pi_{V'}:\mathbb{R}^{2n}\to V$ denotes the usual orthogonal projection onto the $k$-dimensional isotropic subspace $V'$. Horizontal projections in the Heisenberg group are well-studied, see for instance \cite{MR2789472,MR2955184}. For the purpose of defining the $\nb_{p,\mathcal{V}^k(\mathbb{H}^n)}$-numbers, we need a variant of these mappings where we project onto \emph{affine} horizontal $k$-planes $\mathbb{V}\in \mathcal{V}^k$ that do not necessarily pass through the origin.

Let $\mathbb{V}$ be such an affine horizontal $k$-plane in $\mathbb{H}^n$, that is,
 $\V=q \cdot (V'\times\{0\})$ for some
$q\in \H^n$ and a $k$-dimensional isotropic subspace  $V'$. We set $\V'=V'\times \{0\}$ and $\V'^\perp\coloneqq {V'}^{\perp}\times{\R}\subset
\H^n$, where ${V'}^\perp$ is the Euclidean orthogonal complement of $V'$ in
$\R^{2n}$. Finally, as in \cite[Section 2]{hahheisenberg}, we define the \emph{(horizontal) projection $P_{\mathbb{V}}:\mathbb{H}^n \to \mathbb{V}$ onto the affine plane $\mathbb{V}$} as
\begin{equation}\label{eq:AffineHorizProj}
P_{\mathbb{V}}:\mathbb{H}^n \to \mathbb{V},\quad P_{\mathbb{V}}(p):=q \cdot P_{\mathbb{V}'}(q^{-1}\cdot p),
\end{equation}
where the projection $P_{\mathbb{V}'}$ onto the horizontal subgroup is defined as in \eqref{eq:HorizProjSubgroup}. Explicitly, writing in coordinates $p=([p],t_p)$ and $q=([q],t_q)$, we obtain
\begin{equation}\label{eq:explicit proj}
    P_{\mathbb{V}}(p)=(v+[q],t_q+\omega([q],v)), \quad v\coloneqq\pi_{V'}([p]-[q]),
\end{equation}
where $\pi_{V'}: \R^{2n}\to V'$ is the Euclidean orthogonal
projection onto $V'.$

The choice of ``$q$'' in the formula $\V=q\cdot \V'$ is not unique, but $P_{\mathbb{V}}$ in \eqref{eq:AffineHorizProj} is nonetheless well-defined. Indeed, it is well known, and can be easily verified by a computation (with $P_{\mathbb{V}'}(q)=(\xi,0)$), that a given point $q$ can be written in a unique way as $q=(\zeta,\tau)\cdot (\xi,0)$ with $\xi\in V',\zeta\in V'^{\perp}$ and $\tau\in \mathbb{R}$. While the point $q$ in $\V=q\cdot \V'$ is not uniquely determined by $\V$, its $\mathbb{V}'^{\perp}$-component $(\zeta,\tau)$ is, and since we can write $P_{\mathbb{V}}(p)=(\zeta,\tau)\cdot P_{\mathbb{V'}}(p)$, we conclude that the projection in \eqref{eq:AffineHorizProj} is well-defined. It is also consistent with the definition in \eqref{eq:HorizProjSubgroup} if $\mathbb{V}\in \mathcal{V}_0^k$. Moreover, it is easy to check from the definition that $P_{\mathbb{V}}$ is $1$-Lipschitz.

The projection $P_{\mathbb{V}}$ plays an important role also in the following decomposition. Every point $p \in \H^n$ can be written in a unique way as
\[
p=p_\V\cdot p_{\V'^\perp}, \quad p_{\V} \in \V,\, p_{\V'^{\perp}}
\in \V'^\perp,
\]
where $p_{\V}=P_{\mathbb{V}}(p)$. First, it is easy to see that $\mathbb{H}^n = q \cdot\mathbb{H}^n=q\cdot \V'\cdot \V'^\perp=\V\cdot  \V'^\perp$, that is, every point $p\in \mathbb{H}^n$ has \emph{some} decomposition $p=p_\V\cdot p_{\V'^\perp}$ as above. Applying the horizontal projection to the subgroup $\V'$, we deduce that $P_{\V'}(p)=P_{\V'}(p_{\V})$, and finally, by what we said earlier, $p_{\V}=(\zeta,\tau)\cdot P_{\V'}(p)$ as desired.

From the definition of $P_\V$ we can see that
\begin{equation}\label{eq:proj and translation}
    w\cdot P_{\V}(p)=P_{w\cdot \V}(w\cdot p), \quad w,p \in \H^n.
\end{equation} Combining  \eqref{eq:AffineHorizProj} (or \eqref{eq:explicit proj}) with \eqref{eq:euclidean metric}
and the fact that $\pi_V(x-w)+w=\pi_{w+V}(x)$ (for every $x,w \in
\R^{2n}$ and subspace $V$) we get
\begin{equation}\label{eq:proj translated plane}
    d_{\H^n}(P_{\V}(x),P_{\V}(y))=|\pi_{[q]+V'}([x])-\pi_{[q]+V'}([y])|,\quad x,y \in \H^n,
\end{equation}
where $\pi_{[q]+V'}: \R^{2n}\to [q]+V'$ is the Euclidean
orthogonal projection onto the affine plane $[q]+V'$ (as usual
$\V=q \cdot (V'\times\{0\})$).

\subsection{The geometric lemma for horizontal $\beta$- and $\nb$-numbers in Heisenberg groups}\label{s:beta_alpha_Heis}
We now verify, for all integers $1\leq k\leq n$, that affine
horizontal planes $\mathcal{V}^k(\mathbb{H}^n)$, horizontal
projections and angles between affine horizontal planes (as in
Definition \ref{d:AngleIsotr} below) satisfy the assumptions of
the `axiomatic' statement in Theorem \ref{thm:axiomatic}. In other
words, we establish a counterpart of Proposition
\ref{prop:euclidean planes} in Heisenberg groups (Section
\ref{ss:systemsHeis}) and verify a Heisenberg tilting estimate in
the spirit of Proposition \ref{p:EuclTilt} (Section
\ref{ss:HeisTilt}). As a consequence, we derive the following
result (for the definitions of the numbers
$\beta_{2p,\mathcal{V}^k(\mathbb{H}^n)}$ and
$\nb_{p,\mathcal{V}^k(\mathbb{H}^n)}$ see
{\eqref{eq:beta_q_V_planes} and \eqref{d:newbetas
proj} (with $\mathcal{V}=\mathcal{V}^k(\mathbb{H}^n)$),
respectively}).

\begin{thm}\label{t:from beta_to_alpha_in_heis}
Let $1\leq k\leq n$ be integers, let $1\leq p<\infty$, and assume
that $E\subset \mathbb{H}^n$ satisfies $E\in \mathrm{Reg}_k(C)\cap
\mathrm{GLem}(\beta_{2p,\mathcal{V}^k(\mathbb{H}^n)},2\textcolor{black}{q},{M})$, then
$E\in \mathrm{GLem}(\nb_{p,\mathcal{V}^k(\mathbb{H}^n)},\textcolor{black}{q},{\hat C M})$ {with $\hat C = \hat C(k,p,\textcolor{black}{q},C)$} \textcolor{black}{for all $q\geq p$ (and all $q\in [1,\infty)$ in case $p<k$).}
\end{thm}

\begin{proof}[Proof of Theorem \ref{t:from
beta_to_alpha_in_heis}] This is a direct consequence of
Propositions  \ref{prop:Heisenberg planes} and   \ref{p:HeisTilt} below, as well as Theorem
\ref{thm:axiomatic}.
\end{proof}

\begin{remark}
Hahlomaa showed in \cite[Thm 1.1]{hahheisenberg} for $1\leq k\leq n$ that
 $E\in
\mathrm{Reg}_k(C)\cap \mathrm{GLem}(\beta_{1,\mathcal{V}^k(\mathbb{H}^n)},2)$  implies that the set $E\subset \mathbb{H}^n$ has
big pieces of bi-Lipschitz images of subsets of $\mathbb{R}^k$. {In particular, this holds for $E\in
\mathrm{Reg}_k(C)\cap \mathrm{GLem}(\beta_{2,\mathcal{V}^k(\mathbb{H}^n)},2)$.}
\end{remark}

\subsubsection{Systems of {planes-projections-angle} in Heisenberg groups}\label{ss:systemsHeis}
{Throughout this section we employ the abbreviating notations $\mathcal{V}^k=\mathcal{V}^k(\mathbb{H}^n)$ and $\mathcal{V}_0^k=\mathcal{V}_0^k(\mathbb{H}^n)$.}

\begin{definition}[Angle between affine horizontal
planes]\label{d:AngleIsotr}
    Given $\V_1,\V_2 \in \calV^k$ we define
        \[
    \angle(\V_1,\V_2)\coloneqq \angle_e(V_1',V_2'),
    \]
    where $\angle_e$ is as in Definition \ref{d:AngleEucl} {and $V_i'$ is the unique isotropic subspace of $\mathbb{R}^{2n}$ such that $\V_i=p^i \cdot (V_i'\times \{0\})$ for some $p^i\in\mathbb{H}^n$.}
\end{definition}
\begin{proposition}\label{prop:Heisenberg planes}
      Fix $n,k\in \mathbb N$ with $k\leq n$.
      Then the triple $(\mathcal V^k,\mathcal P,\angle)$,
      where $\mathcal P\coloneqq \{P_{\V}\}_{\V \in \mathcal V^k}$,
      is a system of planes-projections-angle for $(\mathbb{H}^n,d_{\mathbb{H}^n})$.
\end{proposition}

\begin{proof}[Proof of Proposition \ref{prop:Heisenberg planes}]
We verify that $(\mathcal V^k,\mathcal P,\angle)$
satisfies the assumptions of Definition \ref{def:planes system}.
First, condition \ref{it:triangle angle} follows immediately
from the corresponding property of the Euclidean angle $\angle_e$ 
that was stated in Proposition \ref{prop:euclidean planes}. Thus,
\begin{equation}\label{eq:triangle ineq angle}
    \angle(\V_1,\V_3)\le \angle(\V_1,\V_2)+\angle(\V_2,\V_3), \quad \V_1,\V_2,\V_3 \in \calV^k.
\end{equation}
The more laborious part is the verification of the second
condition, \ref{it:pit}. This is the content of Proposition
\ref{prop:double pythagora} below.
\end{proof}

\begin{proposition}[Heisenberg two planes Pythagorean theorem]\label{prop:double pythagora} {There exists an absolute constant $N_0>1$ such that the following holds.}
    Let $\V,\W\in \calV^k$ for some $1\le k\le n$. Let $x,y \in \H^n$ be distinct and such that
    $$d_{\H^n}(x,\V)\le c\,d_{\H^n}(x,y)\quad{\text{and}\quad d_{\H^n}(y,\V)\le c\,d_{\H^n}(x,y)}
    $$
    for some $c>0.$ Then
    \begin{equation}\label{eq:double pythagora}
        d_{\H^n}(x,y)^2\le
        d_{\H^n}(P_{\W}(x),P_{\W}(y))^2+d_{\H^n}(x,y)^2\left (\angle(\V, \W)+{N_0(1+c)}\frac{d_{\H^n}(y,\V)+d_{\H^n}(x,\V) }{d_{\H^n}(x,y)}\right)^2.
    \end{equation}
\end{proposition}

We will need several preliminary lemmata. The first result is essentially
present in \cite[Lemma 2.2]{hahheisenberg} (see also \cite[Proposition 2.15]{MR2789472} for $\mathbb{V}\in \mathcal{V}_0^k$),  but we include a proof  for
completeness:
\begin{lemma}[Minimal distance vs.\ projection]\label{l:MinDistProj}
    Let $\V\in \calV^k$ for some $1\le k\le n$. Then
    \begin{equation}\label{eq:proj dist comparison}
        2^{-\frac54}d_{\H^n}(p,P_\V(p))\le d_{\H^n}(p,\V)\le d_{\H^n}(p,P_\V(p)), \quad p \in \H^n.
    \end{equation}

\end{lemma}
\begin{proof} It suffices to prove the first inequality in \eqref{eq:proj dist comparison}; the second one follows immediately from the fact that $P_{\V}(p)\in \V$.
    Thanks to \eqref{eq:proj and translation} (and a rotation {of the form $(x,t)\mapsto \left(Ax,(\det A)t\right) $ for suitable $A\in U(n)$}; see \cite[Lemma 2.1]{MR2955184}) we can assume without loss of generality
    that $\V\in \calV^k_0$ and $\V=\{(x,0), \ : \  x=(x_1,\ldots,x_k,0,\ldots,0) \in \R^{2n]} \}$, so that
    $\V^\perp=\{(y,t), \ : \  y=(0,\ldots,0,y_1,\ldots,y_{2n-k}) \in \R^{2n}, t \in \R \}$ (recall \eqref{eq:model isotropic}).

    Since $\mathbb{H}^n=\mathbb{V} \cdot \V^\perp$ we can write any $p\in \mathbb{H}^n$ as
    $$p= p_\V\cdot p_{\V \perp}=(x,0)\cdot (y,t) ,$$
    with $p_\V =P_{\V}(p) \in \V$, $p_{\V^\perp}\in \V^\perp$.  Clearly,
    \begin{equation}\label{eq:distance proj}
        d_{\H^n}(p,p_\V)^4=\|p_{\mathbb{V}^{\bot}}\|_{\mathbb{H}^n}^4=|y|^4+16t^2.
    \end{equation}
    Now fix any $q \in \V$, $q=(\bar x,0)$, $\bar x =(\bar x_1,\ldots,\bar x_k,0,\ldots,0)$. Then
    \[
    d_{\H^n}(p,q)^4=(|x-\bar x |^2+|y|^2)^2+16(t+\omega(y,\bar x- x))^2.
    \]
    We distinguish two cases for $q$. Assume first that $|\omega(y,\bar x-x)|\le \frac12|t|$.  Then
    \[
    d_{\H^n}(p,q)^4\ge |y|^4+4t^2 \overset{\eqref{eq:distance proj}}{\ge} \frac14     d_{\H^n}(p,p_{\V})^4
    \]
    If instead $|\omega(y,\bar x-x)|> \frac12|t|$, using Young's inequality
    \begin{align*}
        d_{\H^n}(p,q)^4&\ge \tfrac12|y|^4+\tfrac12(|x-\bar x |^2+|y|^2)^2\ge \tfrac12|y|^4+\tfrac12(2|x-\bar x ||y|)^2\\
        &\ge \tfrac12|y|^4+ 2 |\omega(y,\bar x-x)|^2\ge \tfrac12|y|^4+ \tfrac12 |t|^2\overset{\eqref{eq:distance proj}}{\ge}
         \tfrac1{32} d_{\H^n}(p,p_{\V})^4.
    \end{align*}
    By the arbitrariness of $q$,  \eqref{eq:proj dist comparison} follows.
\end{proof}

We can now prove a version of the Pythagorean theorem with a single
plane, which will be useful in the proof of Proposition \ref{prop:double pythagora}.
\begin{lemma}[Basic Pythagoras-type theorem]\label{lem:pythagora} {There exists an absolute constant $N\geq 1$ such that the following holds.}
    Let $\V\in \calV^k$ for some $1\le k\le n$.
    Let $p^1,p^2 \in \H^n$ be such that
\begin{equation}\label{eq:epsilon_ass}
    d_{\H^n}(p^i,\V)\le c \,d_{\H^n}(p^1,p^2), \quad i=1,2,
    \end{equation}
    for some $c>0.$ Then
    \begin{equation}\label{eq:pythagora}
        d_{\H^n}(p^1,p^2)^2\le d_{\H^n}(P_\V(p^1),P_\V(p^2))^2+N(1+c^2)(d_{\H^n}(p^1,\V)+d_{\H^n}(p^2,\V))^2.
    \end{equation}
\end{lemma}
\begin{proof}
{By the same reasoning as at the beginning of the proof of Lemma \ref{l:MinDistProj}}, it is not restrictive to assume that $\V=V\times \{0\}\in \calV_0^k$ and  $V=\mathbb{R}^k\times \{(0,\ldots,0)\}$. In particular $p^i=(x_i+y_i,t_i)$, $p^i_\V=P_{\V}(p_i)=(x_i,0)$ and
    \[
    p^i=(x_i,0)\cdot (y_i,t_i-\omega(x_i,y_i)),
    \]
    where $x_i \in V$ and $y_i\in V^{\bot}=\{(0,\ldots,0)\}\times \mathbb{R}^{n-k}$.
    Set \begin{equation}\label{eq:Bound_epsilon_c}\eps\coloneqq d_{\H^n}(p^1,\V)+ d_{\H^n}(p^2,\V)\quad\left(\overset{\eqref{eq:epsilon_ass}}{\leq }2 c\, d_{\mathbb{H}^n}(p^1,p^2)\right).\end{equation}
From \eqref{eq:proj dist comparison} and for $i=1,2:$
    \begin{equation*}
        \eps^4\ge d_{\H^n}(p^i,\V)^4\ge (1/32) d_{\H^n}(p^i,p_\V^i)^4= (1/32)(|y_i|^4+16(t_i-\omega(x_i,y_i))^2).
    \end{equation*}
    In particular,
    \begin{equation}\label{eq:estimates}
        |y_i|^2\le \sqrt{32} \eps^2,\quad   (t_i-\omega(x_i,y_i))^2\le 2\eps^4.
    \end{equation}
    Moreover,
    \[
    d_{\H^n}(p^1,p^2)^4=(|x_1-x_2|^2+|y_1-y_2|^2)^2+16\left[t_1-t_2+\omega(-(x_2+y_2),(x_1+y_1))\right]^2.
    \]
    and since $\V\in \calV_0^k$,  we have $d_{\H^n}(p^1_\V,p^2_\V)=|x_1-x_2|$.{
In the following, $N$ denotes a constant whose value may change from line to line, but which can be chosen independently of $p^1$ and $p^2$.}
    To prove
   \eqref{eq:pythagora},
    we  compute, by {using that $x_1$ and $x_2$ belong to the isotropic subspace $V$,}
    \begin{align*}
        d_{\H^n}(p^1,p^2)^4&=(|x_1-x_2|^2+|y_1-y_2|^2)^2+16[t_1-t_2+\omega(x_1,y_2)+\omega(y_1,x_2)+{\omega(y_1,y_2)}]^2\\
        &=(|x_1-x_2|^2+|y_1-y_2|^2)^2+16\big [t_1-t_2-\omega(x_1,y_1)+\omega(x_2,y_2)+{\omega(y_1,y_2)}\\
        &\quad +\omega(x_1-x_2,y_1+y_2)\big ]^2\\
        &\le |x_1-x_2|^4+|y_1-y_2|^4+2|x_1-x_2|^2|y_1-y_2|^2+ N (t_1-\omega(x_1,y_1))^2\\
        &\,\, +N(t_2-\omega(x_2,y_2))^2+N\, \omega(x_1-x_2,y_1+y_2)^2{+ N\,\omega(y_1,y_2)^2}\\
        &\overset{\eqref{eq:estimates}}{\le} d_{\H^n}(p^1_\V,p^2_\V)^4+N \eps^4+ Nd_{\H^n}(p^1,p^2)^{{2}}\eps^2\\
    &\overset{\eqref{eq:Bound_epsilon_c}}{\le}  d_{\H^n}(p^1_\V,p^2_\V)^4+N{(c^2+1)}d_{\H^n}(p^1,p^2)^2\eps^2.
    \end{align*}
    Dividing by $d_{\H^n}(p^1,p^2)^2$, the conclusion follows.
\end{proof}

The last ingredient for the proof of Proposition \ref{prop:double pythagora} is the following.
\begin{lemma}
    Let  $\V \in \calV^k$ for some $1\leq k\leq n$ be such that $\V=q \cdot (V' \times \{0\})$. Then
    \begin{equation}\label{eq:distance from plane}
        d_{\H^n}(x,\V)\ge d_{\R^{2n}}([x],[q]+V'), \quad x \in \H^{n}.
    \end{equation}
\end{lemma}
\begin{proof}{
This follows immediately from \eqref{eq:koranyi inequaly}.}
\end{proof}

We can pass to the proof of the main result of this section.

\begin{proof}[Proof of Proposition \ref{prop:double pythagora}]
Let $x,y\in \H^{n}$ be arbitrary and let $p,q \in \H^n$ be such
that $\V=p \cdot (V' \times \{0\})$ and $\W=q \cdot (W' \times
\{0\})$. From the Euclidean two-planes Pythagorean theorem in
Lemma \ref{prop:two planes pitagora} and the inequality
$|\pi_{[p]+V'}[x]-\pi_{[p]+V'}[y]|\le |[x]-[y]|$ (Euclidean
projections are 1-Lipschitz), we  have
\begin{align*}
        |\pi_{[p]+V'}[x]-&\pi_{[p]+V'}[y]|^2\le  |\pi_{[q]+W'}([x])-\pi_{[q]+W'}([y])|^2\\
        &+\left (|[x]-[y]|\angle_e(V', W')+d_{\mathbb{R}^{2n}}([y],[p]+V')+d_{\mathbb{R}^{2n}}([x],[p]+V') \right)^2.
\end{align*}
Recalling \eqref{eq:koranyi inequaly}, \eqref{eq:proj translated
plane} and \eqref{eq:distance from plane}, as well as Definition \ref{d:AngleIsotr} for the  angle between affine horizontal planes, the above implies
\begin{align*}
    &d_{\H^n}(P_{\V}(x),P_{\V}(y))^2\\
    &\le d_{\H^n}(P_{\W}(x),P_{\W}(y))^2+\left (d_{\H^n}(x,y)\angle(\V, \W)+d_{\H^n}(y,\V)+d_{\H^n}(x,\V) \right)^2.
\end{align*}
Combining this with the (single-plane) Pythagora's theorem in
Lemma \ref{lem:pythagora} we obtain, with an absolute constant $N\geq 1$,
\begin{align*}
        &d_{\H^n}(x,y)^2-N(1+c^2)(d_{\H^n}(x,\V)+d_{\H^n}(y,\V))^2\\
        &\le d_{\H^n}(P_{\W}(x),P_{\W}(y))^2+\left (d_{\H^n}(x,y)\angle(\V, \W)+d_{\H^n}(y,\V)+d_{\H^n}(x,\V) \right)^2,
\end{align*}
which immediately implies \eqref{eq:double pythagora}. This
concludes the proof of Proposition \ref{prop:double pythagora}.
\end{proof}
With Proposition \ref{prop:double pythagora} in hand, the proof of
Proposition \ref{prop:Heisenberg planes} is now complete.

\subsubsection{The Heisenberg tilting estimate}\label{ss:HeisTilt}
To conclude the proof of Theorem \ref{t:from beta_to_alpha_in_heis} it remains to prove the following Heisenberg tilting estimate.

\begin{proposition}[Heisenberg tilting estimate]\label{p:HeisTilt} {Let $1\leq k\leq n$ and $E\subset \mathbb{H}^n$ with $E\in {\rm Reg}_k(C_E)$  with a system $\Delta$ of dyadic cubes.}
   Let $Q_1,Q_0 \in \Delta$ and $\lambda_0,\lambda_1 \ge 1$ with $\lambda_1Q_1\subset \lambda_0 Q_0$, $Q_1\in \Delta_{j-1}\cup \Delta_j$, $Q_0 \in \Delta_{j-1}$ {for some $j-1\in \mathbb{J}$.}
    Then, if $\V_{Q_0},\V_{Q_1} \in\calV^k$ are affine horizontal planes,
    we have
    \begin{equation}\label{eq:tilting}
        \angle(\V_{Q_0},\V_{Q_1})\le \lambda_0^{k+1}C(\beta_{p,\V_{Q_0}}(\lambda_0 Q_0)+\beta_{p,\V_{Q_1}}(\lambda_1 Q_1)), \quad  p\in [1,\infty),
    \end{equation}
    where $C$ is a constant depending only on $k$ and $C_E$.
\end{proposition}

This requires a few preliminary results. The first one  says that  given  two $k$-dimensional subspaces $V_1,V_2$  of $\rr^N$ such that $V_1$  contains $(k+1)$ sufficiently independent points whose mutual distances are almost preserved when projected onto $V_2$, then the Euclidean angle between $V_1$ and $V_2$ is small. Since this is a classical, purely Euclidean, result, but we were unable to find the precise statement in the literature, we include a proof in Appendix \ref{s:AppendixB}.
\begin{lemma}[Small angle criterion]\label{lem:planes lemma}
    Fix $k, N \in \N$ with ${1}\leq k\le N$ and $c>0$. Then there exists a constant {$D=D(k,c)>0$} such that the following holds.  Let $V_1,V_2$ two $k$-dimensional subspaces of $\R^N$ and $r>0$, $\eps\in(0,1)$ be arbitrary. Suppose there exist $y_0,...,y_k\in V_1$ with $|y_i-y_j|\le r$ such that
    \begin{enumerate}[label=\roman*)]
        \item\label{it:indep}
        \begin{equation}\label{eq:suff indep}
            {\sup_{i=0,...,k} }\,d_{\R^N}(y_i,W)>cr
        \end{equation}
        for every $W$ $(k-1)$-dimensional  {affine} subspace of $V_1$,
        \item\label{it:pyth}
        \begin{equation}\label{eq:close projection}
            |y_j-y_i|^2\le (1+\eps^2)|\pi_{V_2}(y_i)-\pi_{V_2}(y_j)|^2, \quad i,j=0,\dots,k.
        \end{equation}
    where $\pi_{V_2}$ is the orthogonal projection onto $V_2$.
    \end{enumerate}
    Then $\angle_e(V_1,V_2)\le D \eps.$
\end{lemma}
In the above  statement,  when $k=1$, by 0-dimensional affine subspace of $V_1$ we mean \emph{a point} in $V_1$ and so \eqref{eq:suff indep} is simply saying that $|y_0-y_1|>2cr.$

\begin{proposition}[Existence of $(k+1)$-independent points]\label{prop:indep points} {Let $1\leq k\leq n$ and let $E \subset \mathbb{H}^n$ be a $k$-regular set with a system $\Delta$ of dyadic cubes. Denote $d(Q)=\mathrm{diam}(Q)$ for $Q\in \Delta$.}
    For every $Q \in \Delta$ there exist $k+1$ points $\{x_0,\dots,x_k\}\subset Q$ such that $B_{cd(Q)}(x_i)\cap E\subset Q$   for all $i=0,\ldots,k$,  $d_{\H^n}(x_i,x_j)\ge c d(Q)$, $i \neq j$, and
    \begin{equation}\label{eq:heis indep points}
        {\sup_{i=0,\ldots,k}} d_{\H^n} (x_i,\W)\ge cd(Q)>0, \quad \text{ for every }\W \in \calV^{k-1},
    \end{equation}
where $c=c(k,C_E){\in (0,1)}$ is constant depending only on $k$ and the regularity constant of $E$ (using  the convention $\calV^{0}\coloneqq \mathbb{H}^n$).
\end{proposition}
\begin{proof} Let $\mu\coloneqq \mathcal{H}^k\llcorner_{E}$.
        Fix $\lambda,c\in(0,1)$  to be chosen later ($c$ chosen after $\lambda$) and that will ultimately depend only on $k$ {and $C_E$}. Fix $Q \in \Delta$. By (5) in Definition \ref{dl:dyadic} there exist a constant $\lambda_0$ depending only on $k$ and $C_E$ and a point $q_0\in Q$ such that $B_{\lambda_0d(Q)}(q_0)\cap E\subset Q$.   Consider a set $H\subset B_{\lambda_0d(Q)/2}(q_0)\cap E$ satisfying $d_{\H^n}(w,z)>\lambda d(Q)$ for all $w,z \in H$ and maximal with respect to inclusion. Clearly $B_{\lambda d(Q)}({z})\cap E\subset Q$ for all $z\in H$, provided $\lambda <\lambda_0/2$. Moreover  the set $H$  contains at least $k+1$ points {(in fact $\sim \lambda^{-k}$ points) if $\lambda$ is chosen small enough depending on $k$ and $C_E$}.  Indeed  $B_{\lambda_0d(Q)/2}(q_0)\cap E$ is contained in the $2\lambda d(Q)$-neighborhood of $H$. Therefore, if $H$ contained at most $k$ points by the $k$-regularity  of $E$ we would have
        \[
C
d(Q)^k \le\mu(B_{\lambda_0d(Q)/2}(q_0)\cap E)\le C_E k (2\lambda)^k d(Q)^k,
        \]
        {for a constant $C$ depending only on $C_E$ and $k$,} which can not be true provided $\lambda$ is chosen small enough. If $k=1$,  we can find two points $x_0,x_1\in H$, which in particular satisfy $d_{\H^n}(x_0,x_1)>\lambda d(Q)$, which implies \eqref{eq:heis indep points} with $c=\lambda/2.$ Hence from now on we assume that $k\ge2.$
        Fix a subset of $k$ distinct points $\{x_0,\dots,x_{k-1}\}\subset H$. If for all $\W \in \calV^{k-1}$ we had ${\sup_{i=0,\ldots,k-1}}d_{\H^n}(x_i,\W)\ge c d(Q)$ we would conclude by adding to this set any $x_k\in H$ distinct from the previous ones. Therefore we can assume that there exists $\W \in \calV^{k-1}$ with ${\sup_{i=0,\ldots,k-1}}d_{\H^n}(x_i,\W)< c d(Q)$.

        If ${\sup_{x\in H}}d_{\H^n}(x,\W)\ge c d(Q)$ we would be again done, hence suppose the contrary.   Consider the set $H'\coloneqq P_\W(H)\subset \W$, where $P_\W$ is the horizontal projection onto $\W$ (as defined in \eqref{eq:AffineHorizProj}).
        Assuming $c\le \lambda$, by the current standing assumptions and the definition of $H$, we have
        \[
       \max\{ d_{\H^n}(y,\mathbb W), d_{\H^n}(z,\mathbb W)\}\le cd(Q)\le c\lambda^{-1} d_{\H^n}(y,z)
       , \quad y,z\in H.
        \]
        Hence, applying the Pythagorean-type theorem  (Lemma \ref{lem:pythagora}), we obtain for $y,z\in H$:

        \begin{equation}\label{eq:ProjInj}
        d_{\H^n}(P_\W(y),P_\W(z))^2\ge d_{\H^n}(y,z)^2(1-N(1+(c\lambda^{-1})^2)(2c\lambda^{-1})^2)\ge (\lambda/2)^2d(Q)^2,\end{equation}
       provided that $c$ is small enough with respect to $\lambda.$
Here $N>0$ is the absolute constant given by Lemma \ref{lem:pythagora}. The estimate \eqref{eq:ProjInj} in particular shows that $\card(H')=\card(H)$. Moreover $H'\subset B_{d(Q)}(P_\W(x_0))\cap \W$, since $P_\W$ is $1$-Lipschitz. Therefore, using that $(\W,d_{\H^n})$ is isometric to $(\R^{k-1},d_{\R^{k-1}})$, using \eqref{eq:ProjInj} and a standard covering argument gives
        \[
        \card(H)=\card(H')\le \frac{c_k}{(\lambda/2)^{k-1} }.
        \]
        Therefore, recalling that  $B_{\lambda_0d(Q)/2}(q_0)\cap E$ is contained in the $2\lambda d(Q)$-neighborhood of $H$, and from the $k$-regularity of $E$:
        \[
         {C}d(Q)^k \le  \mu( B_{\lambda_0d(Q)/2}(q_0)\cap E)\le C_E  \card(H) (2\lambda)^k d(Q)^k\le  C_E\cdot c_k\cdot d(Q)^k\frac{(2\lambda)^k}{ (\lambda/2)^{k-1} }.
        \]
        Choosing $\lambda$ small enough, depending only on $k$, {$C$} and $C_E$ we reach a contradiction, hence $
        {\sup_{x\in H}}d_{\H^n}(x,\W)
        \ge c d(Q)$ for all $W \in \calV^{k-1}.$
\end{proof}

We can now prove the main Heisenberg tilting estimate.
\begin{proof}[Proof of Proposition \ref{p:HeisTilt}]
 Since $\beta_{1,\V}(S) \le \beta_{p,\V}(\cdot)$, it is sufficient to show the case $p=1.$  We will write $\V_1,\V_0$ in place of $\V_{ Q_0},\, \V_{ Q_1}$ for brevity.
    Let $\delta>0$ be a small constant to be chosen depending only on $k$ and the regularity constant of $E$ (that we call $C_E$).
    If $\lambda_0^{k+1}(\beta_{1,\V_{0}}({\lambda_0} Q_0)+\beta_{1,\V_{1}}( {\lambda_1}Q_1))\ge \delta>0$ there is nothing to prove, indeed  $\angle(\V_{0},\V_{1})\le 1$ always holds. Hence we assume
    \begin{equation}\label{eq:beta reasonably small}
    \lambda_0^{k+1}(\beta_{1,\V_{0}}({\lambda_0} Q_0)+\beta_{1,\V_{1}}( {\lambda_1}Q_1))\le \delta.
    \end{equation}
    Let $\{x_0,\dots,x_k\}\subset Q_1$ be a set of $(k+1)$-independent points as given by Proposition \ref{prop:indep points} (with respect to $Q_1$). In particular $B_{c d(Q_1)}(x_i)\cap E\subset Q_1$ for all $i$, $d_{\H^n}(x_i,x_j)\ge c d(Q_1)$, $i \neq j$, and ${\sup_{i=0,\ldots,k}}d_{\H^n} (x_i,\W)\ge cd(Q_1)>0$ for every $\W \in \calV^{k-1}$, where $c{\in (0,1)}$ depends only on $k$ and $C_E.$ Note also that, by the definition of dyadic systems,
    {$\ell(Q_1)\le \ell(Q_0)\le 2 \ell(Q_1)$.}
    Set $B_i\coloneqq B_{cd(Q_1)/4}(x_i)$ and note that $B_i\cap B_j=\emptyset$ for all $i\neq j$ and $B_i\cap E\subset Q_1$ for all $i$, {and thus $B_i\cap E\subset \lambda_1 Q_1 \subset \lambda_0 Q_0$.}
    Using the $k$-regularity of $E$ we have
    \begin{align*}
         \beta_{1,\V_{0}}({\lambda_0} Q_0)+\beta_{1,\V_{1}}( {\lambda_1}Q_1)&\ge {\frac{1}{\mu(\lambda_0Q_0)}}\int_{B_i\cap E} \frac{d_{\H^n}(x,\V_1)+d_{\H^n}(x,\V_0)}{{\diam(\lambda_0 Q_0)}} d \cH^k\\
        &{\gtrsim_{C_E,k}
        \lambda_0^{-k-1}
        }
        \inf_{B_i\cap E}  \frac{d_{\H^n}(x,\V_1)+d_{\H^n}(x,\V_0)}{d(Q_1)}.
    \end{align*}
    Therefore for every $i=0,\dots,k$ there exists a point $p_i \in B_i\cap E$ satisfying
    \begin{equation}\label{eq:double close}
        d_{\H^n}(p_i,\V_1)+d_{\H^n}(p_i,\V_0) \le \tilde C \lambda_0^{k+1} d(Q_1)(\beta_{1,\V_{0}}({\lambda_0} Q_0)+\beta_{1,\V_{1}}( {\lambda_1}Q_1)),
    \end{equation}
    where $\tilde C$ is a constant depending only on $C_E,k$.
     By the triangle inequality, since $p_i \in B_i,$ we also have $d_{\H^n}(p_i,p_j)\ge cd(Q_1) /2 $ for all $j\neq i$ and
    \[
    {\sup_{i=0,\ldots,k}}d_{\H^n} (p_i,\W)\ge 3cd(Q_1)/4>0, \quad \text{for every $\W \in \calV^{k-1}$}.
    \]
    In particular from \eqref{eq:double close} and \eqref{eq:beta reasonably small}, provided we choose $\delta<c/(10\tilde C),$ we can find points $\{y_0,\dots, y_k\}\subset \V_1$  such that $d_{\H^n}(y_i,y_j)\ge cd(Q_1)/3  $, for all $i \neq j$,
    \begin{equation}\label{eq:close to plane}
    \begin{split}
         d(y_i,\V_0)&\le 2\tilde  C\lambda_0^{k+1}  d(Q_1)(\beta_{1,\V_{0}}({\lambda_0} Q_0)+\beta_{1,\V_{1}}( {\lambda_1}Q_1))\\
         &\le  6c^{-1}\tilde  C\lambda_0^{k+1} d_{\H^n}(y_i,y_j) ({\beta_{1,\V_{0}}({\lambda_0} Q_0)+\beta_{1,\V_{1}}( {\lambda_1}Q_1)}), \quad \text{ for all $i=0,...,k$}
    \end{split}
    \end{equation}
    and
    \begin{equation}\label{eq:well indep}
     {\sup_{i=0,\ldots,k}} d_{\H^n} (y_i,\W)\ge cd(Q_1)/4>0, \quad \text{for every $\W \in \calV^{k-1}$}.
    \end{equation}
    From \eqref{eq:close to plane} and the choice of $\delta$ we also have 
    \begin{equation}\label{eq:c estimate for pythagora}
        d(y_i,\V_0)\le \frac{3}{5}d_{\H^n}(y_i,y_j), \quad \text{ for all $i=0,...,k$}.
    \end{equation}
    We can now use these points to estimate the angle between $\V_0,\V_1$. From the Pythagorean-type theorem given by Lemma \ref{lem:pythagora}, using \eqref{eq:c estimate for pythagora} and \eqref{eq:close to plane},
    \begin{align*}
        d_{\H^n} (y_i,y_j)^2\le d_{\H^n} (P_{\V_0}(y_i),P_{\V_0}(y_i))^2 (1-2N\big(12c^{-1}\tilde C \lambda_0^{k+1}(\beta_{1,\V_{0}}({\lambda_0} Q_0)+\beta_{1,\V_{1}}( {\lambda_1}Q_1))\big)^2)^{-1},
    \end{align*}
    where $N$ is the absolute constant given by Lemma \ref{lem:pythagora} and {we have assumed that $\delta$ is so small that} $2N\big(12c^{-1}\tilde C \delta)^2<1/2$.
    We can write the above as
    \[
    |y_i'-y_j'|^2\le |\pi_{V_0'}(y_i')-\pi_{V_0'}(y_j')|^2(1+4N(12c^{-1}\tilde C \lambda_0^{k+1}(\beta_{1,\V_{0}}({\lambda_0} Q_0)+\beta_{1,\V_{1}}( {\lambda_1}Q_1))^2)),
    \]
    where $V_0',V_1'\in \R^{2n}$ are the $k$-dim isotropic subspaces such that $\V_{1}=q_1\cdot (V_1'\times\{0\}), \V_{0}=q_0\cdot (V_0'\times\{0\})$ for some $q_0,q_1\in \H^n$ (see Section \ref{ss:Heis}), $y_i'\in \V_1'$ are such that $y_i=q_1\cdot (y_i',0)$
    and finally $\pi_{V_0'}$ is the Euclidean orthogonal projection onto $V_0'.$

    The key observation is now that from \eqref{eq:well indep} we have
    \begin{equation}\label{eq:independent euclidean}
     {\sup_{i=0,\ldots,k}} \,  d_{\R^{2n}}(y_i',W')\ge cd(Q_1)/4
    \end{equation}
    for every $(k-1)$-dimensional subspace $W'\subset \V_1'$ (where by 0-dimensional subspace we mean simply $W'=\{0_{\mathbb H^n}\}$). Indeed, since $\V_{1}=q_1\cdot ({V_1'\times \{0\}})$, for every such $W'$ it holds that $\W\coloneqq q_1\cdot (W'\times\{0\})\in \calV^{k-1}$ and {$d_{\R^{2n}}(y_i',W')=d_{\H^n}(y_i,\W)$}, since $d_{\H^n}$ is invariant under left translation. Applying Lemma \ref{lem:planes lemma} (provided $\delta$ is small enough) shows that
    \[
    \angle(\V_{1},\V_{0})=\angle_e(V_1',V_0')\le C\lambda_0^{k+1}(\beta_{1,\V_{0}}({\lambda_0} Q_0)+\beta_{1,\V_{1}}( {\lambda_1}Q_1)),
    \]
    where $C$ depends only on $k$ and $C_E$. This concludes the proof.
\end{proof}

\subsection{The challenge with horizontal $\beta$-numbers}\label{s:Juillet}

Traveling salesman theorems  have been studied extensively in the
first Heisenberg group
\cite{MR2371434,MR2789375,MR3456155,MR3512421} using
\emph{horizontal $\beta$-numbers} $\beta_{\infty,\mathcal{V}^1(\mathbb{H}^1)}$,
that is, quantitatively controlled approximation by horizontal
lines. Juillet \cite{MR2789375} gave an example of a rectifiable
curve in $(\mathbb{H}^1,d_{\mathbb{H}^1})$ for which the
$\beta_{\infty,\mathcal{V}^1(\mathbb{H}^1)}$-numbers are not \emph{square}
summable, and in fact not summable with any exponent $p<4$.
Horizontal $\beta$-numbers are however summable with exponent
$p=4$ for every rectifiable curve in $\mathbb{H}^1$, \cite[Theorem
I]{MR3512421}, and if the rectifiable curve is additionally
$1$-regular, then the summability can be upgraded to a geometric
lemma with exponent $p=4$, see \cite[Proposition 3.1]{MR3678492}.
Conversely, summability of the
$\beta_{\infty,\mathcal{V}^1(\mathbb{H}^1)}$-numbers with an exponent $p<4$ for
a set $E\subset \mathbb{H}^1$ is known to be \emph{sufficient} for
the construction of a rectifiable curve containing $E$
\cite{MR3456155}. It is an open question whether one can match the
exponents in the two implications of the traveling salesman
theorem, thus \emph{characterizing} sets contained in a
rectifiable curve of $(\mathbb{H}^1,d_{\mathbb{H}^1})$ in terms of
$4$-summability of the $\beta_{\infty,\mathcal{V}^1(\mathbb{H}^1)}$-numbers.
Here we show that a characterization of uniform $1$-rectifiability
in $\mathbb{H}^n$ for $n>1$ is not possible.

\begin{proposition}\label{p:NonChar} Let $n>1$, $n\in
\mathbb{N}$. Then the following holds:
\begin{enumerate}
\item\label{i:NonChar} for every $1\leq p<4$, there is a
$1$-regular curve $\Gamma$ in $(\mathbb{H}^n,d_{\mathbb{H}^n})$
with $\Gamma \notin
\mathrm{GLem}(\beta_{\infty,\mathcal{V}^1(\mathbb{H}^n)},p)$,
\item\label{ii:NonChar} for every $p>2$, there is a $1$-regular
set $E \in \mathrm{GLem}(\beta_{\infty,\mathcal{V}^1(\mathbb{H}^n)},p)$,
$E\subset \mathbb{H}^n$, that is not contained in a $1$-regular
curve.
\end{enumerate}
\end{proposition}

The curve $\Gamma$ in part \eqref{i:NonChar} will be obtained from
a suitable curve in $\mathbb{H}^1$  by isometrically embedding the
first Heisenberg group into $\mathbb{H}^n$. On the other hand, a
set $E$ verifying part \eqref{ii:NonChar} can be first constructed
in $\mathbb{R}^2$ and then mapped by an isometric embedding of
$\mathbb{R}^2$ into $\mathbb{H}^n$, which exists for $n\geq 2$. To
make this rigorous, we need to deal with the issue that the family
$\mathcal{V}^1(\mathbb{H}^n)$ in $\mathbb{H}^n$ contains more horizontal lines
than those obtained via the isometric embeddings of $\mathbb{H}^1$
or $\mathbb{R}^2$ into $\mathbb{H}^n$. A priori, the sets $\Gamma$
and $E$ could therefore be better approximable by horizontal lines
than their isometric copies in $\mathbb{H}^1$ and $\mathbb{R}^2$,
respectively.

\medskip

We consider the isometric embeddings
\begin{displaymath}
\iota_1:\mathbb{H}^1 \hookrightarrow \mathbb{H}^n,\quad
\iota_1(x,y,t)=(x,0,\ldots,0;y,0,\ldots,0,t),
\end{displaymath}
and
\begin{displaymath}
\iota_2:\mathbb{R}^2 \hookrightarrow \mathbb{H}^n,\quad
\iota_1(x_1,x_2)=(x_1,x_2,0,\ldots,0).
\end{displaymath}
Here $\mathbb{H}^n$, $n\geq 1$, is endowed with the Kor\'{a}nyi
metric, and $\mathbb{R}^2$ with the Euclidean distance. The
following result, Lemma \ref{l:H1betaHn}, relates the relevant
$\beta$-numbers for sets $E\subset \mathbb{H}^1$ and
$\iota_1(E)\subset \mathbb{H}^n$, as well as for sets $E\in
\mathbb{R}^2$ and $\iota_2(E)\subset \mathbb{H}^n$. This is in
spirit of \cite[Lemma 3.2]{https://doi.org/10.1112/jlms.12582},
which states an analogous result for $\mathbb{H}^1\times
\mathbb{R}^2$ instead of $\mathbb{H}^n$. The relevant
$\beta$-numbers are a special instance of the more general
definition given in {Definition
\ref{d:AbstractBeta}}, that is
\begin{displaymath}
\beta_{\infty,\mathcal{V}}(S)= \inf_{\ell\in
\mathcal{V}}\sup_{y\in S}\frac{\sfd(y,\ell)}{\mathrm{diam}(S)}
\end{displaymath}
for $0<\mathrm{diam}(S)<\infty$.{
In this section, we will also denote by
$\mathcal{V}^1(\mathbb{R}^2)$ the family of all affine lines in
$\mathbb{R}^2$.}

\begin{lemma}\label{l:H1betaHn} Assume that $n>1$, $n\in
\mathbb{N}$. Let $A\subset \mathbb{H}^1$ be a set with
$0<\mathrm{diam}(A)<\infty$. Then
\begin{equation}\label{eq:H1betaHn}
\beta_{\infty,\mathcal{V}^1(\mathbb{H}^1)}(A) \sim_n
\beta_{\infty,\mathcal{V}^1(\mathbb{H}^n)}(\iota_1(A)).
\end{equation}
Let $A\subset \mathbb{R}^2$ be a set with
$0<\mathrm{diam}(A)<\infty$. Then
\begin{equation}\label{eq:R21betaHn}
\beta_{\infty,\mathcal{V}^1(\mathbb{R}^2)}(A) \sim
\beta_{\infty,\mathcal{V}^1(\mathbb{H}^n)}(\iota_2(A)) .
\end{equation}
\end{lemma}

\begin{proof} We begin with the first part of the lemma, which
reads \begin{equation}\label{eq:H1betaHn:expl}
 \inf_{\ell \in \mathcal{V}^1(\mathbb{H}^1)}\sup_{a\in
A}\frac{d_{\mathbb{H}^1}(a,\ell)}{\mathrm{diam}_{\mathbb{H}^1}(A)}
\sim_n \inf_{\ell \in \mathcal{V}^1(\mathbb{H}^n)}\sup_{p\in
\iota_1(A)}\frac{d_{\mathbb{H}^n}(p,\ell)}{\mathrm{diam}_{\mathbb{H}^n}(\iota_1(A))}
\end{equation}
The inequality $\gtrsim$ in \eqref{eq:H1betaHn:expl} is clear
since $ \mathcal{V}^1(\mathbb{H}^n) \supset \iota_1\left(
\mathcal{V}^1(\mathbb{H}^1)\right)$ and the restriction of
$d_{\mathbb{H}^n}$ to $\iota_1(\mathbb{H}^1)$ is isometric to
$d_{\mathbb{H}^1}$. We prove the reverse inequality. It is
invariant under Heisenberg dilations, so we make without loss of
generality the assumption that
\begin{equation}\label{eq:diamA}
\mathrm{diam}_{\mathbb{H}^1}(A)=\mathrm{diam}_{\mathbb{H}^n}(\iota_1(A))=1.
\end{equation}
It suffices to consider $\ell \in \mathcal{V}^1(\mathbb{H}^n)$
such that, say,
\begin{equation}\label{eq:beta(l)}
\beta(\ell):= \sup_{p\in \iota_1(A)} d_{\mathbb{H}^n}(p,\ell)<1/8.
\end{equation}
If no such lines exist, the inequality $\lesssim$ in
\eqref{eq:H1betaHn:expl} holds trivially true. For each $\ell \in
\mathcal{V}^1(\mathbb{H}^n)$ as in \eqref{eq:beta(l)}, we will
construct a horizontal line $\bar \ell \in
\mathcal{V}^1(\mathbb{H}^1)$ such that
\begin{equation}\label{goal:betaComp}
d_{\mathbb{H}^1}(a,\bar \ell)\lesssim \beta(\ell),\quad a\in A,
\end{equation}
where the implicit constant is allowed to depend on the
dimensional parameter ``$n$'', but not on $\ell$ or $a$. Since
$\ell$ is a horizontal
 line in $\mathbb{H}^n$, it can be parameterized by
$ \ell(s) = q \cdot (s v,0)$, $s\in \mathbb{R}$,
 for suitable $q\in \mathbb{H}^n$ and $v\in S^{2n-1}$. We will
 show that the  (horizontal) line $\underline{\ell}$ in
$\mathbb{H}^1$ which is parmeterized by
\begin{equation}\label{eq:bar_ell}
\underline{\ell}(s) = (q_1,q_{n+1},q_{2n+1}) \cdot (sv_1,s
v_{n+1},0),\quad s\in \mathbb{R},
\end{equation}
has the desired property \eqref{goal:betaComp}. To see this, for
every $p=\iota_1(a)\in \iota_1(A)$,  we choose that $s_p\in
\mathbb{R}$  such that
 \begin{equation}\label{eq:approx_Hn_param}
 d_{\mathbb{H}^n}(p,\ell(s_p))=d_{\mathbb{H}^n}(p,\ell) \leq \beta(\ell),\quad
 p\in \iota_1(A).
 \end{equation}
 Without loss of generality, we may assume that
 \begin{equation}\label{eq:small_time}
 |s_p|\lesssim 1,\quad p\in \iota_1(A)
 \end{equation}
 because $\iota_1(A)$ has diameter $1$ and, by initially changing $q$ if
 necessary, we may choose the parametrization $\ell$ in such a way
 that $\ell(0)$ lies close to a point in $\iota_1(A)$. By the definition of the Kor\'{a}nyi
 metric and the embedding $\iota_1$, inequality \eqref{eq:approx_Hn_param}
 implies that
 \begin{equation}\label{eq:Eucl_dist_Comp}
 |q_i+v_i s_p|=|p_i- q_i - v_i s_p| \leq \beta(\ell),\quad p\in
 \iota_1(A),\; i \notin \{1,n+1,2n+1\}.
 \end{equation}
 (Recall that the $i$-th coordinates of $p\in \iota_1(A)$ are zero
 for $i \notin \{1,n+1,2n+1\}$.)
 Consider now $p=\iota_1(a),p'=\iota_1(a')\in \iota_1(A)$ with
 \begin{equation}\label{eq:pt_diam}
 d_{\mathbb{H}^n}(p,p')=d_{\mathbb{H}^1}(a,a') \geq \tfrac{1}{2}.
 \end{equation}
The existence of such points is ensured by \eqref{eq:diamA}. We
then find
 \begin{align}\label{eq:time}
 |s_p-s_{p'}|= d_{\mathbb{H}^n}(\ell(s_p),\ell(s_{p'}))
 \overset{\eqref{eq:approx_Hn_param}}{\geq}
 d_{\mathbb{H}^n}(p,p') - 2\beta(\ell)\overset{\eqref{eq:pt_diam}}{\geq}
 \tfrac{1}{2}- \tfrac{1}{4}= \tfrac{1}{4}.
 \end{align}
 On the other hand, \eqref{eq:Eucl_dist_Comp} applied to ``$p$''
 and ``$p'$'' yield
 \begin{equation}\label{eq:v_i}
\tfrac{1}{4}
|v_i|\overset{\eqref{eq:time}}{\leq}|v_i(s_p-s_{p'})|\overset{\eqref{eq:Eucl_dist_Comp}}{\leq}
2\beta(\ell),\quad i \notin
 \{1,n+1,2n+1\}.
 \end{equation}
Combining this information with \eqref{eq:small_time} and
\eqref{eq:Eucl_dist_Comp}, we find that also
\begin{equation}\label{eq:q_i}
|q_i|\lesssim \beta(\ell), \quad  i \notin
 \{1,n+1,2n+1\}.
\end{equation}
Recall the definition of $\bar \ell$ stated in \eqref{eq:bar_ell},
for arbitrary $a=(p_1,p_{n+1},p_{2n+1})\in A$. We have
\begin{align*}
d_{\mathbb{H}^1}(a,\underline{\ell}) &\leq
d_{\mathbb{H}^1}((p_1,p_{n+1},p_{2n+1}),\underline{\ell}(s_p))\\
&= d_{\mathbb{H}^n}(p,\underline{\ell}(s_p))\\
&\lesssim d_{\mathbb{H}^n}(p,\ell(s_p)) + \sum_{i\notin
\{1,n+1,2n+1\}} |q_i + s_p v_i| + \sqrt{|s_p| \sum_{\substack{i,j \notin
\{1,n+1,2n+1\}\\{|i-j|=n}}}|v_i||q_{{j}}|}\\
&\lesssim \beta(\ell),
\end{align*}
where $p\coloneqq\iota_1(a)=(p_1,0,\ldots,0;p_{n+1},0,\ldots,0,p_{2n+1})$
and $\ell$ is as above, and the last inequality follows from
\eqref{eq:approx_Hn_param}, \eqref{eq:small_time},
\eqref{eq:Eucl_dist_Comp}, \eqref{eq:v_i}, and \eqref{eq:q_i}.
This shows \eqref{goal:betaComp} and concludes the proof of the
first part of Lemma  \ref{l:H1betaHn}.

\medskip

Next, we prove the (easier) second part of the lemma, that is,
\begin{equation}\label{eq:R21betaHn:expl}
\inf_{\ell \in \mathcal{V}^1(\mathbb{R}^2)}\sup_{a\in
A}\frac{d_{\mathbb{R}^2}(a,\ell)}{\mathrm{diam}_{\mathbb{R}^2}(A)}
\sim \inf_{\ell \in \mathcal{V}^1(\mathbb{H}^n)}\sup_{p\in
\iota_2(A)}\frac{d_{\mathbb{H}^n}(p,\ell)}{\mathrm{diam}_{\mathbb{H}^n}(\iota_2(A))}.
\end{equation}
for $A\subset \mathbb{R}^2$ with
$0<\mathrm{diam}_{\mathbb{R}^2}(A)<\infty$. Again, the inequality
$\gtrsim$ is clear. The converse inequality follows immediately by
using the ($1$-Lipschitz) projection
\begin{displaymath}
\pi:(\mathbb{H}^n,d_{\mathbb{H}^n}) \to
(\mathbb{R}^2,d_{\mathbb{R}^2}),\quad
\pi(x_1,\ldots,x_{2n},t)=(x_1,x_2).
\end{displaymath}
For every $\ell\in \mathcal{V}^1(\mathbb{H}^n)$, we have $\bar
\ell := \pi(\ell)\in \mathcal{V}^1(\mathbb{R}^2)$ and
\begin{displaymath}
d_{\mathbb{R}^2}(a,\bar \ell)\leq
d_{\mathbb{H}^{n}}(\iota_2(a),\ell),\quad a\in A.
\end{displaymath}
which concludes the proof.
\end{proof}

With Lemma \ref{l:H1betaHn} in place, we proceed to the main
result of this section.

\begin{proof}[Proof of Proposition \ref{p:NonChar}] We start with \eqref{i:NonChar}. The desired curve
$\Gamma$ can be (essentially) obtained by embedding Juillet's
example \cite[Theorem 0.4]{MR2789375} from $\mathbb{H}^1$ into
$\mathbb{H}^n$. To be more precise, Juillet's construction can be
adapted to yield for every $1\leq p<4$ the existence of a
$1$-regular curve $\Gamma_1\subset \mathbb{H}^1$ with the property
that $\Gamma_1\notin
\mathrm{GLem}(\beta_{\infty,\mathcal{V}^1(\mathbb{H}^1)},p)$. This
requires some justification.

First, Juillet's construction is stated for $p=2$, but a similar
construction can be carried out for any exponent $1<p<4$, by
choosing $\theta_n = \frac{C}{n^{2/p}}$ (instead of $\theta_n =
C/n$) on \cite[p.1046]{MR2789375}. This, along with the required
minor changes in the construction, was already discussed in
\cite[Proof of Proposition
3.1]{https://doi.org/10.1112/jlms.12582}.

Second, Juillet's construction for $p=2$ (and the described
modification thereof for arbitrary $p>1$) yields an
$L(p)$-Lipschitz curve $\omega:[0,1]\to \mathbb{H}^1$ that is
obtained as horizontal lift of a (Euclidean) Lipschitz curve
$\omega^{\mathbb{C}}:[0,1]\to \mathbb{R}^2$, which in turn is the
uniform limit of a sequence $(\omega_n^{\mathbb{C}})_{n\in
\mathbb{N}}$ of certain polygonal curves
$\omega^{\mathbb{C}}:[0,1]\to \mathbb{R}^2$. We need to argue that
$\Gamma_1:=\omega([0,1])$ is $1$-regular with respect to the
Kor\'{a}nyi distance. \textcolor{black}{This is proven in detail in Section 4 of \cite{FPZ} with the following main steps: C}omputations similar to the ones in
\cite[Algorithm 5.3 {(Lemma 5.7)}]{2020arXiv200211878B} show that
$\omega^{\mathbb{C}}([0,1])\in \mathrm{Reg}_1(C)$ with $C$ bounded
by a constant depending on $p$\textcolor{black}{; see the proof of \cite[Theorem 4.2]{FPZ}}. Without loss of generality, we may
then assume that the parametrization $\omega^{\mathbb{C}}$
satisfies
\begin{equation}\label{eq:PlanarREg}
\mathcal{H}^1((\omega^{\mathbb{C}})^{-1}(B_r^{\mathbb{R}^2}(z)))\leq
C r,\quad z\in \omega^{\mathbb{C}}([0,1]),\,
0<r<\mathrm{diam}(\omega^{\mathbb{C}}([0,1])),
\end{equation}
{cf., \cite[Lemma 2.3]{MR2337487}}.
Denoting $\pi:\mathbb{H}^1 \to \mathbb{R}^2$, $\pi(z,t)=z$, the
following inclusions hold for $p\in \Gamma_1$ and $r>0$,
\begin{align*}
\omega^{-1}\left(B_r(p)\right) = \{s\colon \omega(s)\in B_r(p)\}&\subset
\{s\colon\pi( \omega(s))\in \pi(B_r(p))\}\\ &\subset \{s\colon
\omega^{\mathbb{C}}(s)\in
B_r^{\mathbb{R}^2}(\pi(p))\}=(\omega^{\mathbb{C}})^{-1}(B_r^{\mathbb{R}^2}(\pi(p))).
\end{align*}
It follows by \eqref{eq:PlanarREg} and the Lipschitz continuity of
$\omega$ that $\Gamma_1=\omega([0,1])$ is upper $1$-regular with
regularity constant depending on $p$ (via the constant $C$ and the
Lipschitz constant $L(p)$). Lower $1$-regularity of $\Gamma_1$ is
automatic since it is a connected set. Hence, $\Gamma_1$ is
$1$-regular and admits dyadic systems.

As a third and final comment, Juillet's (modified) construction in
fact shows that for every dyadic system $\Delta$ on $\Gamma_1$,
there is $Q_0\in \Delta$ such that
\begin{equation}\label{eq:beta_infty}
\sum_{Q\in
\Delta_{Q_0}}\beta_{\infty,\mathcal{V}^1(\mathbb{H}^1)}(2Q)^p\mathcal{H}^1(Q)=\infty.
\end{equation}
This was stated using multiresolution families,
\cite[(0,1)]{MR2789375}, rather than dyadic systems, but the two
formulations are easily seen to be equivalent, recalling also
{\cite[Lemma 2.23]{CarlesonPart1}}
and the comment below {\cite[Corollary
4.6]{CarlesonPart1}}.
 Clearly, if \eqref{eq:beta_infty} holds,
then $\Gamma_1 \notin \mathrm{GLem}
(\beta_{\infty,\mathcal{V}^1(\mathbb{H}^1)},p)$. Having
established this result for $\Gamma_1 \subset \mathbb{H}^1$, it
follows by the first part of Lemma \ref{l:H1betaHn} that
$\Gamma:=\iota_1(\Gamma_1)\subset \mathbb{H}^n$ has the properties
stated in part \eqref{i:NonChar} of Proposition \ref{p:NonChar}
for the given exponent $p<4$.

\medskip

We now prove \eqref{ii:NonChar}. By the second part of Lemma
\ref{l:H1betaHn}, it suffices to find for every $p>2$ a
$1$-regular set $E\subset \mathbb{R}^2$ with $E\in
\mathrm{GLem}(\beta_{\infty,\mathcal{V}^1(\mathbb{R}^2)},p)$ such
that $E$ is not contained in a $1$-regular curve of
$(\mathbb{R}^2,d_{\mathbb{R}^2})$ (or equivalently, $E\notin
\mathrm{GLem}(\beta_{\infty,\mathcal{V}^1(\mathbb{R}^2)},2)$). It
is well-known that sets with these properties exist, but we are
not aware of a reference where this is stated explicitly.
{A possible way of obtaining the set $E$ is to
apply the construction given in \cite[Counterexample
20]{MR1113517} with a sequence $(\alpha_n)_{n\in \mathbb{N}}$ of
angles such that $\sum_{n=1}^{\infty}\alpha_n^2=\infty$ yet
$\sum_{n=1}^{\infty}\alpha_n^p <\infty$ for the given exponent
$p>2$.}
\end{proof}

\appendix



\section{The Euclidean small angle criterion}\label{s:AppendixB}

This appendix contains the proof of the `small angle criterion' stated in Lemma \ref{lem:planes lemma}, which states that the angle between two Euclidean subspaces is small provided
that they are close to each other at sufficiently many
`independent' points. 

\begin{proof}[Proof of Lemma \ref{lem:planes lemma}]
    By scaling it is enough to prove the statement for $r=1$. Moreover, up to a rotation  we can assume that $V_{2}=\{x_{k+1}=\dots=x_{N}=0\}$. Finally up to translating all the points $y_i$, $i=0,\ldots,k$ by $-y_0$, we can assume that $y_0$ is the origin (indeed $|\pi_{V_2}(y_i)-\pi_{V_2}(y_j)|$ is left unchanged by translations of $y_i,y_j$ by the same vector). In particular we can view the points $y_i$, $i=0,\dots,k$ as vectors in $\R^N$ with norm less than one. {As we now have $r=1$ and $y_0=0$, we can also conclude from the assumption that $\sup_{i=0,\ldots,k}\,d_{\mathbb{R}^N}(y_i,W)>cr$ for every $(k-1)$-dimensional affine subspace $W$ of $V_1$ that, in fact,  $\sup_{i=1,\ldots,k}\,d_{\mathbb{R}^N}(y_i,W)>c$ for every $(k-1)$-dimensional subspace $W$ (through the origin) of $\mathrm{span}\{y_1,\ldots,y_k\}$. This observation ensures that CLAIM 2 stated below is applicable in our situation. Note that  for $k=1$ we are simply saying that $|y_1-y_0|=|y_1|>c.$  }

    Observe also that in this configuration $|(x)_{{N-k}}|_{\R^d}=d_{\R^N}(x,V_2)$, where $(x)_{N-k}\in \R^d$ denotes the {last ${d\coloneqq} N-k$} entries of any point $x \in \R^N.$  Note also that  hypothesis \ref{it:indep} ensures that the vectors $y_i$, $i=1,...,k$ are linearly independent.

For the rest of the proof, we denote $\pi=\pi_{V_2}$.    By Pythagoras' theorem, and since $y_0$ is the origin,  we have
    \[
    d_{\R^N}(y_i,V_2)^2+|\pi(y_i)|^2=|y_i|^2\overset{\eqref{eq:close projection}}{\le}  (1+\eps^2)|\pi(y_i)|^2,
    \]
    hence
    \[
    d_{\R^N}(y_i,V_2)\le |\pi(y_i)|\eps\le \eps, \quad  i=1,\dots,k.
    \]

    \medskip
    \noindent CLAIM 1: There exists a constant $D$, depending {only on $c$ and $k$}, such that every $w \in V_1$ can be written as $w=\sum_{{i=1}}^k a_i  y_i$ with {$a_i\in \mathbb{R}$},  $|a_i|\le D |w|$, $i=1,...,k$.
    \medskip

Let us first show how this would allow us to conclude the proof.  Indeed for every $w \in V_1\cap B_1^{\R^N}(0)$,
\[
d_{\R^N}(w,V_2)=|(w)_{{N-k}
}|_{\R^d}\le \sum_{i=1}^k |a_i||(y_i)_{{N-k}}|=\sum_{i=1}^k |a_i|d_{\R^N}(y_i,V_2)\le D\cdot k \eps.
\]

CLAIM 1 in the case $k=1$ is immediate since $|y_1|>c$, as observed above, hence from now on we assume that $k\ge 2.$
To prove CLAIM 1  we first prove the following elementary fact.

    \medskip

    \noindent CLAIM 2: For every $c>0$ there exists $c'=c'({c},k)>0$ such that for all independent vectors $v_1,...,v_k  \in  \R^N$, with $|v_i|< 1$ and satisfying $
{\sup_{i=1,\ldots,k}} \,   d_{\R^N}(v_i,W)>c$
    for every $W$ $(k-1)$-dimensional subspace of ${\rm span}\{v_1,\dots,v_k\}$, it holds that $\det(A A^t)\ge c'$, where $A$ is the matrix having $v_i$ as columns.

    \medskip

    Let us show how to use this to prove CLAIM 1.  Fix an orthonormal frame $\{e_1\}_{i=1}^k$ such that $V_1={\rm span}\{e_1,...,e_k\}$. We can write each $y_i$ with respect to this frame as: $y_i=\sum_{j=1}^k b_j^i e_j$, $b_j^i \in \R$. By a classical linear algebra fact, the volume of the $k$-parallelotope $y_1,...,y_k$ (plus the origin) is equal both to
    $|\det B|$, where $B$ is the matrix having as entries $\{b_j^i\}_{i,j}$ and also to  $\sqrt{|\det (A \cdot A^t)|}$, where $A$ is the matrix having $y_i$ as columns (with their $\R^N$-coordinates).

    Therefore from CLAIM 2 we have that $|\det B|\ge c'=c'(c,k)>0.$  Let now $w\in V_1$ be arbitrary. Then $w=\sum_{j=1}^k t_j e_j$ for some $t_j \in \R$, and also  $w= \sum_{i=1}^k a_i y_i$ for some $a_i\in \mathbb{R}$. Set $\bar t\coloneqq (t_1,...,t_k), \bar a\coloneqq (a_1,...,a_k)$. Standard linear algebra gives that $\bar a=B^{-1}\bar t$. Moreover, since $\{e_j\}_{j=1,\ldots,k}$ is orthonormal, $|\bar t|=|w|$. Therefore, {since $|y_i|\leq 1$, $i=1,\ldots,k$, there exists a constant $c_k$ such that}
    $$|\bar a|\le \|B^{-1}\| |\bar t|   \le \|B\|^{k-1} |\det B|^{-1}|\bar t|\le c_k|\det B|^{-1}|\bar t|\le c_k{c'}^{-1}|w|,  $$
    which proves CLAIM 1 with $D=c_k\, c'^{-1}$.

    It remains to prove CLAIM 2. {Let $v_1,\ldots,v_k$ be as in the assumption of the claim.
    Consider the $k$-simplex $C_k$ determined by the vertices $\{v_0:=0,v_1,\ldots,v_k\}$, and let $C_{k-1}$ be the $(k-1)$-simplex  with vertices $\{v_0:=0,v_1,\ldots,v_{k-1}\}$. Thus $C_{k-1}$ is contained in the $(k-1)$-dimensional subspace $W:=\mathrm{span}\{v_1,\ldots,v_{k-1}\}$ of $\mathrm{span}\{v_1,\ldots,v_k\}$. By assumption, the vertex $v_k$ of $C_k$ is at distance at least $c$ from $W$. It follows that
    \begin{displaymath}
      \frac{1}{k!}|\det(AA^t)|=  \mathrm{vol}_k(C_k)= \frac{d_{\mathbb{R}^N}(v_k,W)}{k}\mathrm{vol}_{k-1}(C_{k-1})\geq \frac{c}{k}\mathrm{vol}_{k-1}(C_{k-1}),
    \end{displaymath}
    where $A$ is the matrix having $v_1,\ldots,v_k$ as columns.
    Thus,  we find that $|\det(AA^t)|\geq c(k-1)! \mathrm{vol}_{k-1}(C_{k-1})$. We proceed iteratively. We observe that our assumptions also guarantee that $$d_{\mathbb{R}^n}(v_{k-1},\mathrm{span}\{v_1,\ldots,v_{k-2}\})\geq c.$$ Indeed, otherwise, we would have $d_{\mathbb{R}^n}(v_{k-1},W')<c$ for $W':=\mathrm{span}\{v_1,\ldots,v_{k-2},v_k\}$, violating the assumptions of the claim. Thus, we can bound $\mathrm{vol}_{k-1}(C_{k-1})$ from below by $(c/(k-1))$ times the volume of the $(k-2)$-simplex with vertices $\{v_0=0,v_1,\ldots,v_{k-2}\}$ and so on. Since the assumptions of CLAIM 2 imply in particular that $d_{\mathbb{R}^N}(v_1,v_2)>c$, we finally conclude that $|\det(A A^t)|$ is bounded from below by a positive number depending on $c$ and $k$ only.}
\end{proof}
\def\cprime{$'$}

\def\cprime{$'$}


\begin{thebibliography}{10}

\bibitem{MR2401600}
Luigi Ambrosio, Nicola Gigli, and Giuseppe Savar\'{e}.
\newblock {\em Gradient flows in metric spaces and in the space of probability
  measures}.
\newblock Lectures in Mathematics ETH Z\"{u}rich. Birkh\"{a}user Verlag, Basel,
  second edition, 2008.

\bibitem{MR1800768}
Luigi Ambrosio and Bernd Kirchheim.
\newblock Rectifiable sets in metric and {B}anach spaces.
\newblock {\em Math. Ann.}, 318(3):527--555, 2000.

\bibitem{MR4277829}
Gioacchino Antonelli and Andrea Merlo.
\newblock Intrinsically {L}ipschitz functions with normal target in {C}arnot
  groups.
\newblock {\em Ann. Fenn. Math.}, 46(1):571--579, 2021.


\bibitem{As14}
Jonas Azzam and Raanan Schul. 
\newblock A quantitative metric differentiation theorem.
\newblock {\em Proc. Amer. Math. Soc.}, 142(4):1351--1357, 2014.

\bibitem{2020arXiv200211878B}
Matthew Badger and Sean McCurdy.
\newblock Subsets of rectifiable curves in {B}anach spaces {I}: {S}harp
  exponents in traveling salesman theorems.
\newblock {\em Illinois J. Math.}, 67(2):203--274, 2023.

\bibitem{MR287283}
J.~A. Baker.
\newblock Isometries in normed spaces.
\newblock {\em Amer. Math. Monthly}, 78:655--658, 1971.

\bibitem{MR2955184}
Zolt\'{a}n~M. Balogh, Katrin F\"{a}ssler, Pertti Mattila, and Jeremy~T. Tyson.
\newblock Projection and slicing theorems in {H}eisenberg groups.
\newblock {\em Adv. Math.}, 231(2):569--604, 2012.

\bibitem{2023arXiv230612933B}
David Bate, Matthew Hyde, and Raanan Schul.
\newblock Uniformly rectifiable metric spaces: Lipschitz images, bi-lateral
  weak geometric lemma and corona decompositions.
\newblock {\em arXiv preprint arXiv:2306.12933}, 2023.

\bibitem{MR0268781}
Leonard~M. Blumenthal.
\newblock {\em Theory and applications of distance geometry.}
\newblock Chelsea Publishing Co., New York,, second edition, 1970.

\bibitem{MR4485846}
Simon Bortz, John Hoffman, Steve Hofmann, Jose~Luis Luna-Garcia, and Kaj
  Nystr\"{o}m.
\newblock Coronizations and big pieces in metric spaces.
\newblock {\em Ann. Inst. Fourier (Grenoble)}, 72(5):2037--2078, 2022.

\bibitem{MR3678492}
Vasileios Chousionis and Sean Li.
\newblock Nonnegative kernels and 1-rectifiability in the {H}eisenberg group.
\newblock {\em Anal. PDE}, 10(6):1407--1428, 2017.

\bibitem{MR1096400}
Michael Christ.
\newblock A {$T(b)$} theorem with remarks on analytic capacity and the {C}auchy
  integral.
\newblock {\em Colloq. Math.}, 60/61(2):601--628, 1990.

\bibitem{MR1113517}
G.~David and S.~Semmes.
\newblock Singular integrals and rectifiable sets in {${\bf R}^n$}: {B}eyondà
  {L}ipschitz graphs.
\newblock {\em Ast\'{e}risque}, (193):152, 1991.

\bibitem{MR1009120}
Guy David.
\newblock Morceaux de graphes lipschitziens et int\'{e}grales singuli\`eres sur
  une surface.
\newblock {\em Rev. Mat. Iberoamericana}, 4(1):73--114, 1988.

\bibitem{MR1123480}
Guy David.
\newblock {\em Wavelets and singular integrals on curves and surfaces}, volume
  1465 of {\em Lecture Notes in Mathematics}.
\newblock Springer-Verlag, Berlin, 1991.

\bibitem{MR1251061}
Guy David and Stephen Semmes.
\newblock {\em Analysis of and on uniformly rectifiable sets}, volume~38 of
  {\em Mathematical Surveys and Monographs}.
\newblock American Mathematical Society, Providence, RI, 1993.

\bibitem{MR1616732}
Guy David and Stephen Semmes.
\newblock {\em Fractured fractals and broken dreams}, volume~7 of {\em Oxford
  Lecture Series in Mathematics and its Applications}.
\newblock The Clarendon Press, Oxford University Press, New York, 1997.
\newblock Self-similar geometry through metric and measure.


\bibitem{MR1731465}
Guy David and Tatiana Toro.
\newblock Reifenberg flat metric spaces, snowballs, and embeddings.
\newblock {\em Math. Ann.}, 315(4):641--710, 1999.

\bibitem{MR2907827}
Guy David and Tatiana Toro.
\newblock Reifenberg parameterizations for sets with holes.
\newblock {\em Mem. Amer. Math. Soc.}, 215(1012):vi+102, 2012.
		

\bibitem{CarlesonPart1}
Katrin F{\"a}ssler and Ivan~Yuri Violo.
\newblock On various {C}arleson-type geometric lemmas and uniform
              rectifiability in metric spaces: {P}art 1.
\newblock {\em Rev. Mat. Iberoam.}, 41(6):2003--2054, 2025.

		
\bibitem{FPZ}
Katrin F{\"a}ssler, Andrea Pinamonti, and Kilian Zambanini.
\newblock On low-dimensional uniform rectifiability in Heisenberg groups.
\newblock {\em arXiv preprint arXiv:2601.03837}, 2026.


\bibitem{MR2371434}
Fausto Ferrari, Bruno Franchi, and Herv\'{e} Pajot.
\newblock The geometric traveling salesman problem in the {H}eisenberg group.
\newblock {\em Rev. Mat. Iberoam.}, 23(2):437--480, 2007.

\bibitem{hahheisenberg}
Immo Hahlomaa.
\newblock A sufficient condition for having big pieces of bilipschitz images of
  subsets of euclidean space in {H}eisenberg groups.
\newblock arXiv:1212.0687, 2012.

\bibitem{MR2789375}
Nicolas Juillet.
\newblock A counterexample for the geometric traveling salesman problem in the
  {H}eisenberg group.
\newblock {\em Rev. Mat. Iberoam.}, 26(3):1035--1056, 2010.


\bibitem{Kir94}
Bernd Kirchheim. 
\newblock Rectifiable metric spaces: local structure and regularity of
              the {H}ausdorff measure.
\newblock {\em Proc. Amer. Math. Soc.}, 121(1):113--123, 1994.


\bibitem{LW09}
Gilad Lerman and J.~Tyler Whitehouse.
\newblock High-dimensional menger-type curvatures---part II: $d$-separation and a menagerie of curvatures.
\newblock {\em Constr. Approx.}, 30(3):325--360, 2009.

\bibitem{LW11}
Gilad Lerman and J.~Tyler Whitehouse.
\newblock High-dimensional Menger-type curvatures. Part I: Geometric multipoles and multiscale inequalities.
\newblock {\em Rev. Mat. Iberoam.}, 27(2):493--555, 2011.



\bibitem{https://doi.org/10.1112/jlms.12582}
Sean Li.
\newblock Stratified {$\beta$}-numbers and traveling salesman in {C}arnot
  groups.
\newblock {\em J. Lond. Math. Soc. (2)}, 106(2):662--703, 2022.

\bibitem{MR3456155}
Sean Li and Raanan Schul.
\newblock The traveling salesman problem in the {H}eisenberg group: upper
  bounding curvature.
\newblock {\em Trans. Amer. Math. Soc.}, 368(7):4585--4620, 2016.

\bibitem{MR3512421}
Sean Li and Raanan Schul.
\newblock An upper bound for the length of a traveling salesman path in the
  {H}eisenberg group.
\newblock {\em Rev. Mat. Iberoam.}, 32(2):391--417, 2016.

\bibitem{MR2789472}
Pertti Mattila, Raul Serapioni, and Francesco Serra~Cassano.
\newblock Characterizations of intrinsic rectifiability in {H}eisenberg groups.
\newblock {\em Ann. Sc. Norm. Super. Pisa Cl. Sci. (5)}, 9(4):687--723, 2010.

\bibitem{MR2337487}
Raanan Schul.
\newblock Ahlfors-regular curves in metric spaces.
\newblock {\em Ann. Acad. Sci. Fenn. Math.}, 32(2):437--460, 2007.

\bibitem{MR993774}
Nicole Tomczak-Jaegermann.
\newblock {\em Banach-{M}azur distances and finite-dimensional operator
  ideals}, volume~38 of {\em Pitman Monographs and Surveys in Pure and Applied
  Mathematics}.
\newblock Longman Scientific \& Technical, Harlow; copublished in the United
  States with John Wiley \& Sons, Inc., New York, 1989.

\bibitem{MR4489627}
Ivan~Yuri Violo.
\newblock A remark on two notions of flatness for sets in the {E}uclidean
  space.
\newblock {\em J. Reine Angew. Math.}, 791:157--171, 2022.

\bibitem{MR3491489}
Yu~Zhou, Zihou Zhang, and Chunyan Liu.
\newblock On isometric representation subsets of {B}anach spaces.
\newblock {\em Bull. Aust. Math. Soc.}, 93(3):486--496, 2016.

\end{thebibliography}
\end{document}